\documentclass[12pt]{article}

\usepackage{amsmath,makeidx,amssymb,amscd,amsfonts, amsthm}
\usepackage{graphicx,epsfig,latexsym} 
\usepackage{color,epsfig}

\usepackage[bookmarksopen, colorlinks,
  linkcolor =blue, urlcolor =red, citecolor =red, menucolor =blue]{hyperref}

\newcommand\A{\mathbb{A}}
\newcommand\C{\mathbb{C}}
\newcommand\R{\mathbb{R}}

\newcommand\norm[2][]{{\left\lVert#2\right\rVert_{#1}}}

\newtheorem{theo}{Theorem}
\newtheorem{lemma}[theo]{Lemma}

\newtheorem{pr}[theo]{Proposition}
\newtheorem{df}{Definition}
\newtheorem{remark}{Remark}

\newcommand{\inner}[2]{\langle #1,#2\rangle}

\def\ve{{\varepsilon}}
\def\eps{{\varepsilon}}

\def\Ga{{\cal G}_\alpha}

\def\Ger{{\cal G}_{\eps,\alpha}}

\def\.{{\;}}
\def\bv{{{\tt{BV}}(\Omega)}}

\def\phia{{\phi^a}}  
\def\phic{{\phi^c}}

\newcommand{\phiak}[1]{\phi^{a}_{{#1}}} 
\newcommand{\phick}[1]{\phi^{c}_{{#1}}} 

\begin{document}
\setcounter{footnote}{1}

\title{On the identification of piecewise constant coefficients in optical diffusion tomography by level set}

\author{J. P. Agnelli
\thanks{ FaMAF-CIEM, Universidad Nacional de C\'ordoba, Medina Allende s/n 5000, C\'ordoba, Argentina.
(\href{mailto:agnelli@famaf.unc.edu.ar}{\tt agnelli@famaf.unc.edu.ar}).}
\and
A.~De Cezaro
\thanks{Institute of Mathematics Statistics and Physics,
Federal University of Rio Grande, Av. Italia km 8, 96201-900 Rio
Grande, Brazil (\href{mailto:decezaromtm@gmail.com}{\tt
decezaromtm@gmail.com}).}
\and
A.~Leit\~ao
\thanks{Department of Mathematics, Federal University of St. Catarina,
P.O. Box 476, 88040-900 Florian\'opolis, Brazil
(\href{mailto:aleitao@mtm.ufsc.br}{\tt aleitao@mtm.ufsc.br}).}
\and
M. Marques Alves
\thanks{Department of Mathematics, Federal University of St.
Catarina, 88040-900 Florian\'opolis, Brazil
(\href{mailto:maicon.alves@ufsc.br}{\tt maicon.alves@ufsc.br}).
}
}
\date{\small\today}

\maketitle

\begin{abstract}
In this paper, we propose a level set regularization approach combined with a
split strategy for the simultaneous identification of piecewise constant diffusion
and absorption coefficients from a finite set of optical tomography data
(Neumann-to-Dirichlet data). This problem is a high nonlinear inverse problem combining
together the
exponential and mildly ill-posedness of  diffusion and absorption coefficients,
respectively.  We prove that the parameter-to-measurement map satisfies sufficient
conditions (continuity in the $L^1$ topology) to guarantee regularization properties
of the proposed level set approach. On the other hand, numerical
tests considering different configurations bring new ideas on how to
propose a convergent split strategy for the simultaneous identification of the
coefficients. 
The behavior and performance of the proposed numerical strategy is illustrated with some numerical examples.
\end{abstract}

\pagestyle{plain}

\section{Introduction}
\label{sec:1}

Optical tomography has demonstrated to be a powerful technique to obtain relevant
physiological information of tissues in a non-invasive manner. The technique relies
on the object under study being at least partially light-transmitting or translucent, so
it works best on soft tissues such as breast and brain tissue
\cite{Ntz99,Hebden97}. By monitoring spatial-temporal variations in the light
absorption and scattering properties of tissue, regional variations in hemoglobin concentration
or blood oxygen saturation can be calculated. 
For a complete overview on optical tomography modalities the reader can consult the
topical reviews \cite{Arridge99, GHA2005} and references therein.

A full description of light propagation in tissue is provided by the radiative transport
equation. However, in this contribution we are interested in the so called static
diffuse optical tomography (DOT). In DOT, light in the near infrared spectral region is
used to measure the optical properties of physiological tissue. In this case, denoting
the photon density by $u$, the equation to consider is the following:
\begin{align}\label{eq:1}
- \nabla \cdot (a(x) \nabla u) + c(x)u & = 0 \quad \mbox{in}
\,\, \Omega   \\
a(x) \frac{\partial u}{\partial \nu}  & = \label{eq:2}
                g \quad \mbox{on}\,\, \Gamma\,,
\end{align}
where $\Omega\subset \R^N$, $N\in \{2,3,4\}$, is open, bounded and
connected with Lipschitz boundary denoted by $\Gamma$, the diffusion and absorption
coefficients $a(x)$ and $c(x)$, respectively, are measurable real-valued functions
and $\nu$ is the outer-pointing normal vector. Moreover, $g\in
H^{-1/2}(\Gamma)$ is the Neumann boundary data. Such boundary condition can be interpreted as the exitance on $\Gamma$.

It is worth mentioning that equation~\eqref{eq:2} is a simplified way of modeling light fluence
boundary condition in diffuse optical tomography, since a more realistic
description is to consider a Robin boundary condition \cite{SAHD1995,
SchArrige99}. However, we believe that our simplified boundary model already
contains the essential aspects for the theoretical study that we present in this
work. 
Further, this setting agrees with the uniqueness identification result derived by Harrach
in \cite{Harrach09}. For more details about boundary conditions in light propagation
models we recommend to consult \cite{SAHD1995, SchArrige99} and references therein. 

Since the optical properties within tissue are determined by the values of the 
diffusion and absorption coefficients, the problem of interest
in DOT is the simultaneous identification of both coefficients from measurements of
near-infrared diffusive light along the tissue boundary.

In this contribution, we proposed a level set regularization approach~\cite{FSL05, CL2010,
CLT09A, CLT09} combined with a split strategy for the simultaneous identification of
piecewise constant diffusion $a(x)$ and absorption $c(x)$ coefficients in
\eqref{eq:1}--\eqref{eq:2}, from 
a finite set of available measurements of the photon density $h: = u|_{\Gamma}$, corresponding to
inputs $g \in H^{-1/2}(\Gamma)$ in \eqref{eq:1}--\eqref{eq:2}.

\medskip
{\bf Related works:}
In~\cite{KALVK99}, a Levenberg--Marquardt method
for recovering internal boundaries of piecewise
constant coefficients of an elliptic PDE as \eqref{eq:1} was
implemented. The proposed method is based on the series expansion
approximation of the smooth boundaries and on the finite
element method. However, in \cite{KALVK99}, there is not a
theoretical result that guarantees regularizing properties of the
iterated approximated solution. Indeed, as far as the authors are
aware, there is not theoretical regularization approaches in the literature for
recovering the pair of coefficients $(a,c)$ in \eqref{eq:1} 
from boundary data.

In~\cite{Xu02}, the authors did a carefully designed experiment aimed
to provide solid evidence that both absorption and scattering images
of a heterogeneous scattering media can be reconstructed independently
from diffuse optical tomography data. The 
authors also discuss
the absorption scattering cross-talk issue.

Although it is well known that the identification 
of $a$ and $c$
simultaneously is not possible 
in a general case~\cite{ArridgeLio98},
recently B. Harrach~\cite{Harrach09} obtained 
a uniqueness result for the simultaneous
recovery of $a$ and $c$ in \eqref{eq:1}--\eqref{eq:2} 
assuming
that $a \geq a_0 >0$ is piecewise constant and $c\in L^\infty_+$
\footnote{The subscript $'+'$ denotes positive essential infima.} is piecewise
analytic. Under this condition both parameters are simultaneously uniquely determined by knowledge of all possible pairs of Neumann and
Dirichlet boundary values  $a\partial_{\nu}u|_{\tilde{\Gamma}}$ and
$u|_{\tilde{\Gamma}}$ on an arbitrarily small open set
$\tilde{\Gamma}$ of the boundary $\Gamma$. The difference between the work of B. Harrach~\cite{Harrach09}
and our work is that here we are considering a more practical approach:
we only have access to a finite number of Neumann-Dirichlet pairs.

We also remark that the  quantitative
photoacoustic tomography problem (QPAT), in the diffusive approach,
also aims to simultaneous recover $(a,c)$ of an elliptic boundary valued problem.
See for example \cite{TCKA12} and references therein. However,
in the QPAT situation the solution of the ``first inverse problem"
generate internal data for the reconstruction. In this sense, the
QPAT problem is very different to the identification problem that we
are facing here.

\medskip
{\bf Novelties:} The novelties of this contribution are divided as follow:

\begin{itemize}
\item We prove the 
continuity of the parameter-to-measurement (forward)
map $F$ (defined in \eqref{eq:def-F}) 
in the $[L^1(\Omega)]^2$ topology.
It is done in Theorem~\ref{th:continuity1} 
of Section~\ref{sec:Properties-F}, and 
it is possible thanks to a 
generalization of Meyers' Theorem \cite{M63} that prove
the regularity of the solution of 
\eqref{eq:1}--\eqref{eq:2} in
$W^{1,p}(\Omega)$ for some $p> 2$ 
(see Theorem~\ref{th:Meyers}).
The proof of Theorem \ref{th:Meyers}
is presented in detail in the Appendix. 

In Section~\ref{sec:Modeling}, we introduce a level set approach. In
contrast to the seminal approach of Santosa on level
set for inverse problems~\cite{San95}, our approach consists in a parametrization
of the non-smooth admissible set of parameters with a pair of
$H^1(\Omega)$ functions concatenated with a restriction of the
search space using nonlinear constraints. Such approach allows us to
enforce the desired additional properties on the pair of parameters
$(a,c)$ (namely: $(a,c)$ is a pair of piecewise constant functions
describing the high diffusion and absorption contrast between the
optical properties of the object) that are not smooth.

Given the continuity of $F$ in the $[L^1(\Omega)]^2$ topology, it is now a
standard result to prove that the level set approach is a regularization method
as in the classical theory of regularization~\cite{CLT09A, CLT09, CL2010, CL2010_1}. Therefore, we only point out the convergence
and stability results without a formal proof in Subsection~\ref{sec:regularization}.

\item Another contribution of the proposed level set approach is related to the
numerical implementation presented in Section~\ref{sec:numerical}. It is worth
to remind the reader that we aim to simultaneous reconstruct the
pair $(a,c)$ of piecewise constant functions from a finite set of
optical measurements. With this aim, we first run several numerical
experiments in order to recover the absorption coefficient $c$,
based on either total or partial knowledge of $a$. From this
experiments we observed that the level set method for identifying $c$
performs well, even if a good approximation of the exact value of $a$ is not known.
This is presented in Subsection~\ref{ssec:numer-c}. 
After that, in
Subsection~\ref{ssec:numer-a} we run another 
set of experiments but
now concerning the identification 
of the diffusion coefficient $a$,
based on either total or partial knowledge of $c$. 
In this case, we
observed that the level set method for identifying 
$a$ performs well if
a good approximation of the exact 
value of $c$ is available, but may generate a
sequence $a_k$ that does not approximate the exact 
$a$ if the initial guess
for the coefficient $c$ is far from its exact value. Such features of the
identification problems suggested one of the main results related to the
numerical perspective presented in Subsection~\ref{ssec:numer-ac}.
Given a initial guess $(a_0, c_0)$, we adopted the strategy to
``freeze'' the coefficient $a_k = a_0$ during the first iterations, and
to iterate the algorithm only with respect to the coefficient $c$. We
follow this strategy until the iterated sequence $c_k$ stagnates.
Then, we freeze the absorption coefficient $c = c_k$ and iterate
the algorithm only with respect to $a$ until the iterated sequence $a_k$ stagnates. Finally, we iterate
both coefficients simultaneously. This numerical strategy has not only demonstrated that gives very good results but also reduces significantly the computational effort.
\end{itemize}

This article is organized as follows. In Section~\ref{sec:Properties-F} we first introduce the parameter-to-measurement (forward)
map F and after that, in Theorem~\ref{th:continuity1}, the continuity of this forward map is demonstrated. Then, in Section~\ref{sec:Modeling}
we present the level set approach, we introduce the concept of generalized minimizers for an appropriate energy functional and we establish the regularization properties.
In other words, we prove the well-posedness result and also convergence results for
exact and noisy data. In Section~\ref{sec:tow-num-sol}, we introduce a smooth functional that is used in the numerical examples. We prove
that the minimizers of such functional converge to a minimizer of the early energy functional in appropriated topologies. 
Section~\ref{sec:numerical} is devoted to numerical experiments and a split strategy is developed. We end this contribution
in Section~\ref{sec:Conclusions} with some conclusions and further developments. In the Appendix, we give a proof for
a generalization of a Meyers' type theorem (see Theorem~\ref{th:Meyers}) about the regularity of the solution of \eqref{eq:1}--\eqref{eq:2}
that are used in the proof of the continuity of the forward map $F$.

\vspace{1cm}
{\bf General Notation.} We denote by $\R^N$, $N\geq 2$, the $N$-dimensional Euclidean space
endowed with the usual scalar product $x\cdot y=\sum_{i=1}^N\,x_iy_i$
and norm $|x|=\sqrt{x\cdot x}$, where $x=(x_i)_{i=1}^N$
and $y=(y_i)_{i=1}^N$. Given two normed vector spaces $(\mathcal{X},\|\cdot\|_{\mathcal{X}})$ and $(\mathcal{Y},\|\cdot\|_{\mathcal{Y}})$
we always consider the product space $\mathcal{X}\times \mathcal{Y}$ endowed with
the product topology generated by the norm $\|(x,y)\|:=\|x\|_{\mathcal{X}} + \|y\|_{\mathcal{Y}}$ (or the equivalent norms
$\left(\|x\|_{\mathcal{X}}^2 + \|y\|_{\mathcal{Y}}^2\right)^{1/2}$ or $\max\{\|x\|_{\mathcal{X}},\|y\|_{\mathcal{Y}} \}$), where $(x,y)\in \mathcal{X}\times \mathcal{Y}$.
We also use the short notation $\mathcal{X}^2=\mathcal{X}\times \mathcal{X}$.

\section{The Parameter to Measurement Map} \label{sec:Properties-F}

We start this section assuming that coefficient $a(x)$ is known for all $ x\in \Gamma$.  Then, for each
input $g \in H^{-1/2}(\Gamma)$ in \eqref{eq:1}--\eqref{eq:2}, we define the
\emph{parameter-to-measurement} (forward) map
\begin{eqnarray}
\label{eq:def-F}
F := F_g\,: D(F) \subset L^1(\Omega)\times L^1(\Omega)  &    \to    &  H^{1/2}(\Gamma)\\
     (a,c)   &  \mapsto  &  h := u|_{\Gamma},\nonumber
\end{eqnarray}
where $u = u(g)$ is the unique 
solution of
\eqref{eq:1}--\eqref{eq:2} given the boundary data $g$ and the pair $(a,c)$ in the
parameter space $D(F)$ defined as:
\begin{df}
\label{df:admissible} 
Denote by $D(F)$ the set of pairs of $L^1(\Omega)$
functions $(a,c)$ on $\Omega$ satisfying the 
following condition:
%


\begin{align}
\label{eq:def.df7}
 0<\underline{a} \leq a(x) \leq \overline{a},
\quad 0<\underline{c}
\leq c(x) \leq \overline{c} \quad \forall x\;\mbox{a.e. in\,}\,
\Omega,
\end{align}
where $\underline{a},\overline{a}$, \underline{c}
and $\overline{c}$ are known positive real numbers.

\end{df}
\noindent We now make some comments about
Definition~\ref{df:admissible} and the definition of the forward map $F$.
First, it is easy to check that $D(F)$ is a convex subset of $[L^1(\Omega)]^2$.
Second, the forward map $F$ is well-defined because for each $(a,c)\in
D(F)$ there exists a unique solution $u \in H^1(\Omega)$ of~\eqref{eq:1}--\eqref{eq:2}
(see~\cite{DL88}).
Third, since $D(F)$ depends on the scalars
$\underline{a},\overline{a}$, \underline{c} and $\overline{c}$ ,
it turns out that $F$ also depends on the latter scalars. However, we
are assuming that the scalars are known, fixed and independent of each given Neumann data $g$ in \eqref{eq:2}.
Fourth, we are not assuming any smoothness condition on the pair
$(a,c)\in D(F)$. In particular the latter fact allow us
to consider solutions of~\eqref{eq:1}--\eqref{eq:2} corresponding
to piecewise constant coefficients.

This section is devoted to prove the continuity of the forward
map $F$ in the $[L^1(\Omega)]^2$ topology. In order to make
such proof  easier to understand we will consider the \emph{parameter-to-solution}
map 
\begin{align}
\label{operator:G}
 G:=G_g\,:\,  D(F) \subset [L^1(\Omega)]^2 & \longrightarrow
 H^1(\Omega)\nonumber \\
 (a,c) & \longmapsto G_g(a,c):=u\,,
\end{align}
where $u = u(g)\in H^1(\Omega)$ is the unique solution
of~\eqref{eq:1}--\eqref{eq:2} for each input data $g\in H^{-1/2}(\Gamma)$ and parameters
$(a,c)\in D(F)$. Moreover, we will use the fact that any solution of~\eqref{eq:1}--\eqref{eq:2} 
satisfies the following weak formulation \cite{DL88}:

\begin{align}
\label{eq:weakf}
 \int_{\Omega}\,a\nabla u\cdot \nabla \varphi \,dx + \int_{\Omega}\,cu\varphi \,dx
= \int_{\Gamma}g\varphi \, d\sigma\qquad \forall \varphi\in H^1(\Omega).
\end{align}

\begin{remark}\label{rm:1}
Given the definition of the forward map $F$ in~\eqref{eq:def-F}, and using the map defined in \eqref{operator:G}, we have that $F$
can be written as
\begin{align}
\label{eq:fbg}
 F=\gamma_0 \circ G,
\end{align}
where $\gamma_0:H^1(\Omega)\to H^{1/2}(\Gamma)$ is the trace operator of order zero \cite{DL88}.
Since the operator $\gamma_0$ is linear and continuous~\cite{DL88}, the continuity of $F$
follows from the continuity of $G$.
\end{remark}



In order to prove the 
continuity of the operator $G$ defined in \eqref{operator:G}
in the desired topology, 
we will use the following 
generalization of Meyers' Theorem \cite{GM99}
on the regularity of the solution 
of \eqref{eq:1}--\eqref{eq:2}. The proof
of Theorem \ref{th:Meyers}
is presented in the Appendix.

\begin{theo}[Generalized Meyers' Theorem]
\label{th:Meyers}
Let $\Omega\subset \R^N$, $N\in \{2,3,4\}$, be a connected bounded open set with a Lipschitz boundary $\Gamma$
and let $ (a,c) \in D(F)$. Then, there exists 
a real number $p_M>2$ (depending only on 
$\Omega$, $\underline{a}, \overline{a}, \underline{c}$ and $\overline{c}$) such that the
following condition hold for every $p\in (2,p_M)$:
If $g\in W^{1-(1/q),q}(\Gamma)'$, where $q := p/(p-1)$, then the unique solution $u$
of~\eqref{eq:1}--\eqref{eq:2} belongs to $W^{1,p}(\Omega)$.
\end{theo}
Next, we present a lemma that is used in the main theorem of this section.
\begin{lemma}
\label{lemma:0}
Let $h$ be a measurable function such that $|h(x)| \leq M$ for all
$x\;\mbox{a.e. in\,}\,\Omega$,
for some constant $M>0$. Then, $h\in L^s(\Omega)$ for all $1\leq s < \infty$
and
\begin{align*}
\norm{h}_{L^s(\Omega)} \leq M^{(s-1)/s} \norm{h}^{1/s}_{L^1(\Omega)}.
\end{align*}
\end{lemma}
\begin{proof}
Note that
\begin{align*}
\norm{h}^s_{L^s(\Omega)} = \int_{\Omega}|h(x)||h(x)|^{s-1} \, dx \leq M^{s-1}
\norm{h}_{L^1(\Omega)},
\end{align*}
which readily implies the desired result.
\end{proof}


In the next theorem we prove the continuity of the \emph{parameter-to-solution}
map $G$ in the $[L^1(\Omega)]^2$ topology. Then, by Remark~\ref{rm:1} we obtain the desired continuity of the \emph{parameter-to-measurement} map $F$.

\begin{theo}\label{th:continuity1}
Let $p\in (2,p_M)$, where $p_M>2$ is given by Theorem~\ref{th:Meyers}, and let $q:=p/(p-1)$.
Then, for any $g\in W^{1-(1/q),q}(\Gamma)'$  the operator
$G$ defined in~\eqref{operator:G} is continuous in the $[L^1(\Omega)]^2$ topology.

As a consequence, for any $g\in W^{1-(1/q),q}(\Gamma)'$, the
forward map $F$ defined in~\eqref{eq:def-F} is also continuous
in the $[L^1(\Omega)]^2$ topology.
\end{theo}
\begin{proof} 
Let $g\in W^{1-(1/q),q}(\Gamma)'$ and consider
the corresponding solutions $u'=u(a', c')$ and  $u=u(a, c)$ of~\eqref{eq:1}--\eqref{eq:2}
with parameters $(a',c'), (a, c) \in D(F)$, respectively.

Since $(a',c',u')$ and $(a,c,u)$ satisfy the identity~\eqref{eq:weakf} for all $\varphi\in H^1(\Omega)$ we have
\begin{align}
\label{eq:273}
\int_\Omega (a \nabla u-a'\nabla u')\cdot\nabla \varphi \,dx +
\int_\Omega(c u-c'u')\varphi \,dx = 0\,.
\end{align}
Defining $w:=u-u'\in H^1(\Omega)$ and using~\eqref{eq:273} with $\varphi=w$
we obtain (after some algebraic manipulations)
\begin{align*}
\int_\Omega \big(a - a'\big) \nabla u'\cdot \nabla w \,dx + \int_\Omega a \nabla
w \cdot \nabla w \,dx + \int_\Omega \big(c - c'\big) u'w \,dx + \int_\Omega c w w \,dx = 0,
\end{align*}
which in turn is equivalent to
\begin{align}
\label{eq:274}
\int_\Omega a(x) |\nabla w|^2 \,dx + \int_\Omega c(x) |w|^2 \,dx = \int_\Omega
\big(a' - a\big) \nabla u'\cdot \nabla w \,dx +  \int_\Omega \big(c'- c\big) u' w \,dx.
\end{align}
In view of Theorem~\ref{th:Meyers} 
(for $(a',c')$) we have 
$u'\in W^{1,p}(\Omega)$. Thus, defining $s:=2p/(p-2)$,
it follows from~\eqref{eq:274},~\eqref{eq:def.df7}, 
Lemma~\ref{lemma:0}
and the H\"{o}lder inequality (note that $1/s+1/p+1/2=1$)
that
%
\begin{align*}
\min\{\underline{a},\underline{c}\}\|w\|^2_{H^1}&\leq \|a'-a\|_{L^{s}}
\norm{\nabla u'}_{L^p}\norm{\nabla w}_{L^2}
+\|c'-c\|_{L^{s}}\norm{u'}_{L^p}\norm{w}_{L^2}\\
&\leq \left(\|a'-a\|_{L^{s}}\norm{\nabla u'}_{L^p}+\|c'-c\|_{L^{s}}\norm{u'}_{L^p}\right)\norm{w}_{H^1}\\
&\leq 2\left(\max\{\overline{a} - \underline{a},\overline{c}-\underline{c}\}\right)^{(s-1)/s}\norm{u'}_{W^{1,p}}\left(\norm{a'-a}_{L^1}+\norm{c'-c}_{L^1}\right)^{1/s}\norm{w}_{H^1}.
\end{align*}

The latter inequality combined with 
the facts that  $G_g(a',c')=u'$,
$G_g(a,c)=u$ (see~\eqref{operator:G}) and $w=u-u'$ give
\begin{align}
\label{eq:cont.g}
 \norm{G_g(a,c)-G_g(a',c')}_{H^1}
\leq \widetilde{M}
\norm{u'}_{W^{1,p}}
\left(\norm{a-a'}_{L^1}+\norm{c-c'}_{L^1}\right)^{1/s},
\end{align}
which proves the continuity of $G_g$ in the $[L^1(\Omega)]^2$ topology, where
$\widetilde{M} := \dfrac{2\left(\max\{\overline{a} - \underline{a},\overline{c}-\underline{c}\}\right)^{(s-1)/s}}
{\min\{\underline{a},\underline{c}\}}$.

The last statement of the theorem now follows easily from the first one and Remark~\ref{rm:1}.
\end{proof}
\noindent
We now make a few comments about Theorem~\ref{th:continuity1}. First, according to Theorem~\ref{th:Meyers}, the real
number $p_M>2$ depends on $\Omega$ and, 
in the present setting, on $\underline{a}$, 
$\overline{a}$, $\underline{c}$ and $\overline{c}$. 
Second, since $q<2$, it follows that $W^{1-(1/q),q}(\Gamma)'\subset H^{-1/2}(\Gamma)$. As a consequence of the latter
inclusion, we have that the condition on $g$ required in Theorem~\ref{th:continuity1} is stronger than the usual
inclusion $g\in H^{-1/2}(\Gamma)$. Third, condition~\eqref{eq:cont.g} gives that both operators $G$ and $F$
are (locally) H\"{o}lder continuous in the $[L^1(\Omega)]^2$ topology.

\section{The level set framework with a finite number of experiments}\label{sec:Modeling}

It is already known that in diffuse optical tomography the full Neumann-to-Dirichlet map
(equivalently, the boundary data $h$ corresponding to
the boundary condition  $a \frac{\partial u}{\partial \nu}$ on $\Gamma$) is required to obtain
uniqueness of the parameters $(a,c)$ in \eqref{eq:1}--\eqref{eq:2}~\cite{Harrach09}.
However, in real applications, only a finite number of observations/measurements are available.
Therefore, in this work we consider that 
we only have access to a  
quantity $\ell \in \mathbb{N}$ of well-placed experiments. In other words, the inverse
problem we tackle consist in given a finite number of inputs $g_m = a\frac{\partial u_m}{\partial
\nu}{|_\Gamma}$ and corresponding data $h_m=u_{m}{|_\Gamma}$, reconstruct simultaneously the diffusion and absorption coefficients
$(a,c)$. As indicated previously, the photon density $u_m$ satisfies 
\begin{align*}
\nabla \cdot (a \nabla u_m) + c \, u_m & = 0 \qquad \mbox{in } \Omega, \\
a \frac{\partial u_m}{\partial \nu} & = g_m\, \quad \mbox{on }
\Gamma, \qquad  m=1, \ldots, \ell.\nonumber
\end{align*}
This problem is known in the literature as the inverse problem for
the Neumann-to-Dirichlet operator with a finite number of
experiments. In this context, the identification problem can be
written in terms of the system of nonlinear equations
\begin{align}\label{eq:system-F_m}
F_m(a,c) = h_m\,, \qquad m=1, \ldots, \ell\,,
\end{align}
where  $F_m:= F_{g_m}$ is defined as in \eqref{eq:def-F}, for each
$m \in\{1,\ldots, \ell\}$.

Moreover, given the nature of the measurements, we  can not expect
that exact data $h_m \in H^{1/2}(\Gamma)$ are available. Instead,
one disposes only an approximate measured data $h_m^\delta \in
L^2(\Gamma)$ satisfying
\begin{align}\label{eq:noise-data}
\norm{h_m^\delta - h_m}_{L^2(\Gamma)} \leq \delta\,,\quad\mbox{for }
m=1,\ldots, \ell\,
\end{align}
where $\delta > 0$ is the noise level.

\begin{remark}
From Theorem~\ref{th:continuity1}, we know that each forward map $F_m$ in \eqref{eq:system-F_m} is continuous in the $[L^1(\Omega)]^2$ topology.
\end{remark}

\subsection{Modeling the parameter space: The level set framework}

In contrast with the previous section, from now on we consider that the pair
of parameters $(a,c)$ are piecewise constant function assuming two
distinct values, i.e. $a(x) \in \{a^1,a^2\}$ and $c(x) \in \{c^1,c^2\}$ a.e. in
$\Omega \subset \mathbb{R}^N$, but we still consider $(a,c) \in D(F)$. 
Hence, one can assume the existence of open and mensurable sets
$\A_1 \subset \subset \Omega$ and $\C_1 \subset \subset \Omega$
with $\mathcal{H}^1(\partial \A_1) < \infty  \mbox{ and }
\mathcal{H}^1(\partial \C_1) < \infty$,\footnote{Here
$\mathcal{H}^1(S)$ denotes the one-dimensional Hausdorff-measure of
the set $S$.} and such that $a(x)=a^1$ if $x \in \A_1$ and $a(x)=a^2$ if $x \in
\A_2:=\Omega - \A_1$; $c(x)=c^1$ if $x \in \C_1$ and $c(x)=c^2$ if $x \in \C_2: =
\Omega - \C_1$.
%
Consequently, the pair of parameters can be modeled as
\begin{equation*} 
(a(x),c(x))  = (a^2 + (a^1-a^2) \chi_{\A_1}(x), c^2 + (c^1-c^2)
\chi_{\C_1}(x)) \, ,
\end{equation*}
where $\chi_{S}$ is the indicator function of the set $S$.

\paragraph{Level set framework:}
In order to model the space of admissible parameters, that is the pair of piecewise constant functions $
(a,c)$, we consider a standard level set (SLS) approach proposed
in \cite{FSL05, CLT09A, CLT09, CL2010}. In particular, our analysis of a level set approach for piecewise constant parameters
follows essentially from the techniques derived in \cite{CLT09}. We notice that many
other level set approaches are known in the literature, see for instance \cite{CL2010_1, DL09, BurOsher05, KALVK99, ZADKS06, DA06}.
For the case where not only the discontinuities but also the values of $a$ and $c$ are unknown then one can use the ideas of the level
set approach presented in \cite{CLT09A}.
Recently, in \cite{CLT11, CL2010, CL2010_1}, piecewise constant level
set approaches (PCLS) were derived for identification of piecewise constant
parameters. The PCLS approach consists in introducing constraints
in the admissible class of level set functions in order to enforce these level set functions
to become piecewise constant. In this context, we do not need to introduce the
Heaviside projector $H$ (see below) to model the parameter space.
However, the introduction of constraints imply different
difficulties in the level set regularization analysis \cite{CLT11,
CL2010}. Advantages and disadvantages of SLS and PCLS approaches
were discussed in \cite{CLT11, CL2010}.


According to the SLS representation strategy, level set functions
$\phia, \phic: \Omega \to \mathbb R$, in $H^1(\Omega)$, are  chosen
in such a way that its zero level set $\Gamma_\phia := \{x \in
\Omega \, ; \ \phia(x) = 0\}$ and $\Gamma_\phic := \{x \in \Omega \,
; \ \phic(x) = 0\}$ define connected curves within $\Omega$ and
the discontinuities of the parameters $(a,c)$ are located ``along''
$\Gamma_\phia$ and $\Gamma_\phic$, respectively.

Introducing the Heaviside projector
$$
H(t) \ := \ \begin{cases}
                     1, & {\rm if \ } t > 0 \\
                     0, & {\rm if \ } t \le 0
                    \end{cases} ,
$$
the diffusion and absorption parameters can be written as
\begin{equation} \label{eq:def-P}
(a, c) \ =  \ \left(a^2 +(a^1-a^2) H(\phia), c^2 + (c^1 - c^2) H(\phic)\right) \ =: \ P(\phia, \phic).
\end{equation}

Notice that $(a(x), c(x))
=(a^i, c^j)$, $x \in \A_i \cap \C_j$ for $i, j \in \{ 1,2\}$, where
the sets $\A_i$ and $\C_j$ are defined by $\A_1 = \{x \in \Omega\,
:\ \phia(x) \geq 0 \}$, $\A_2 = \{x \in \Omega\, :\ \phia(x) < 0
\}$, $\C_1 = \{x \in \Omega\, :\ \phic(x) \geq 0 \}$ and $\C_2 = \{x
\in \Omega\, :\ \phic(x) < 0 \}$.
Thus, the operator $P$ establishes a straightforward relation between the level sets of $\phia$ and
$\phic$ and the sets $\A_i$ and $\C_j$ that characterize the coefficients $(a,c)$.

As already observed in \cite{CLT09}, the operator $H$ maps
$H^1(\Omega)$ into the space
\begin{equation*}
\mathcal{V}_{0,1}\ := \ \{z \in L^\infty(\Omega) \, | \, z =
\chi_{S} \, , \
                      {S} \subset \Omega  {\rm \ measurable, \ }
                      \mathcal{H}^1(\partial S) < \infty \} \,.
\end{equation*}
Therefore, the operator $P$ in \eqref{eq:def-P} maps $H^1(\Omega)
\times H^1(\Omega)$ into the admissible class $\mathcal{V}$ defined by
\begin{equation*} 
\mathcal{V}\ := \ \{(z^1,z^2) \in [L^\infty(\Omega)]^2 \, | \,
(z^1,z^2) = (a^2 + (a^1 - a^2)\,\chi_{\A_1}, c^2 + (c^1 -
c^2)\,\chi_{\C_1})   \, , \mbox{ for some } \ {\A_1, \C_1} \subset \Omega \}\,.
\end{equation*}

Within this framework, the inverse problem in \eqref{eq:system-F_m}, with
data given as in \eqref{eq:noise-data}, can be written in the form
of the operator equation
\begin{equation} \label{eq:inv-probl-fps}
F_m(P(\phia, \phic) ) \ = \ h_m^\delta \qquad m= 1,
\ldots, \ell .
\end{equation}

Let us make the following general assumption:

\bigskip

\begin{minipage}{0.95\textwidth}
\noindent {\bf (A1)} \label{ass:5} Equation \eqref{eq:system-F_m}
has a solution, i.e. there exists $ (a^*, c^*) \in
L^\infty(\Omega) \times L^\infty(\Omega)$ satisfying $F(a^*,
c^*)=h_m$, for $m=1,\ldots,\ell$. Moreover, there exists a pair of functions $(\phia^*, \phic^*) \in
[H^1(\Omega)]^2$ satisfying $P(\phia^*, \phic^*) = (a^*, c^*)$, with $|\nabla \phia^*| \neq 0$ and $|\nabla
\phic^*| \neq 0$ in a neighborhood of $\{\phia^* = 0\}$ and
$\{\phic^* = 0\}$ respectively and such that $H(\phia^*)\! = \!  z^a = \!  \chi_{\A_1} \in L^\infty(\Omega)$,
$H(\phic^*)\!=\!z_c=\! \chi_{\C_1} \in L^\infty(\Omega)$.

\end{minipage}

\bigskip

\subsection{Level set regularization} \label{sec:regularization}

Since the unknown coefficients $(a, c)$ are piecewise
constant functions, a natural alternative to obtain stable solutions
of the operator equation \eqref{eq:system-F_m} is to use a least-square
approach combined with a total variation regularization. This
corresponds to a Tikhonov-type regularization \cite{CLT09A, CLT09,
CL2010}. Within the level set framework presented above, the Tikhonov-type regularization
approach for obtaining a regularized solution to the operator
equation \eqref{eq:inv-probl-fps} is based on the minimization of
the energy functional

\begin{align} \label{eq:Tikhonov-functional}
{\cal F}_\alpha (\phia, \phic) := & \sum_{m=1}^{\ell} \| F_m
(P(\phia, \phic)) - h_m^\delta \|^2_{L^{2}(\Gamma)} + \alpha R(\phia,
\phic)  \,,
\end{align}

where
$$
R(\phia, \phic) =   \left( \beta_a |H(\phia)|_\bv + \beta_c|H(\phic)|_\bv
+ \| \phia - \phiak{0} \|^2_{H^1(\Omega)} + \| \phic - \phick{0}
\|^2_{H^1(\Omega)} \right)\,,
$$
$\alpha > 0$ is the unique regularization parameter and the constants $\beta_j$ play the role of
scaling factors. This approach is based on TV-$H^1$ penalization. 
The $H^1$--terms act as a regularization for the level set functions on the space $H^1(\Omega)$ whereas the $\bv$-seminorm terms are well known for penalizing the length of the
Hausdorff measure of the boundary of the sets $\{x \in \Omega \,:\,
\phia(x) > 0 \}$, $\{x \in \Omega \, : \, \phic(x) > 0 \}$ (see
\cite{EG92}).

In general, variational minimization techniques involve compact
embedding arguments on the set of admissible minimizers and
continuity of the operator in such set to guarantee the existence
of minimizers. The Tikhonov functional in
\eqref{eq:Tikhonov-functional} does not allow such characteristic,
since the Heaviside operator $H$ and consequently the operator $P$
 are discontinuous. Therefore, given a minimizing sequence
$(\phiak{k}, \phick{k})$ for ${\cal F}_\alpha$ we cannot
prove existence of a (weak-*) convergent subsequence. Consequently,
we cannot guarantee the existence of a minimizer in $[H^1(\Omega)]^2$. 
In other words, the graph of ${\cal F}_\alpha$ is not
closed in the desired topology.

To overcome this difficulty in \cite{FSL05, CLT09} was
introduced the concept of generalized minimizers were the graph of
${\cal F}_\alpha$ becomes closed. It allow us to
guarantee the existence of minimizers of the Tikhonov functional
\eqref{eq:Tikhonov-functional}. For sake of completeness, we present
the concept of generalized  minimizers below.



\paragraph{The concept of generalized minimizers:}
 For each $\ve > 0$, we define the smooth approximation to
$H$ given by:
$$
H_\ve(t) := \left\{
  \begin{array}{rl}
    1 + t/\ve  & \mbox{ for \ } t \in \left[-\ve,0\right] \\
    H(t)       & \mbox{ for \ } t \in \mathbb{R} \backslash \left[-\ve,0\right] \\
\end{array} \right.
$$
and the corresponding operator
\begin{equation} \label{eq:def-Pve}
P_{\ve}(\phia, \phic) \ := \ (a^1 H_\ve(\phia) + a^2 (1 -
H_\ve(\phia)),  c^1 H_\ve(\phic) + c^2 (1 - H_\ve(\phic))) \,.
\end{equation}

\begin{df} \label{def:vector}
Let the operators $H$, $P$, $H_\ve$ and $P_{\ve}$ be defined as
above.\\
\smallskip

(a) A {\bf vector} $(z^1, z^2,\phia, \phic) \in
[L^\infty(\Omega)]^2 \times [H^1(\Omega)]^2$ is
called {\bf admissible} when there exist sequences $\{ \phiak{k} \}$
and $\{ \phick{k} \}$ of $H^1(\Omega)$-functions satisfying
$$\lim\limits_{k\to\infty} \| \phiak{k} - \phia \|_{L^2(\Omega)} =
0\,,\quad \lim\limits_{k\to\infty} \| \phick{k} - \phic
\|_{L^2(\Omega)} = 0,$$ 
and there exists a sequence $\{ \ve_k \} \in \mathbb R^+$ converging to zero such that 
$$\lim\limits_{k\to\infty}
\| H_{\ve_k}(\phiak{k})-z^1 \|_{L^1(\Omega)} = 0\, \mbox{ and }
\lim\limits_{k\to\infty} \| H_{\ve_k}(\phick{k})-z^2 \|_{L^1(\Omega)}
= 0\,.$$
\smallskip

(b) A {\bf generalized minimizer} of the Tikhonov
functional $\mathcal{F}_\alpha$ in \eqref{eq:Tikhonov-functional} is
considered to be any admissible vector $(z^1, z^2,\phia,
\phic)$ minimizing
\begin{equation} \label{eq:gzphi}
{\cal{{G}}}_\alpha(z^1, z^2,\phia, \phic) := \sum_{m=1}^{\ell}
     \| F_m(Q(z^1,z^2)) - h_m^\delta \|^2_{L^{2}(\Omega)} + \alpha
     R(z^1,z^2,\phia, \phic)
\end{equation}
over the set of admissible vectors, where 
$$Q: [L^\infty(\Omega)]^2
\ni (z^1,z^2) \mapsto (a^1 z^1 +
a^2(1-z^1), c^1 z^2 + c^2 (1-z^2)) \in [L^{\infty}(\Omega)]^2\,,$$
and the functional $R$ is defined by
\begin{equation*} \label{def:R}
R(z^1, z^2,\phia, \phic) \ := \ \rho(z^1,z^2,\phia, \phic) \, ,
\end{equation*}
with
\begin{align*}
\rho(z^1,z^2,\phia, \phic)  :=  \inf  \left\{ \liminf_{k\to\infty}
\left(\beta_a |H_{\ve_k} (\phiak{k})|_\bv + \beta_c|H_{\ve_k} (\phick{k})|_\bv +  \|
\phiak{k} - \phiak{0}\|^2_{H^1(\Omega)} + \| \phick{k} - \phick{0}
\|^2_{H^1(\Omega)}\right) \right\}\,.
\end{align*}
Here the infimum is taken over all sequences $\{\ve_k\}$ and
$\{\phiak{k}, \phick{k}\}$ characterizing $(z^1,z^2,\phia,
\phic)$ as an admissible vector.
\end{df}

\subsection{Convergence analysis of the level set approach} \label{sec:conv-an}
In this subsection we present the regularization properties of the
proposed level set approach to the inverse problem of identifying
$(a, c)$ in the diffuse optical tomography 
model~\eqref{eq:1}--\eqref{eq:2}.
Since the results follow straightforward arguments presented in
\cite{CLT09A, CLT09, CL2010} we do not present their proofs here.

\begin{theo} The following assertions hold true. \label{th:converg}
\mbox{} \\

i) The functional $\Ga$ in \eqref{eq:Tikhonov-functional} attains
minimizers on the set of admissible vectors.\\

ii) {\bf [Convergence for exact data]} Assume that we have exact
data, i.e. $h^\delta=h$. For every $\alpha > 0$ denote by
$(z^1_\alpha, z^2_\alpha, \phiak{\alpha},
\phick{\alpha} )$ a minimizer of $\Ga$ on the set of
admissible vectors. Then, for every sequence of positive numbers
$\{\alpha_k\}$ converging to zero there exists a subsequence,
denoted again by $\{\alpha_k\}$, such that $(z^1_{\alpha_k},
z^2_{\alpha_k}, \phiak{\alpha_k}, \phick{\alpha_k} )$ is strongly convergent in
$[L^1(\Omega)]^2 \times [L^2(\Omega)]^2$. Moreover,
the limit is a solution of \eqref{eq:system-F_m}. \\

iii) {\bf [Convergence for noisy data]} Let $\alpha = \alpha(\delta)$
be a function satisfying $\lim_{\delta \to 0}$ $\alpha(\delta) = 0$
and $\lim_{\delta \to 0} \delta^2 \alpha(\delta)^{-1} = 0$.
Moreover, let $\{ \delta_k \}$ be a sequence of positive numbers
converging to zero and $\{ h^{\delta_k}\} \in L^{2}(\Gamma)$
be corresponding noisy data satisfying
\eqref{eq:noise-data}. Then, there exists a subsequence, denoted
again by $\{ \delta_k \}$, and a sequence $\{ \alpha_k :=
\alpha(\delta_k) \}$ such that $(z^1_{\alpha_k},
z^2_{\alpha_k},\phiak{\alpha_k}, \phick{\alpha_k})$ converges in
$[L^1(\Omega)]^2 \times [L^2(\Omega)]^2$ to a solution of \eqref{eq:inv-probl-fps}.
\end{theo}
\begin{proof}
The proof follows the arguments presented in \cite{CLT09},
Theorem~6, Theorem~8 and Theorem~9 respectively and therefore is omitted.
\end{proof}

\section{Numerical realization} \label{sec:tow-num-sol}

In this section we introduce the functional $\Ger$, which can be
used for the purpose of numerical implementations. This functional
is defined in such a way that it's minimizers are ``close'' to the
generalized minimizers of $\mathcal{F}_\alpha$ in a sense that will
be clear later (see Proposition~\ref{th:just}). For each $\eps
> 0$ we define the functional
\begin{align} \label{eq:m-reg}
\Ger(\phia, \phic) \ := \sum_{m=1}^{\ell} \|
F_m(P_\eps(\phia,\phic)) - h_m^\delta \|^2_{L^{2}(\Gamma)} +
    \alpha R_\ve(\phia, \phic)\; ,
\end{align}
where
\begin{align}\label{eq:Rve}
R_\ve(\phia, \phic) : = \left( \beta_a |H_\ve(\phia)|_{\bv} + \beta_c
|H_{\ve} (\phic)|_{\bv} + \| \phia - \phiak{0} \|_{H^1(\Omega)}^2 + \| \phic - \phick{0} \|_{H^1(\Omega)}^2 \right)\,.
\end{align}

The next result guarantees that for $\eps \to 0$ 
the functional $\Ger$ attains a
minimizer. Moreover,  the minimizers of $\Ger$ approximate a generalized minimizer
 of $\mathcal{F}_\alpha$.

\begin{pr} \label{th:just}
\mbox{}\\
i) Given $\alpha$, $\beta_j$, $\eps > 0$ and $\phiak{0}$, $\phick{0}$ in
$H^1(\Omega)$, then the functional $\Ger$ in \eqref{eq:m-reg}
attains a minimizer on $[H^1(\Omega)]^2$.\\
ii) Let $\alpha$, $\beta_j$ be given. For each $\eps > 0$ denote by
$(\phiak{\eps,\alpha},\phick{\eps,\alpha})$ a minimizer of $\Ger$. There exists a sequence of positive
numbers $\{ \eps_k \}$ converging to zero such that $(
H_{\eps_k}(\phiak{{\eps_k},\alpha}),
H_{\eps_k}(\phick{{\eps_k},\alpha}), \phiak{{\eps_k},\alpha},
\phick{{\eps_k},\alpha})$ converges
strongly in $[L^1(\Omega)]^2 \times [L^2(\Omega)]^2 $ and the
limit is a generalized minimizer of $\mathcal{F}_\alpha$ in the set
of admissible vectors.
\end{pr}
\begin{proof}
The proof follows from Lemma~10 and Theorem~11 presented in \cite{CLT09}.
Therefore, we do not present the details it in this paper.
\end{proof}

Proposition~\ref{th:just} justifies the use of functional $\Ger$ in
order to obtain numerical approximations to the generalized
minimizers of $\mathcal{F}_\alpha$. It is worth noticing that,
differently from $\mathcal{F}_\alpha$, the minimizers of $\Ger$ can
be actually computed. In the next subsection we derive the first
order optimality conditions for the functional $\Ger$, which will
allow us to compute the desired minimizers.

\subsection{Optimality conditions for the Tikhonov functional $\Ger$}
For the numerical purposes we have in mind, it is necessary to
derive the first order optimality conditions for a minimizer of the
functional $\Ger$. To this end, we consider $\Ger$ in
\eqref{eq:m-reg} and we look for the G\^ateaux directional
derivatives with respect to $\phia$, $\phic$. In order to simplify 
the presentation, we will assume that the values $a^1, a^2, c^1, c^2$ are known.
Since $H'_\eps(\varphi)$ is self-adjoint,%
\footnote{Notice that $H'_\eps(t) = \begin{cases} 1/\eps & t \in (-\eps, 0) \\
0 & else \end{cases}$.} the optimality conditions for a minimizer of
the functional $\Ger$ can be written in the form of the system of
equations
\begin{subequations} \label{eq:formal0} \begin{eqnarray}
\alpha (\Delta-I)(\phia - \phiak{0}) = L_{\eps,\alpha}^a(\phia,\phic)
\, , \ \ \ \alpha (\Delta-I)(\phic - \phick{0}) =
L_{\eps,\alpha}^c(\phia,\phic) \, , \,
{\rm in}\ \Omega \label{eq:formal01} \\
\frac{ \partial}{\partial \nu}(\phia - \phiak{0}) = 0 \, , \ \ \ \ \ \ \ \ \ \ \ \ \ \ \ \ \ \ \ \
 \ \ \, \frac{ \partial}{\partial \nu}(\phic - \phick{0}) = 0 \, , \quad  \quad   {\rm on} \ \Gamma
\ \ \ \ \ \ \ \ \ \, \label{eq:formal02}
\end{eqnarray} \end{subequations}
where $\nu(x)$ is the external unit normal vector at $x \in \Gamma$ and
\begin{subequations} \label{eq:formal1}
\begin{eqnarray}
L_{\eps,\alpha}^a(\phia,\phic) & = &  (a^1 - a^2) \,
H'_\eps(\phia)\,
      \left[\sum_{m=1}^l \left(\frac{ \partial F_m( P_\eps(\phia,\phic) )}{\partial \phia }\right)^* ( F_m(P_\eps(\phia,\phic)) - h_m^\delta ) \right] \nonumber \\
&   & - \alpha \beta_a \, \left[ H'_\eps(\phia) \, \nabla \!\cdot\!
      \left( \frac{\nabla H_\eps(\phia)}{|\nabla H_\eps (\phia)|} \right) \right]\label{eq:formal10} \\
%
L_{\eps,\alpha}^c(\phia,\phic) & = &  (c^1 - c^2) \,
H'_\eps(\phic)\, \left[ \sum_{m=1}^l \left(\frac{ \partial F_m( P_\eps(\phia,\phic) )}{\partial \phic }\right)^* ( F_m(P_\eps(\phia,\phic)) - h_m^\delta ) \right] \nonumber \\
&   & - \alpha \beta_c \, \left[ H'_\eps(\phic) \, \nabla \!\cdot\!
      \left( \frac{\nabla H_\eps(\phic)}{|\nabla H_\eps (\phic)|} \right) \right]. \label{eq:formal11}
\end{eqnarray}
\end{subequations}

Note that, in order to implement the numerical algorithm for solving
the optimality conditions, we need to calculate the
adjoint of the derivatives $\frac{\partial F_m}{\partial \phia}$ 
and $\frac{\partial F_m}{\partial \phic }$.

\begin{remark}\label{remark:F-adjoint-derivative}
Given the level set functions $\phia, \phic \in H^1(\Omega)$ and inputs $g_m \in H^{1/2}(\Gamma)$ for $m=1,\ldots, \ell$,  denote the residual $r_m:=  F_m(P_\eps(\phia,\phic)) -
h_m^\delta \in L^2(\Gamma)$. Then, 
\begin{align}\label{eq:derifative-F-a}
\left(\frac{ \partial F_m( P_\eps(\phia,\phic) )}{\partial \phia }\right)^* r_{m} =
\nabla u_{m} \cdot \nabla w_{m}\,
\end{align}
and
\begin{align}\label{eq:derifative-F-c}
\left(\frac{ \partial F_m( P_\eps(\phia,\phic) )}{\partial \phic }\right)^* r_{m}
= - u_{m} \, w_{m}\,
\end{align} 
where $u_{m}$ and $w_{m}$ are the unique solutions of the following elliptic boundary problems
\begin{align}\label{eq:pde-u}
- \nabla (a \nabla u_{m}) + c\, u_{m} & = 0 \,,\quad \,\,\,\, \mbox{in } \quad \Omega\\
 a \frac{\partial u_{m}}{\partial \nu} & = g_m \,,\quad \mbox{on } \quad \Gamma\,, \nonumber
\end{align}
\begin{align}\label{eq:pde-w}
- \nabla (a \nabla w_{m}) + c\, w_{m} & = 0 \,,\quad \,\,\,\, \mbox{in } \quad \Omega\\
  a \frac{\partial w_{m}}{\partial \nu} & = r_{m} \,,\quad \mbox{on } \quad \Gamma\,, \nonumber
\end{align}
for $m=1,\ldots, \ell$ respectively.
\end{remark}

We have already introduced all the ingredients necessary to implement an algorithm based on the 
level set regularization approach to solve the identification problem in
diffuse optical tomography (see Table~\ref{tab:algor-old}). The iterative
algorithm consists in minimizing, for $k \geq 1$, the functional
\begin{align} 
\Ger^{(k)}(\phia, \phic) \ := \sum_{m=1}^{\ell} \|
F_m(P_\eps(\phia,\phic)) - h_m^\delta \|^2_{L^{2}(\Gamma)} +
    \alpha R^{(k)}_\ve(\phia, \phic)\; ,
\end{align}
where $R^{(k)}_\ve$ is the functional $R_\ve$ defined in (\ref{eq:Rve}) with $\phi^j_0$ replaced by $\phi^j_{k-1}$.
The minimizer of each functional can be computed solving the formal optimality conditions (\ref{eq:formal0}) with $\phi^j_0$ replaced by $\phi^j_{k-1}$.

Each iteration of the proposed algorithm consists in the next five steps:

\begin{minipage}{17.0cm}
\noindent $\bullet$ In the first step the residual vector $[r_{k,m}
]_{m=1}^\ell \in [L^2(\Gamma)]^\ell$, corresponding to the $k$th iteration
$(\phiak{k},\phick{k})$, is evaluated. This requires the solution of
$\ell$ elliptic BVP's given by \eqref{eq:pde-u}.
\smallskip

\noindent $\bullet$ The second step consists in computing the adjoint of the
partial derivatives of $F_m$ applied to the residuals. This is done by solving $\ell$ elliptic BVP given by \eqref{eq:pde-w} to get 
the solutions $[w_{k,m}]_{m=1}^\ell \in [H^1(\Omega)]^{\ell}$ and then computing the products given by Remark~\ref{remark:F-adjoint-derivative}.
\smallskip

\noindent $\bullet$ In the third step, the terms 
$L_{\eps,\alpha}^a(\phiak{k},\phick{k})$ and
$L_{\eps,\alpha}^c(\phiak{k},\phick{k})$ given by equations~\eqref{eq:formal10} and ~\eqref{eq:formal11} are calculated. 
\smallskip

\noindent $\bullet$ The fourth step consists in computing the updates $\delta\phiak{k}$,
$\delta\phick{k} \in H^1(\Omega)$ for the level-set functions
$\phi^a$ and $\phi^c$. This corresponds to solving
two non-coupled elliptic BVP's, namely \eqref{eq:formal01} and \eqref{eq:formal02}.

\noindent $\bullet$ Finally, update the level set functions and go to step 1 until a stopping criteria is reached.

\end{minipage}

\begin{table}[t!]
\fbox{ \begin{minipage}{17.0cm}
\begin{tt}\begin{itemize}
\item[{\bf 1.}]  Evaluate the residual
$[r_{k,m}]_{m=1}^\ell := [ F_m( P_\eps(\phiak{k},\phick{k}) ) - h_m^\delta
]_{m=1}^\ell = [ u_{k,m} |_{\Gamma} - h_m^\delta ]_{m=1}^\ell$, where $[
u_{k,m} ]_{m=1}^\ell \in [ H^1(\Omega) ]^\ell$ and each function solves \eqref{eq:pde-u}.
\medskip

\item[{\bf 2.}]  Evaluate $\left[ \left(\frac{ \partial F_m( P_\eps(\phiak{k},\phick{k}) )}{\partial \phiak{k}}\right)^* \, r_{k,m} \right]_{m=1}^\ell =
[ \nabla w_{k,m} \cdot \! \nabla u_{k,m} ]_{m=1}^\ell \in [ L^2(\Omega) ]^\ell$, and
$\left[ \left(\frac{ \partial F_m( P_\eps(\phiak{k},\phick{k}) )}{\partial
\phick{k}}\right)^* \, r_{k,m} \right]_{m=1}^\ell = -[  w_{k,m} \, u_{k,m} ]_{m=1}^\ell \in [
L^2(\Omega) ]^\ell$, where $[u_{k,m}]_{m=1}^\ell$ are the functions computed in
Step {\bf 1} and $[w_{k,m}]_{m=1}^\ell \in [ H^1(\Omega) ]^\ell$ solve
\eqref{eq:pde-w}.
\medskip

\item[{\bf 3.}] Calculate $L^a_{\eps,\alpha}(\phiak{k},\phick{k})$ and
$L^c_{\eps,\alpha}(\phiak{k},\phick{k})$ given by equations \eqref{eq:formal10} and \eqref{eq:formal11}.

\item[{\bf 4.}] Evaluate the updates $\delta\phiak{k}$, $\delta\phick{k} \in H^1(\Omega)$ by
solving 
\medskip

\centerline{
 \hfil $(\Delta - I) \delta\phi^j_k = L^j_{\eps,\alpha}(\phiak{k},\phick{k})
           \, , \, {\tt in}\ \Omega \, ; \quad
           \frac{\partial \delta\phi^j_k}{\partial \nu} = 0 \, , \, {\tt on}\ \Gamma \; .$ 
 \hfil} \smallskip

\item[{\bf 5.}] Update the level set functions\,
$\phiak{k+1} = \phiak{k} + \frac{1}{\alpha} \; \delta\phiak{k},$ \,
$\phick{k+1} = \phick{k} + \frac{1}{\alpha} \; \delta\phick{k}$.
\end{itemize} \end{tt}
\end{minipage}
} \caption{An explicit algorithm based on the iterative
regularization method for solving the identification problem in
diffuse optical tomography. \label{tab:algor-old} }
\end{table}
\medskip

A similar algorithm was successfully implemented in \cite{FSL05,CLT09} to
solve the inverse potential problem under the framework of
level sets and multiple level sets respectively.
Regarding our coefficient identification problem in diffuse optical tomography
the algorithm outlined above also seems to be effective (see next Section), but in this case, it has the disadvantage
that in each iteration step one has to solve $2\ell+2$ elliptic BVP's. Then, if the number $\ell$ of experiments is large the computational cost will be
high.

\newpage
\section{Numerical Experiments} \label{sec:numerical}

In this section we implement a 
numerical algorithm based on the level set approach
derived in the previous sections for identifying the coefficient pair $(a,c)$ in
\eqref{eq:1}--\eqref{eq:2}. First, the identification of the absorption coefficient
$c$, based on either total or partial knowledge of $a$, is considered in Section~%
\ref{ssec:numer-c}. Then, the separate identification of the diffusion coefficient $a$,
based on either total or partial knowledge of $c$, is considered in Section~%
\ref{ssec:numer-a}. Finally, the simultaneous identification of the pair $(a,c)$ is investigated
in Section~\ref{ssec:numer-ac}.

In all the numerical experiments of this Section we considered $\Omega = (0,1)\times(0,1)$ and four ($\ell=4$) different inputs $g_m \in L^2{(\Gamma})$ were applied
as Neumann boundary conditions in order to compute the corresponding Dirichlet data $h_m$. Each of
these functions is supported at one of the four sides of $\Gamma$, for instance 
$$
g_1(x) = \left\{ \begin{array}{rl} 1 \, , & {if} \, x \,  \in \,  (\frac14, \frac34)\times \{0\} \\
                                    0 \, , & {\rm else} \end{array} \right. , 
$$
and $g_2, g_3$ and $g_4$ are defined in a similar way. 
All boundary value problems were solved using a Galerkin Finite Element method in an uniform grid with 50 nodes at
each boundary side. We used a custom implementation using MATLAB.

\begin{figure}[ht!]
\centerline{
\includegraphics[width=4.4cm]{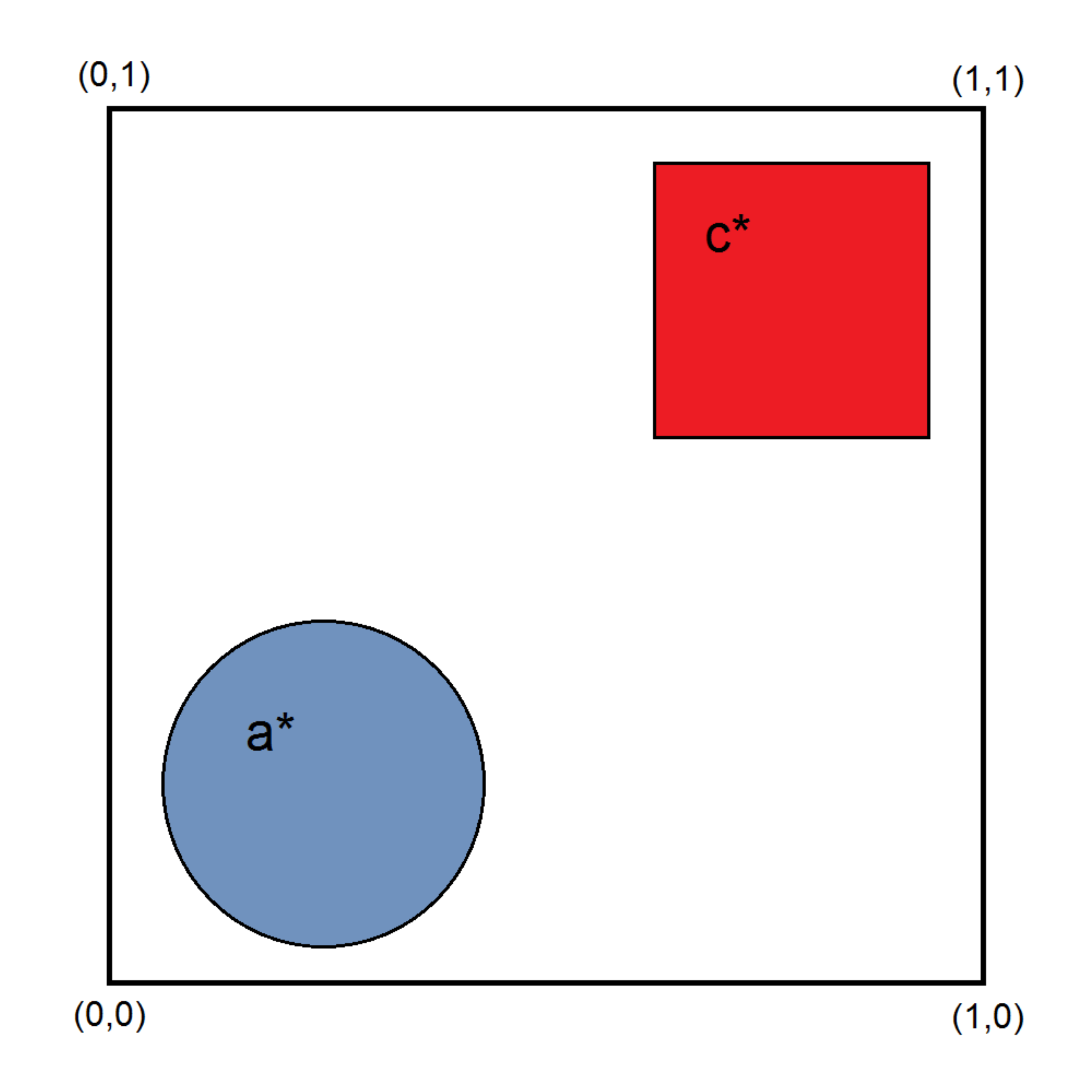}
\includegraphics[width=4.4cm]{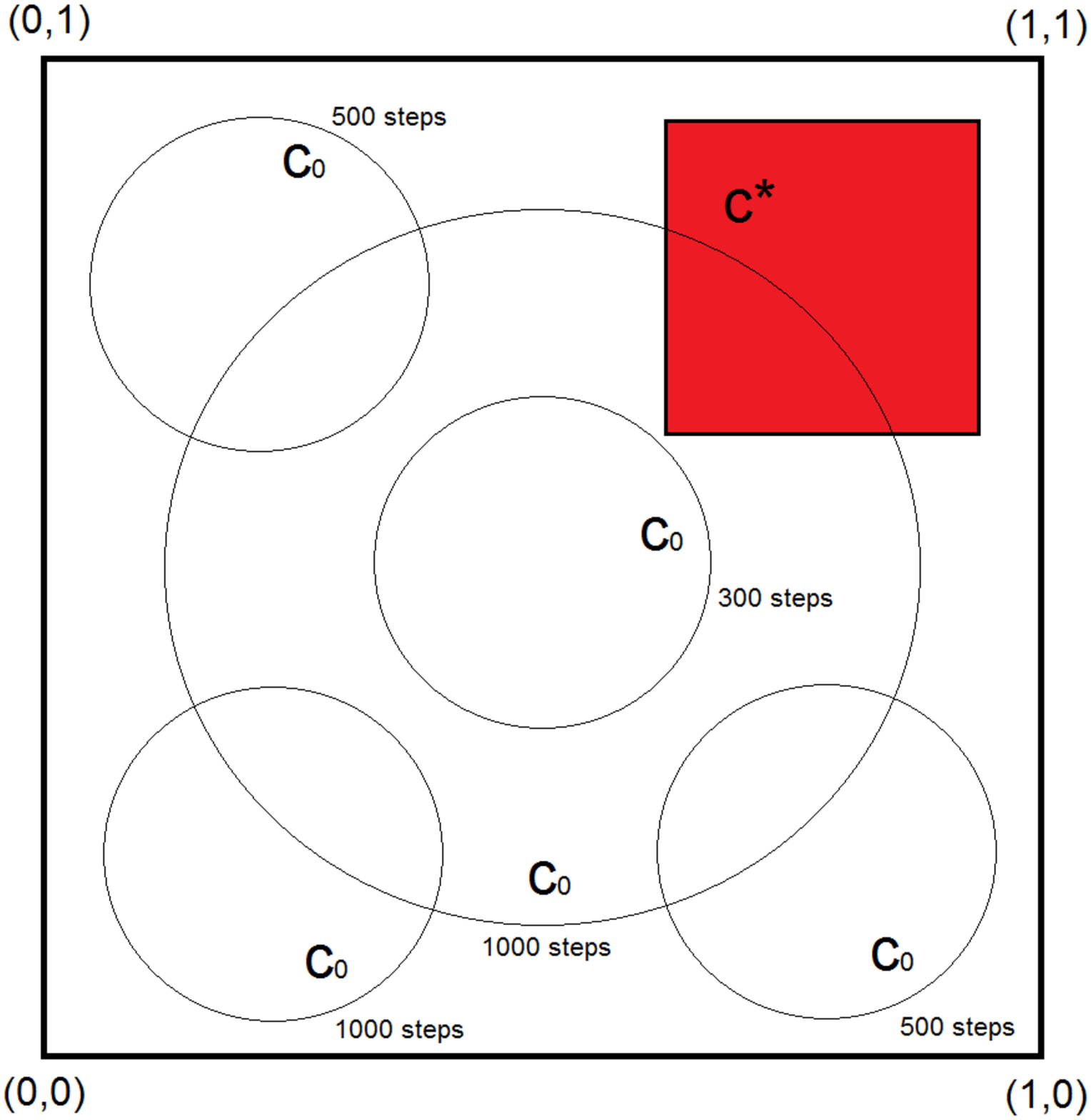}
\includegraphics[width=4.4cm]{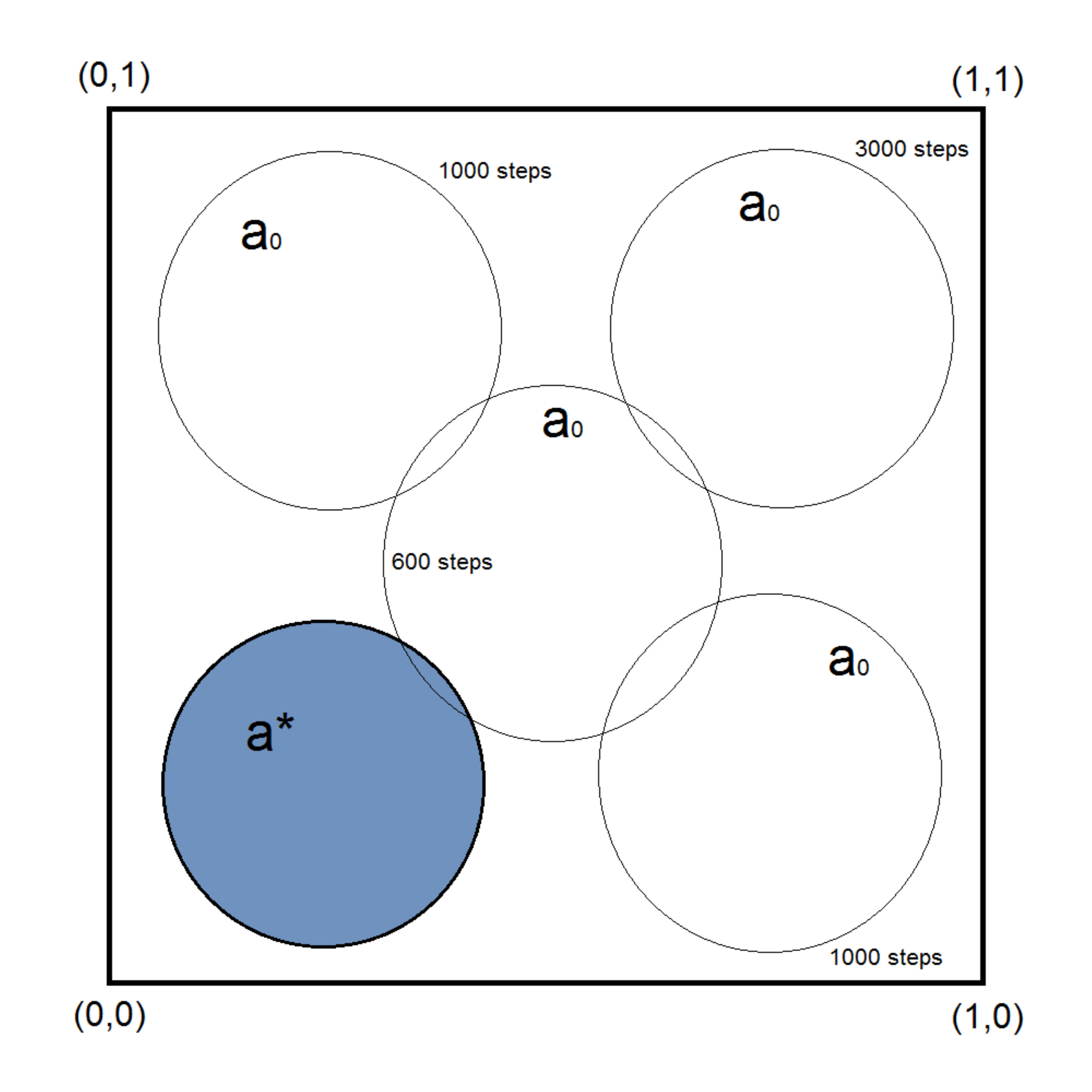}
\includegraphics[width=4.4cm]{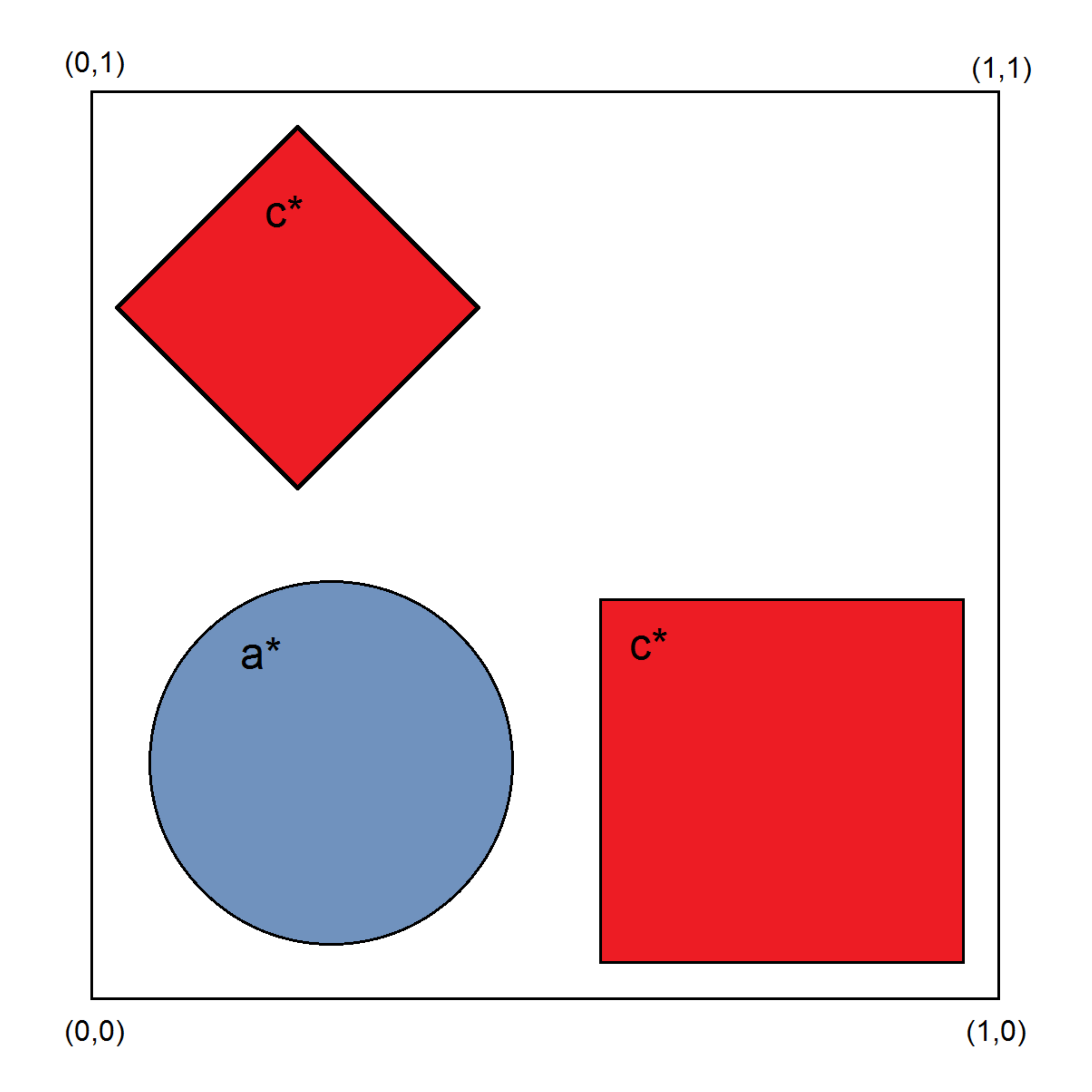} 
}
\centerline{\hfil (a) \hskip4cm (b) \hskip4cm (c) \hskip4cm (d) \hfil}
\caption{\small
{\bf (a)} Exact coefficients for the first experiments in Sections~\ref{ssec:numer-c}
and~\ref{ssec:numer-a}.
{\bf (b)} Number of iterations needed for the identification of the absorption
coefficient $c^*$, starting from distinct initial guesses $c_0$ ($a^*$ is given;
see Section~\ref{ssec:numer-c}).
{\bf (c)}  Number of iterations needed for the identification of the diffusion
coefficient $a^*$, starting from distinct initial guesses $a_0$ ($c^*$ is given;
see Section~\ref{ssec:numer-a}).
{\bf (d)} Exact coefficients for the second experiments in Sections~\ref{ssec:numer-c}
and~\ref{ssec:numer-a}. }
\label{fig:separate-ident}
\end{figure}

\subsection{Identification of the absorption coefficient} \label{ssec:numer-c}

In what follows we consider the identification of the absorption coefficient $c$, based on
either total or partial knowledge of $a$. The values assumed for these coefficients were (see Figure~\ref{fig:separate-ident}):

$$
a^*(x) = \left\{ \begin{array}{rl} 10 \, , & {\rm inside \ blue \ inclusion} \\
                                    1 \, , & {\rm elsewhere} \end{array} \right. , \quad
c^*(x) = \left\{ \begin{array}{rl} 10 \, , & {\rm inside \ red \ inclusion} \\
                                    1 \, , & {\rm elsewhere.} \end{array} \right.
$$
As initial guess for the level set method we have chosen distinct piecewise constant functions
$c_0$, whose supports are shown in Figure~\ref{fig:separate-ident}~(b). It is worth noting that each $c_0$ corresponds
to a level set function $\phick{0} \in H^1(\Omega)$. In all cases the initial level set
function $\phick{0}$ was a paraboloid but with different minima.

The constant values assumed by the exact solution $c^*$ are supposed to be known, as well as the
exact diffusion coefficient $a^*$. Moreover, exact data was considered for the reconstruction
(i.e., $\delta = 0$) and we tested the iterative level set regularization without the penalizing
term $|H_{\ve} (\phi^j)|_{\bv}$, i.e., $\beta_j = 0$ (see \cite[Remark 5.1]{FSL05} ). 

In this and in all the following computed experiments of this Section, we considered the operator $P_{\ve}$ defined in~(\ref{eq:def-Pve}) with $\ve=1/10$.
This election was motivated by the fact that as $\ve$ increases, the supports
of the functions appearing on the right-hand side of \eqref{eq:formal01} and \eqref{eq:formal02}
become larger (due to the term $H_{\ve}$). Consequently, 
the updates $\delta\phiak{k}$, $\delta\phick{k} \in H^1(\Omega)$ given by these equations have large
values. If $\ve$ becomes too large, the level set method becomes unstable. Therefore, $\ve$ was chosen to match the
mesh size considered to solve the boundary problems.

The inverse problem we tackle here reduces to a shape identification problem for the absorption
coefficient.

Notice that, for each initial guess $c_0$ in Figure~\ref{fig:separate-ident}~(b), a
corresponding number of steps is plotted. It stands for the number of iterations needed
to compute an approximation of $c^*$ (starting from the corresponding $c_0$) with a
precision of $10^{-2}$ in the $L^2$-norm.

This experiment allow us to determine the computational effort necessary for the
reconstruction of $c^*$ with respect to distinct choices of $c_0$.
{\bf The identification problem for the absorption coefficient is known to be mildly
ill-posed}~\cite{EngHanNeu96, KNS08}.
This fact is in agreement with the values plotted in Figure~\ref{fig:separate-ident}~(b),
in the sense that the number of iterations necessary to achieve a good quality
reconstruction {\bf do not strongly oscillate} with the initial guess.

\begin{figure}[b!]
\centerline{
\includegraphics[width=5.7cm]{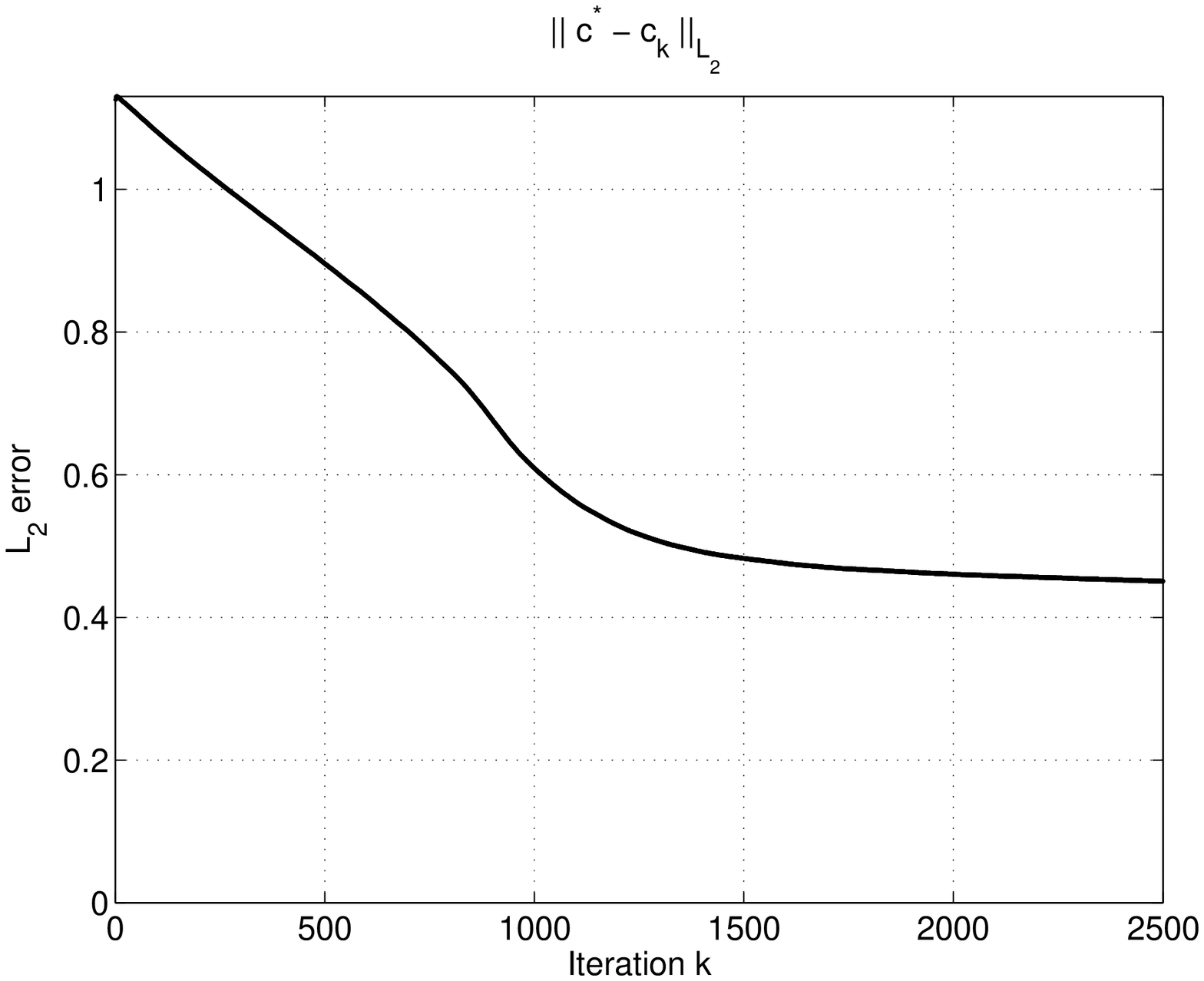}       
\includegraphics[width=5.5cm] {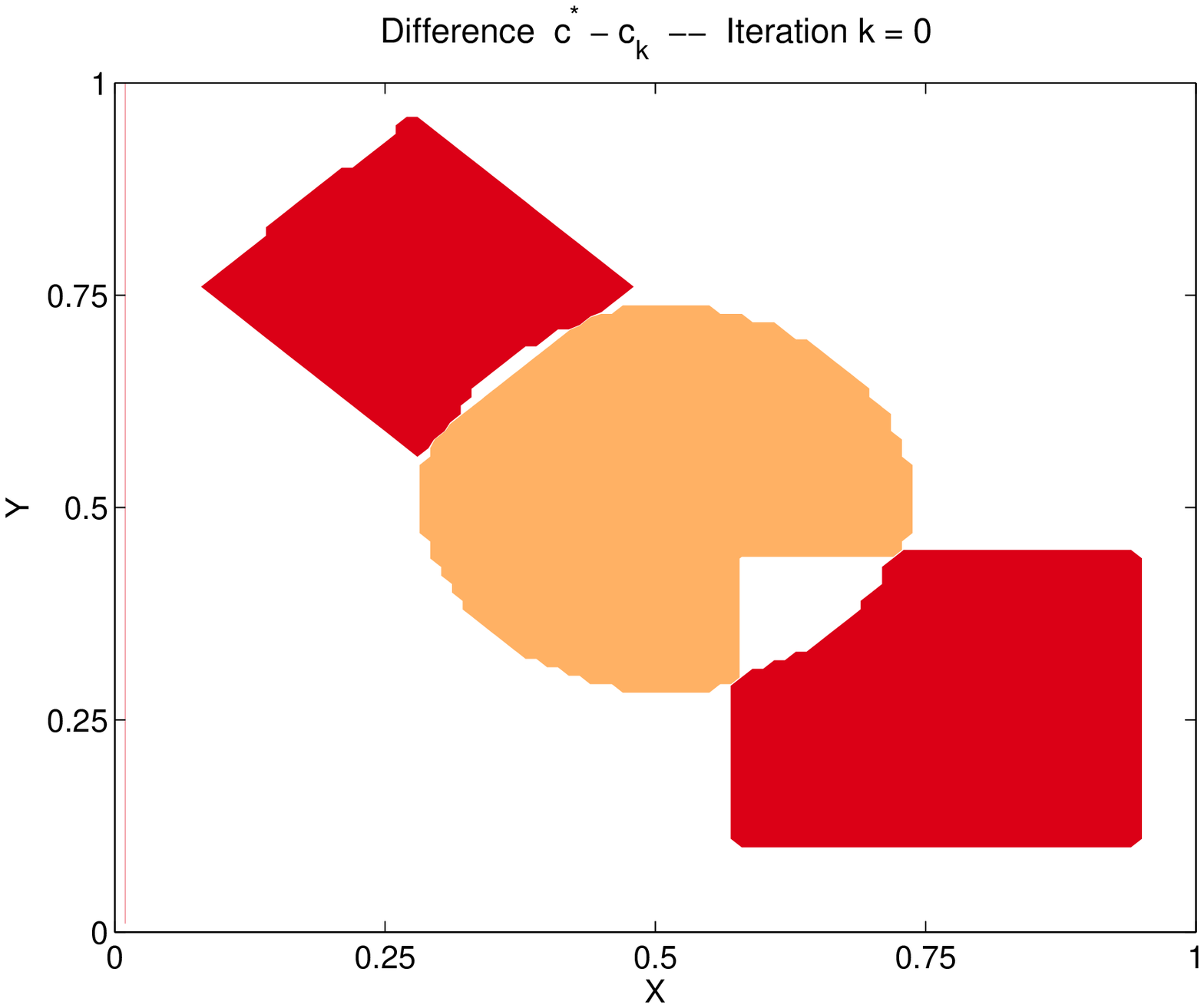}           
\includegraphics[width=5.5cm] {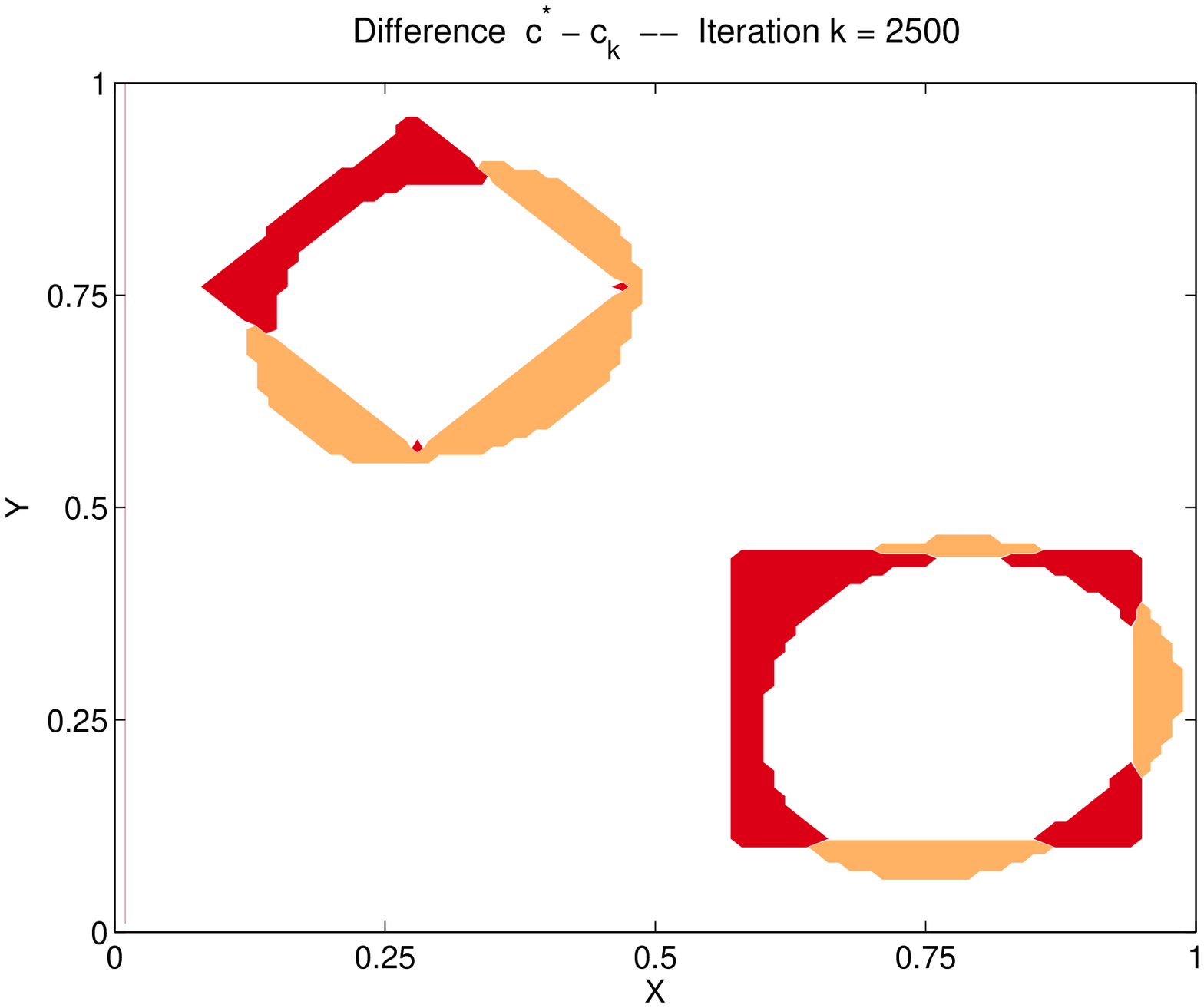} } 
\centerline{\hfil (a) \hskip5.5cm (b) \hskip5.5cm (c) \hfil}
\centerline{
\includegraphics[width=5.7cm]{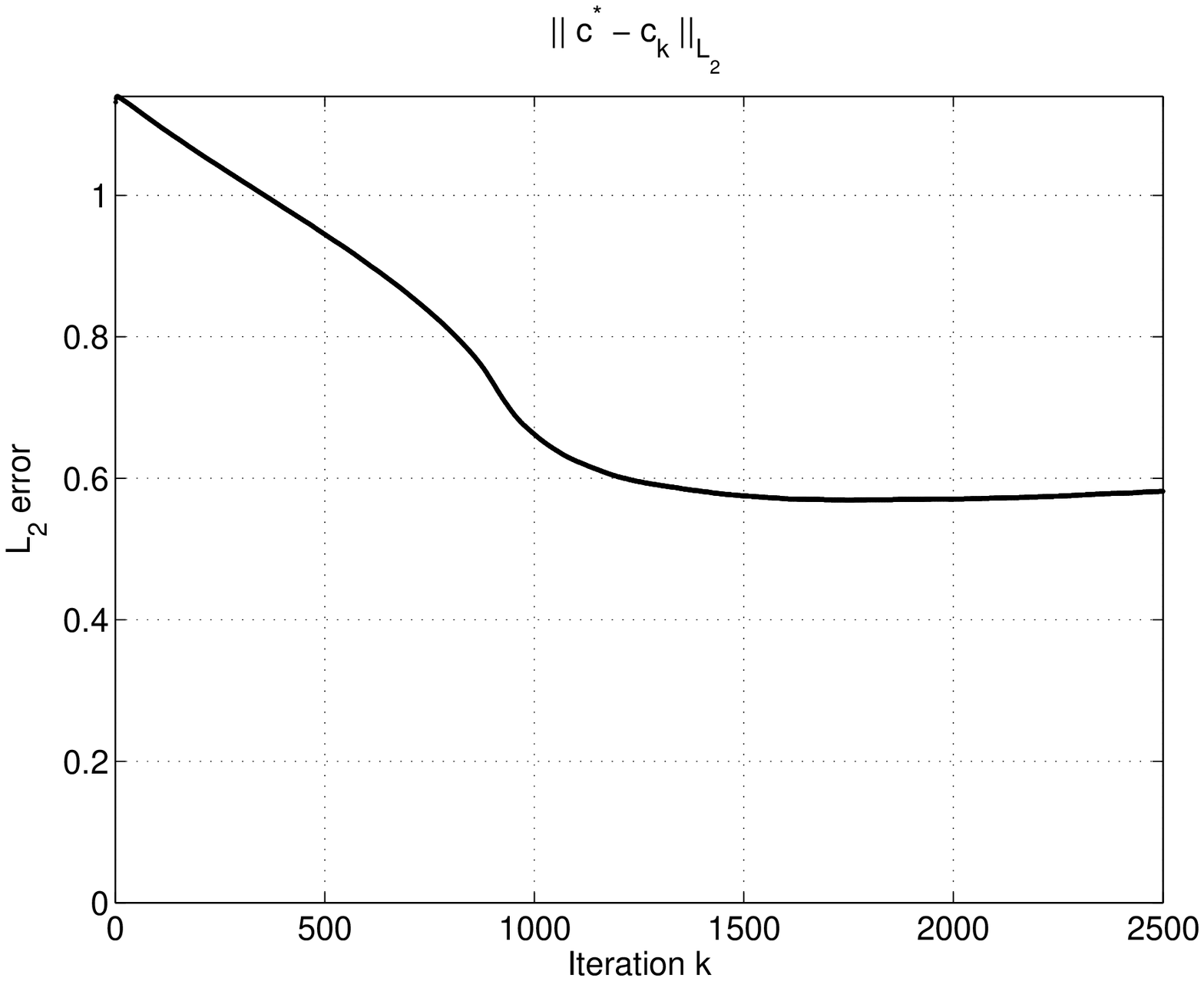} 
\includegraphics[width=5.5cm]{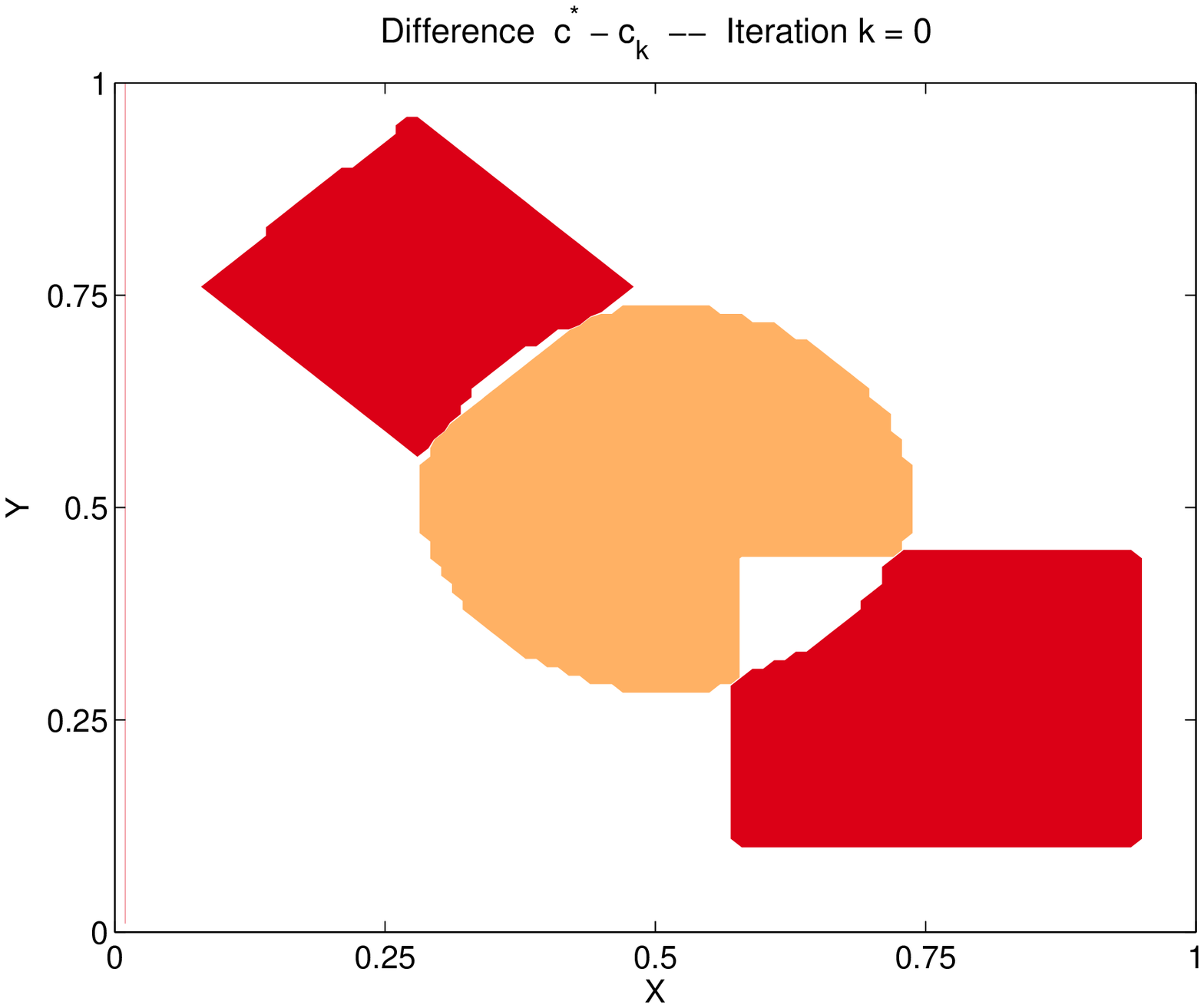}  
\includegraphics[width=5.5cm]{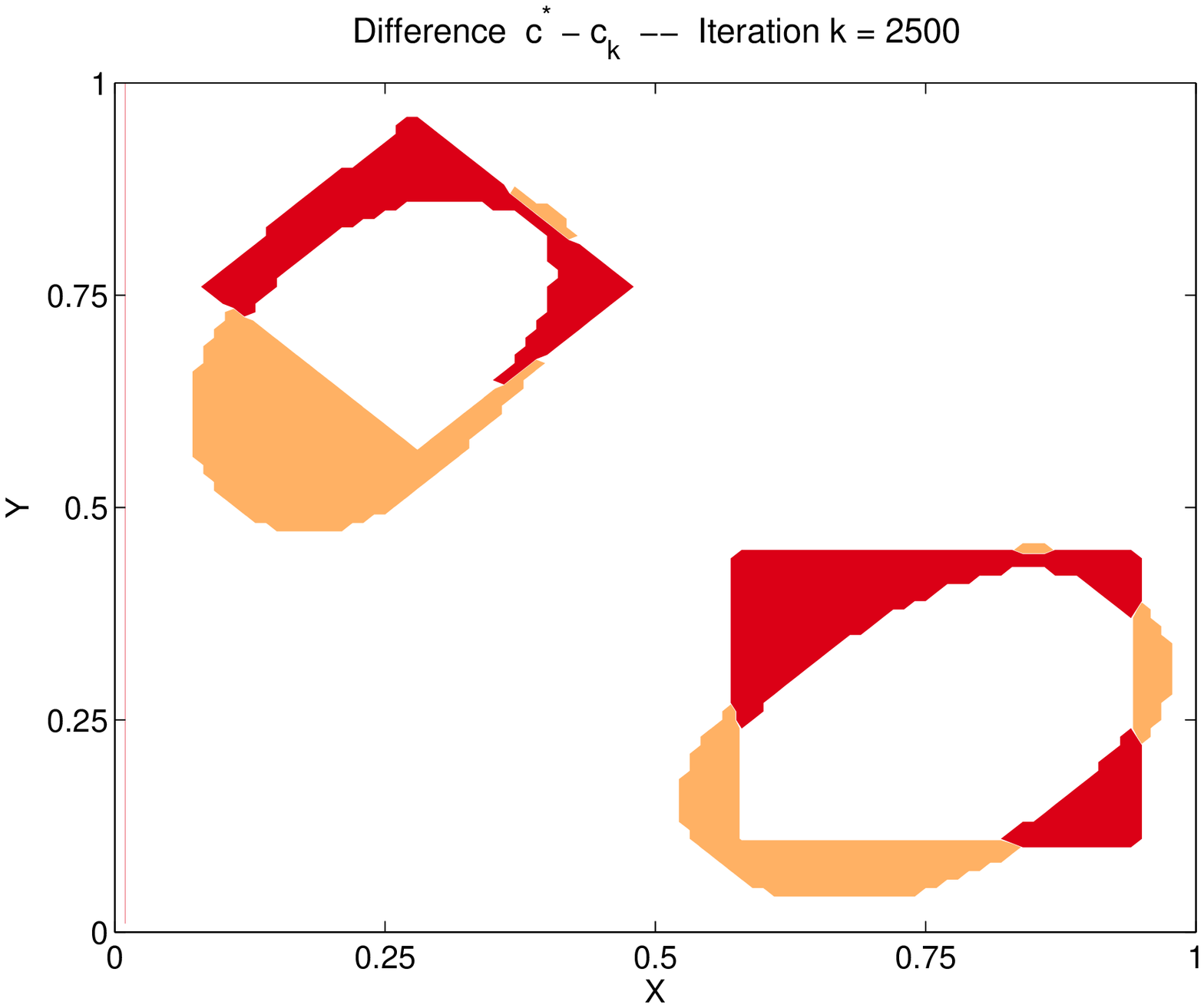} }  
\centerline{\hfil (d) \hskip5.5cm (e) \hskip5.5cm (f) \hfil}
\caption{\small Section~\ref{ssec:numer-c}, 2nd example.
{\bf (a)}--{\bf (c)} Identification of $c^*$ from exact knowledge of $a^*$.
{\bf (a)} Evolution of the $L^2$ error.
{\bf (b)} Difference between the initial guess $c_0$ and $c^*$.
{\bf (c)} Difference between $c_{2500}$ and $c^*$.
{\bf (d)}--{\bf (f)} Identification of $c^*$ from partial knowledge of $a^*$.
{\bf (d)} Evolution of the $L^2$ error.
{\bf (e)} Difference between the initial guess $c_0$ and $c^*$.
{\bf (f)} Difference between $c_{2500}$ and $c^*$. }
\label{fig:example-c2}
\end{figure}

We conduct yet another experiment for identifying only the absorption coefficient.
This time, we assume the exact solution of 
problem \eqref{eq:1}--\eqref{eq:2} to be
given by the coefficient pair $(a^*,c^*)$ in Figure~\ref{fig:separate-ident}~(d).
The setup of the inverse problem remains the same (domain, available data, parameter
to output operator, etc.). \\
On the first run of the algorithm, see Figure~\ref{fig:example-c2}~(a)--(c), the
diffusion coefficient $a^*$ is assumed to be exactly known. In this situation, the level set method is
able to identify the absorption coefficient (see Figure~\ref{fig:example-c2}~(a)
for the evolution of the iteration error), and the iteration stagnates after that.
The corresponding differences between the exact solution $c^*$ and the initial guess $c_0$ and between the exact solution $c^*$ and the final
iterate $c_{2500}$ are plotted in pictures~(b) and~(c) respectively. \\
On the second run, see Figure~\ref{fig:example-c2}~(d)--(f), we use the approximation
$a(x) \equiv 1$ for diffusion coefficient $a^*$ and iterate to recover $c^*$. In this case, the
level set method is still able to identify the absorption coefficient, however with a
poorer accuracy. Once again, the iteration stagnates after the numerical convergence
is reached (see Figure~\ref{fig:example-c2}~(d) for the evolution of the error). The
corresponding differences for the initial guess $c_0$ and for the
final iterate $c_{2500}$ are plotted in pictures~(e) and~(f) respectively. \\
Notice that the number of iteration steps needed to recover $c^*$ (approximately
2000 in both runs) is much larger than in the previous experiment. This can be
explained by the complexity of the geometry of the support of $c^*$ in this
experiment \cite{FSL05}. This complexity and non smooth geometry also influence the quality of the reconstruction. 

\subsection{Identification of the diffusion coefficient} \label{ssec:numer-a}

In what follows we consider the identification of 
the diffusion coefficient $a$, based on
either total or partial knowledge of $c$. In the first set of experiments, we consider
problem \eqref{eq:1}--\eqref{eq:2} in the unit square with four pairs of NtD experiments,
and the same exact solution $(a^*,c^*)$ as in Section~\ref{ssec:numer-c} (see Figure~%
\ref{fig:separate-ident}~(a)).

\begin{figure}[b!]
\centerline{
\includegraphics[width=5.7cm]{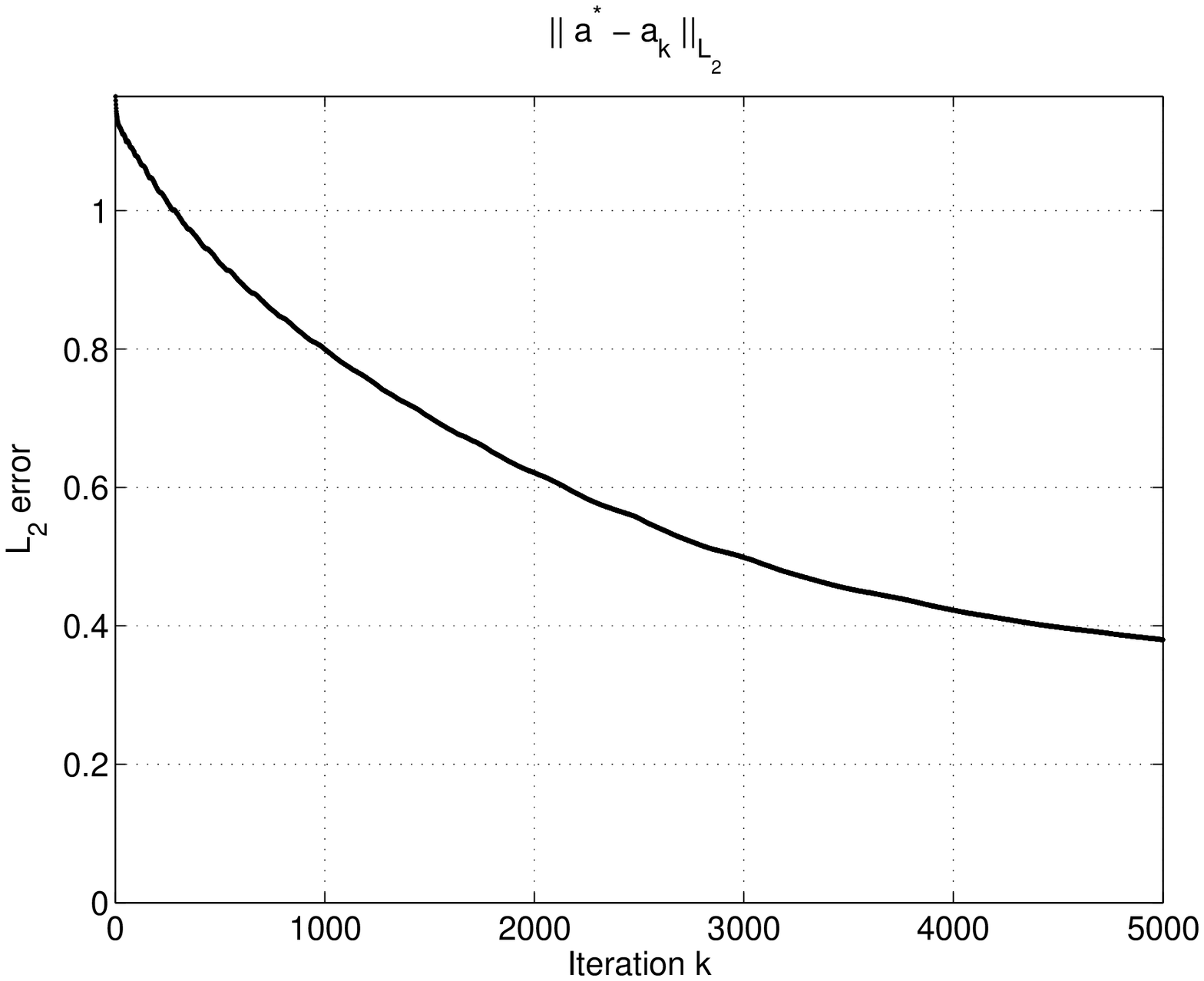}      
\includegraphics[width=5.5cm]{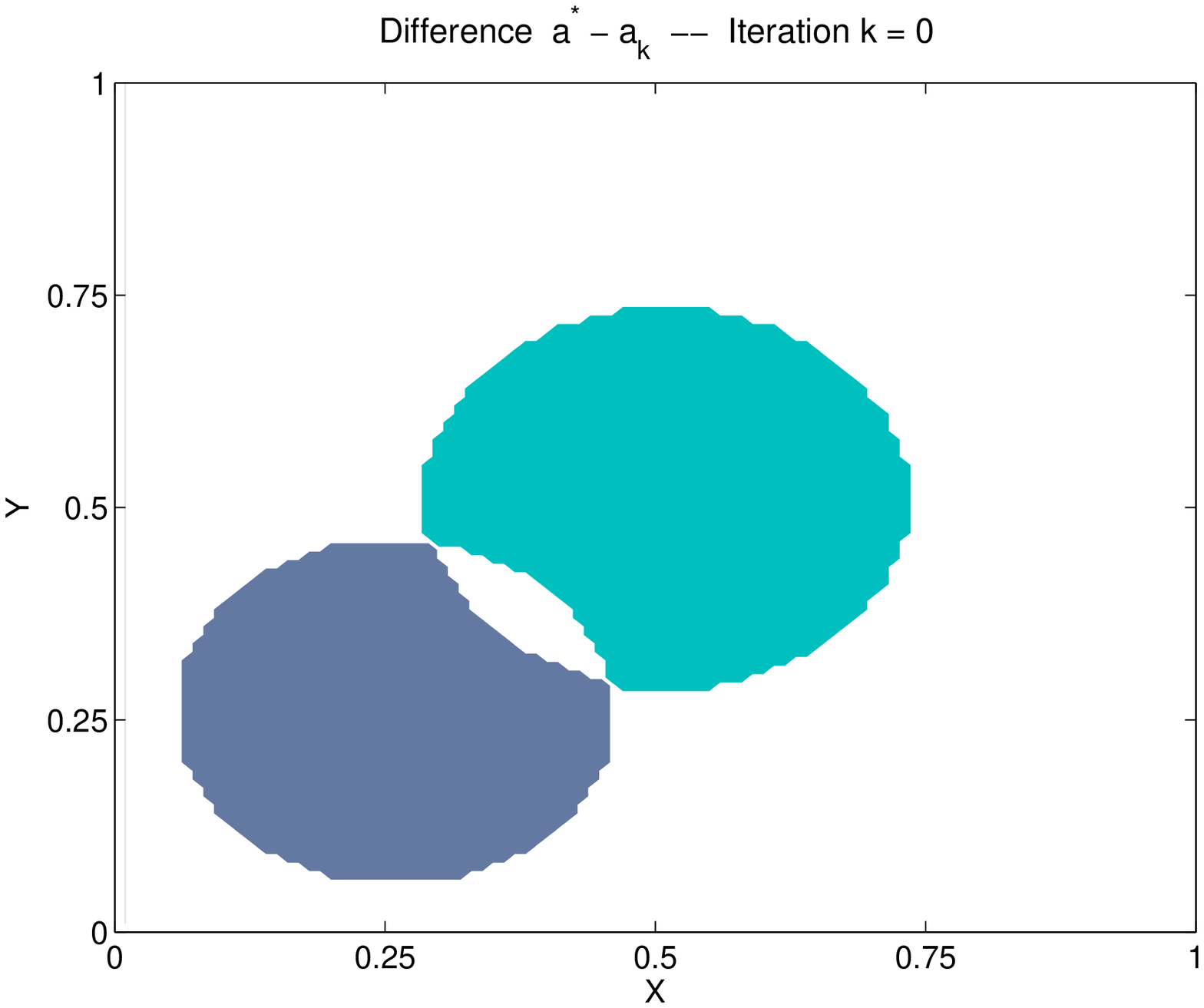}      
\includegraphics[width=5.5cm]{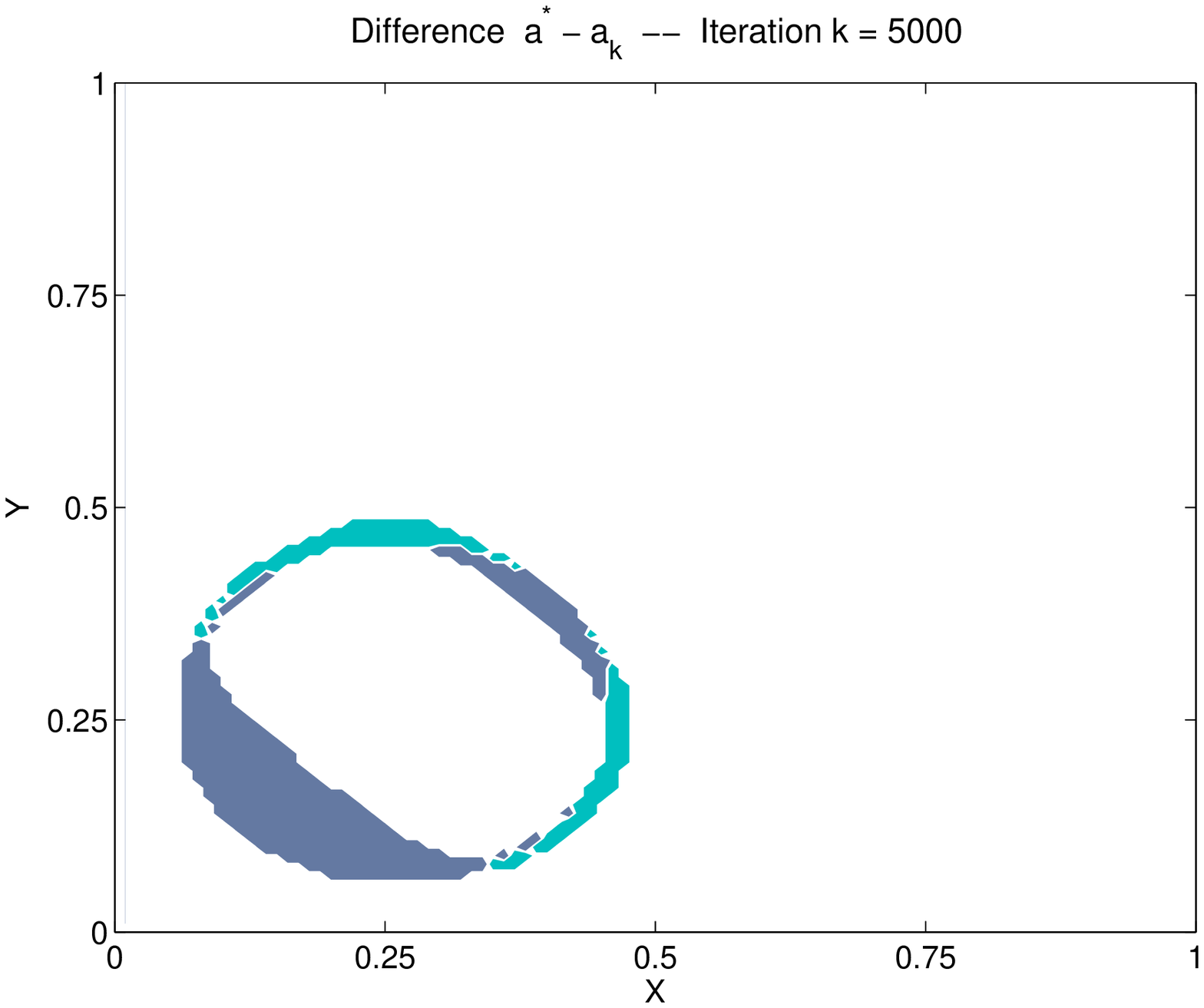} } 
\centerline{\hfil (a) \hskip5.5cm (b) \hskip5.5cm (c) \hfil}
\centerline{
\includegraphics[width=5.7cm]{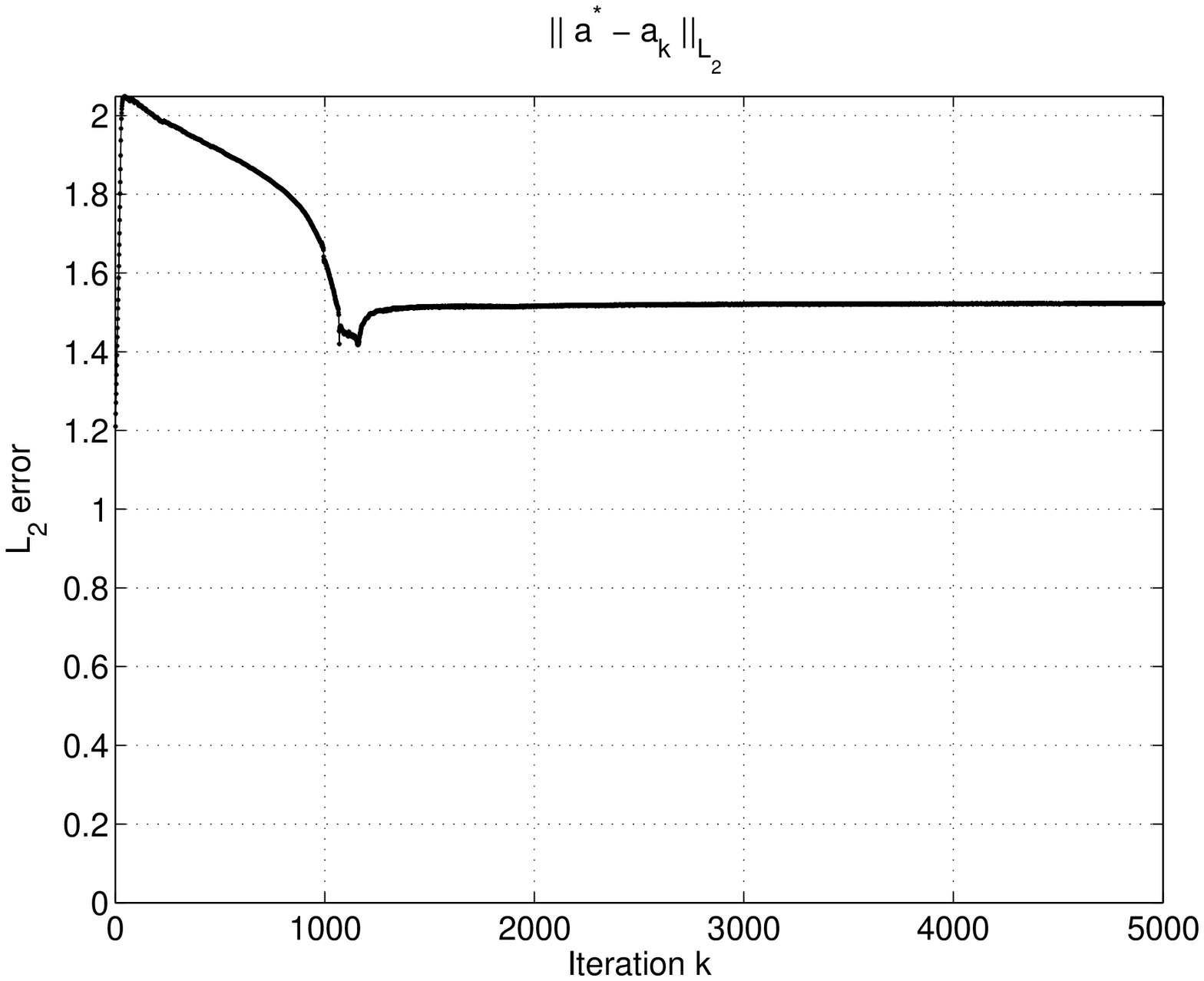}      
\includegraphics[width=5.5cm]{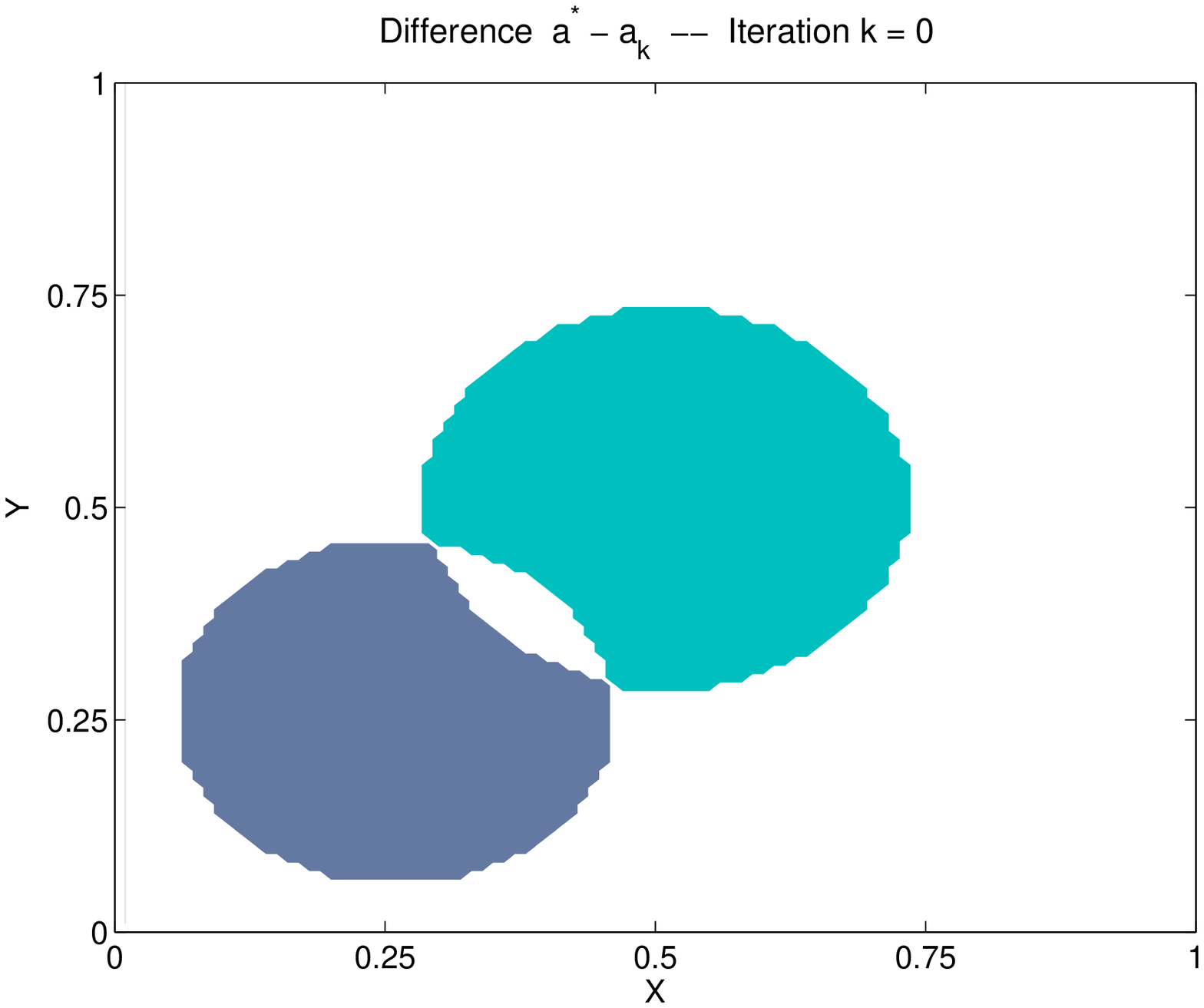}      
\includegraphics[width=5.5cm]{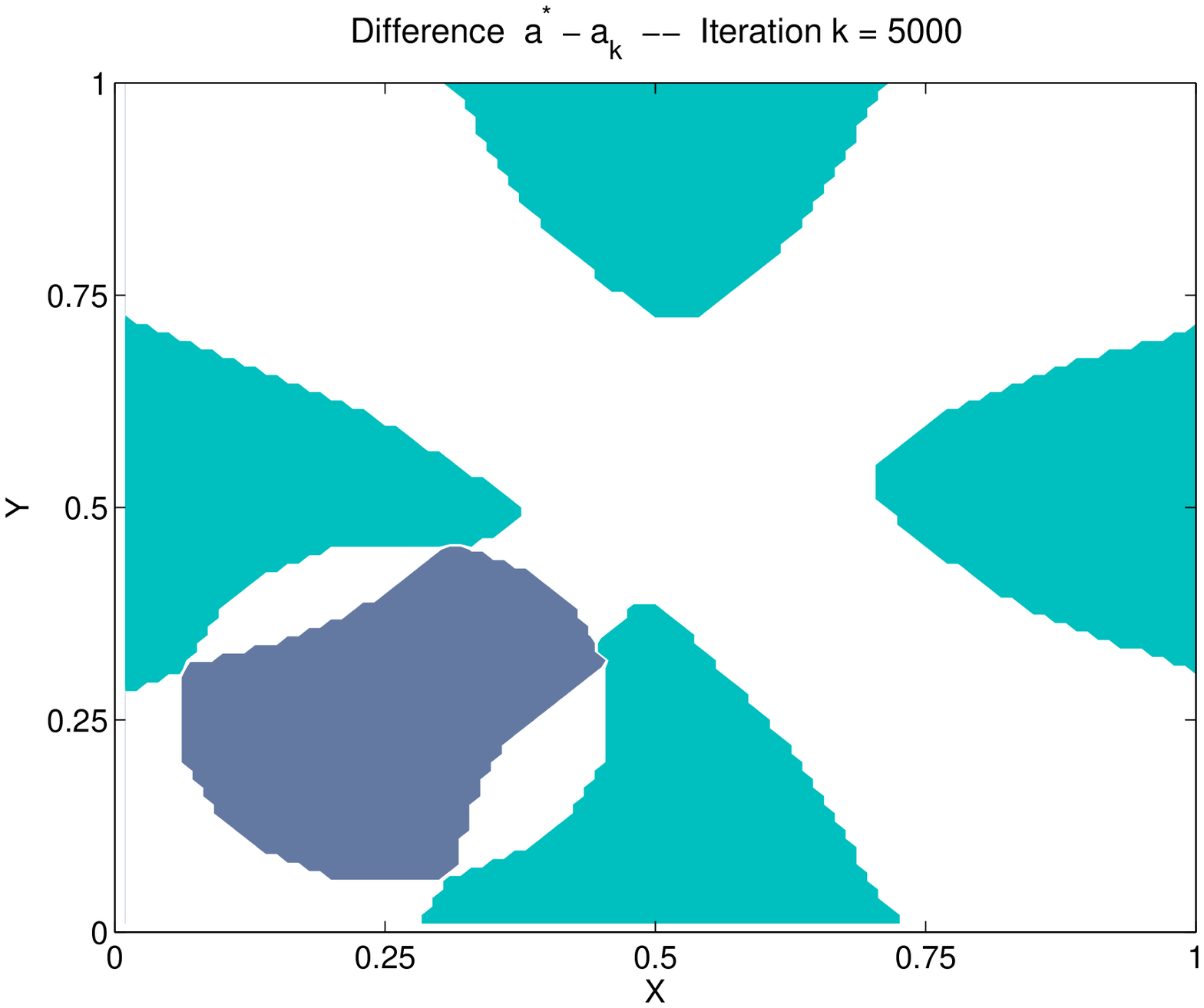} }  
\centerline{\hfil (d) \hskip5.5cm (e) \hskip5.5cm (f) \hfil}
\caption{\small Section~\ref{ssec:numer-a}, 2nd example.
{\bf (a)}--{\bf (c)} Identification of $a^*$ from exact knowledge of $c^*$.
{\bf (a)} Evolution of the $L^2$ error.
{\bf (b)} Difference between the initial guess $a_0$ and $a^*$.
{\bf (c)} Difference between $a_{5000}$ and $a^*$.
{\bf (d)}--{\bf (f)} Identification of $a^*$ from partial knowledge of $c^*$.
{\bf (d)} Evolution of the $L^2$ error.
{\bf (e)} Difference between the initial guess $a_0$ and $a^*$.
{\bf (f)} Difference between $a_{5000}$ and $a^*$. }
\label{fig:example-a2}
\end{figure}

As initial guess for the level set method we choose distinct piecewise constant functions
$a_0$, whose supports are shown in Figure~\ref{fig:separate-ident}~(c). Analogous as
in Section~\ref{ssec:numer-c}, the constant values of the exact solution $a^*$
are assumed to be known, as well as the exact absorption coefficient $c^*$.
Moreover, exact data are used for the reconstruction (i.e., $\delta = 0$) and the scaling factors $\beta_j=0$.

This time, the inverse problem reduces to a shape identification problem for the
diffusion coefficient.
Once again we plot, for each initial guess $a_0$, a corresponding number of steps (see
Figure~\ref{fig:separate-ident}~(c)). It stands for the number of iterations needed to
compute an approximation of $a^*$ (starting from the corresponding $a_0$) with a precision
of $10^{-2}$ in the $L^2$-norm.

This experiment allow us to determine the computational effort necessary for the
reconstruction of $a^*$ with respect to distinct choices of $a_0$.
{\bf The identification problem for the diffusion coefficient is known to be
exponentially ill-posed} \cite{EngHanNeu96, Isa06}.
This fact is in agreement with the values plotted in Figure~\ref{fig:separate-ident}~(c),
meaning that the number of iterations necessary to achieve a good quality reconstruction
{\bf does strongly oscillate} with the initial guess.

We conduct yet another experiment for identifying only the diffusion coefficient.
This time, we assume the exact solution of 
problem \eqref{eq:1}--\eqref{eq:2} to be
given by the coefficient pair $(a^*,c^*)$ in Figure~\ref{fig:separate-ident}~(d).
The setup of the inverse problem remains the same (domain, available data, parameter
to output operator, etc.). \\
On the first run of the algorithm, see Figure~\ref{fig:example-a2}~(a)--(c), the
absorption coefficient $c^*$ is assumed to be exactly known. In this situation, the level set method is
able to identify the diffusion coefficient (see Figure~\ref{fig:example-a2}~(a)
for the evolution of the iteration error), and the iteration stagnates after that.
The corresponding differences between the exact solution $a^*$ and the initial guess $a_0$ and between the exact solution $a^*$ and the final
iterate $a_{5000}$ are plotted in pictures~(b) and~(c) respectively. \\
On the second run, see Figure~\ref{fig:example-a2}~(d)--(f), we considered
the approximation $\c(x) \equiv 1$ for the absorption
coefficient $c^*$ and iterate to recover $a^*$. In this case, the level set method
is no longer able to identify the diffusion coefficient. The
iteration once again stagnates, but this time at some configuration
far from the exact solution (see Figure~\ref{fig:example-a2}~(d) for
the evolution of the $L^2$ error). The corresponding differences for the initial guess $a_0$ and for the final iterate
$a_{5000}$ are plotted in pictures~(e) and~(f) respectively.

\subsection{Identification of both diffusion and absorption coefficients}
\label{ssec:numer-ac}

\begin{figure}[b!]
\centerline{
\includegraphics[width=6cm]{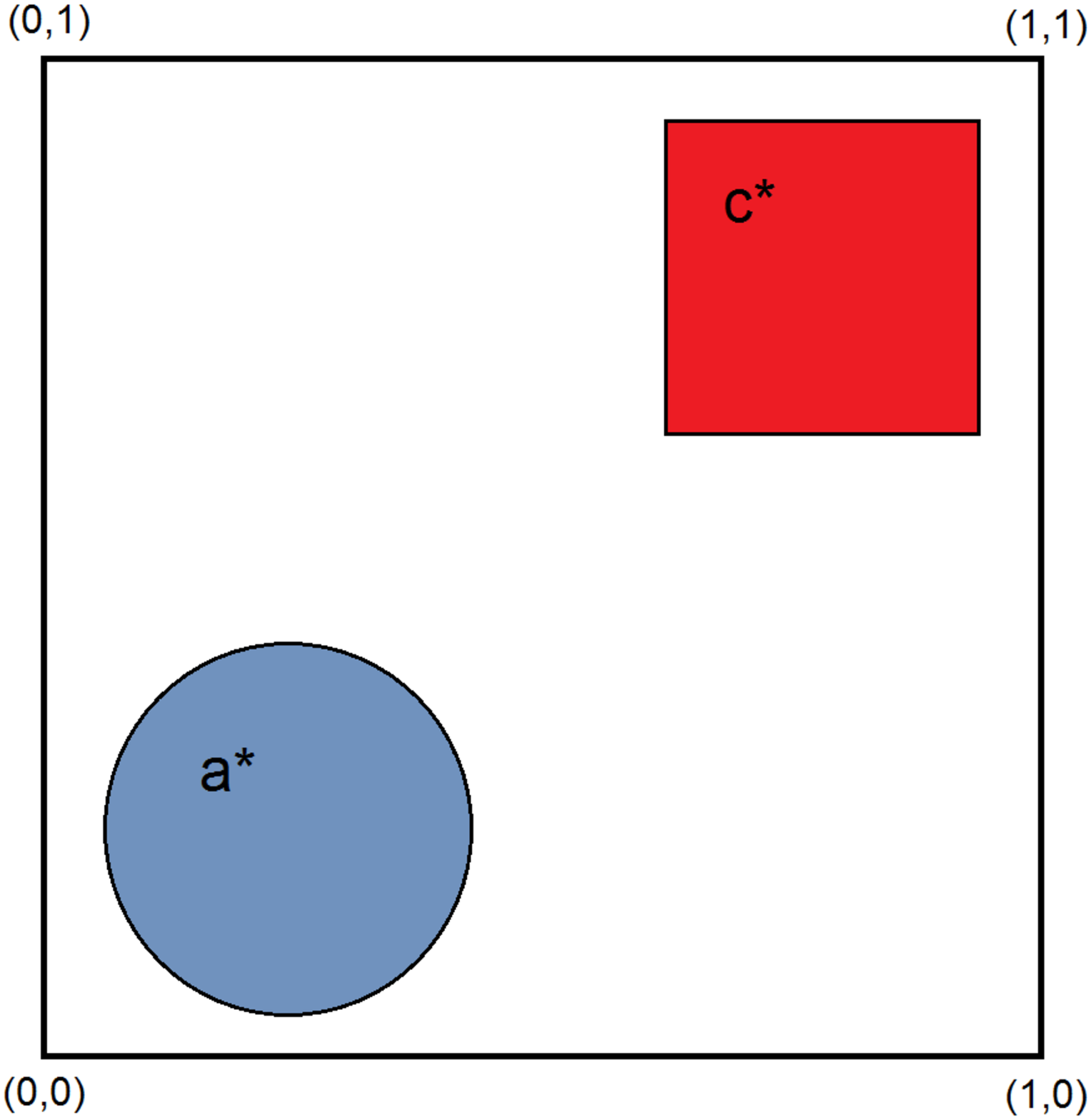}
\includegraphics[width=6cm]{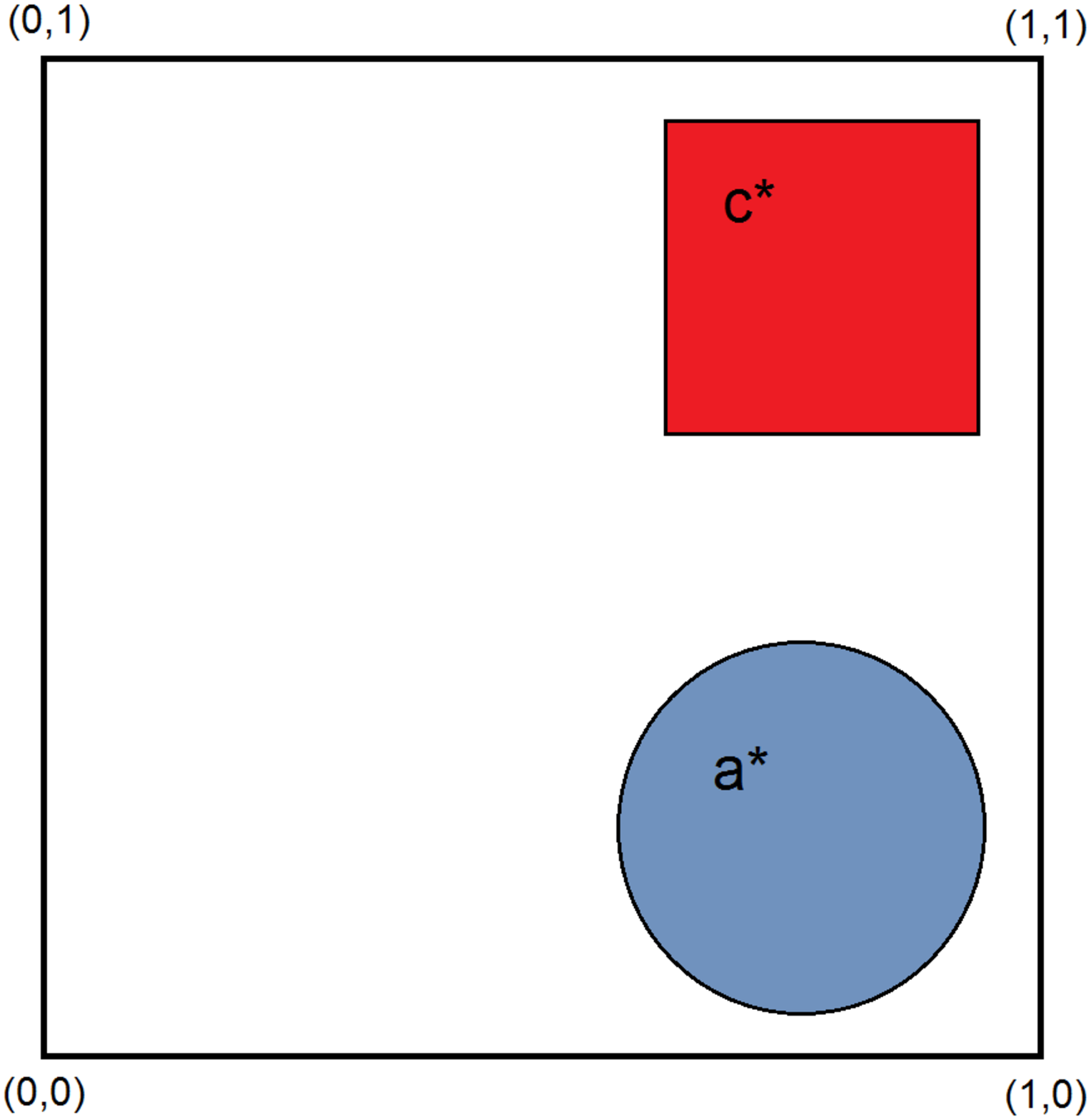}
\includegraphics[width=6cm]{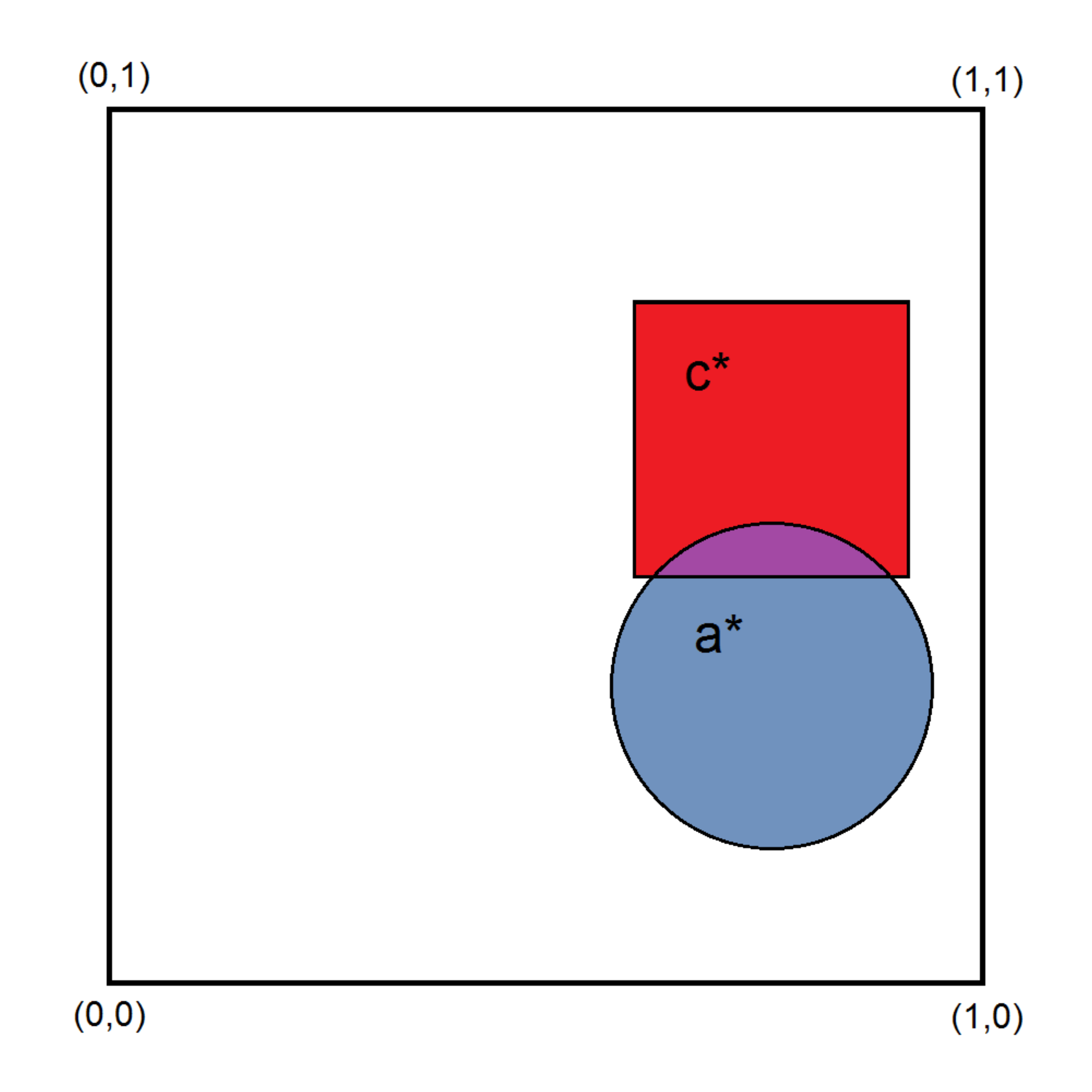} 
}
\centerline{\hfil (a) \hskip5.5cm (b) \hskip5.5cm (c) \hfil}
\caption{\small Section~\ref{ssec:numer-ac}.
{\bf (a)}--{\bf (c)} Exact solution for the first, second and third example. }
\label{fig:example-ac}
\end{figure}

In this last set of experiments we consider the level set algorithm for the simultaneous
identification of the coefficient pair $(a,c)$ 
in \eqref{eq:1}--\eqref{eq:2}. Three
examples are considered and the corresponding exact solutions are shown in Figure~\ref{fig:example-ac}.
The setup of the inverse problem is the same as in Sections~\ref{ssec:numer-c}
and~\ref{ssec:numer-a} (domain, available data, parameter to output operator, \dots).

In the first example, the solution pair $(a^*,c^*)$ is the one shown in Figure~%
\ref{fig:example-ac}~(a). In order to devise an efficient iteration strategy for
the simultaneous identification of both coefficients, we must take some facts into
account: \\
{\bf F1)} \ From the 2nd example in Section~\ref{ssec:numer-c}, we have learned that the
method for identifying $c^*$ performs well, even if a good approximation for
$a^*$ is not known (see Figure~\ref{fig:example-c2}~(d)--(f)). \\
{\bf F2)} \ On the other hand, from the 2nd example in Section~\ref{ssec:numer-a}, we have
learned that the level set method for identifying $a^*$ performs well if a good approximation
for $c^*$ is available, but may generate a sequence $a_k$ that does not approximate
$a^*$ if a good approximation to $c^*$ is not known. \\
{\bf F3)} \ In the first run of the level set algorithm for the simultaneous identification
of $(a^*,c^*)$ we updated both coefficients $(a_k,c_k)$ in every step and observed
that the iteration error $\| c_k - c^* \|$ decreases from the very first iteration.
However, the iteration error $\| a_k - a^* \|$ only starts improving when
$\| c_k - c^* \|$ is sufficiently small.

\begin{figure}[b!]
\centerline{
\includegraphics[width=4.3cm] {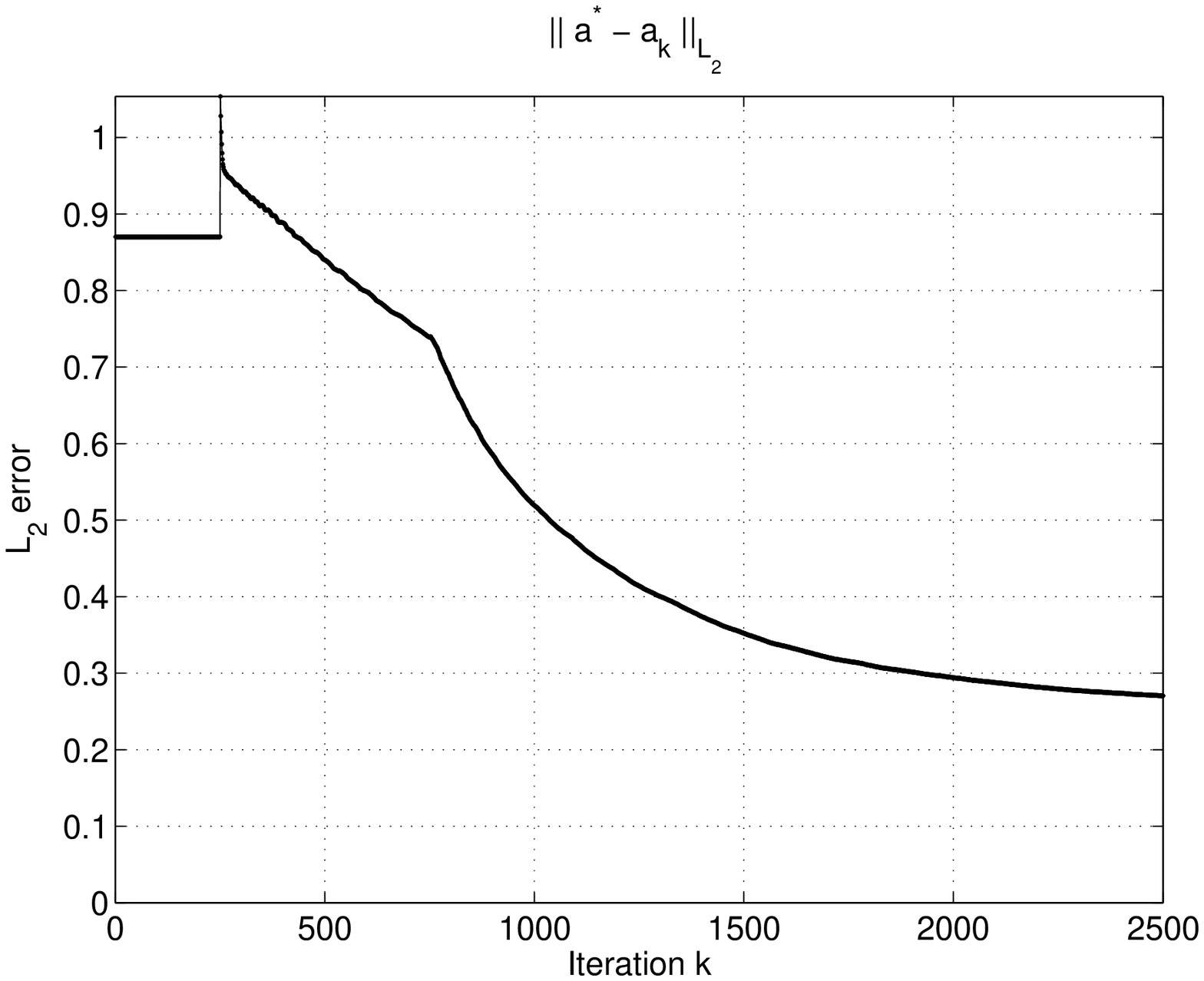}        
\includegraphics[width=4.2cm] {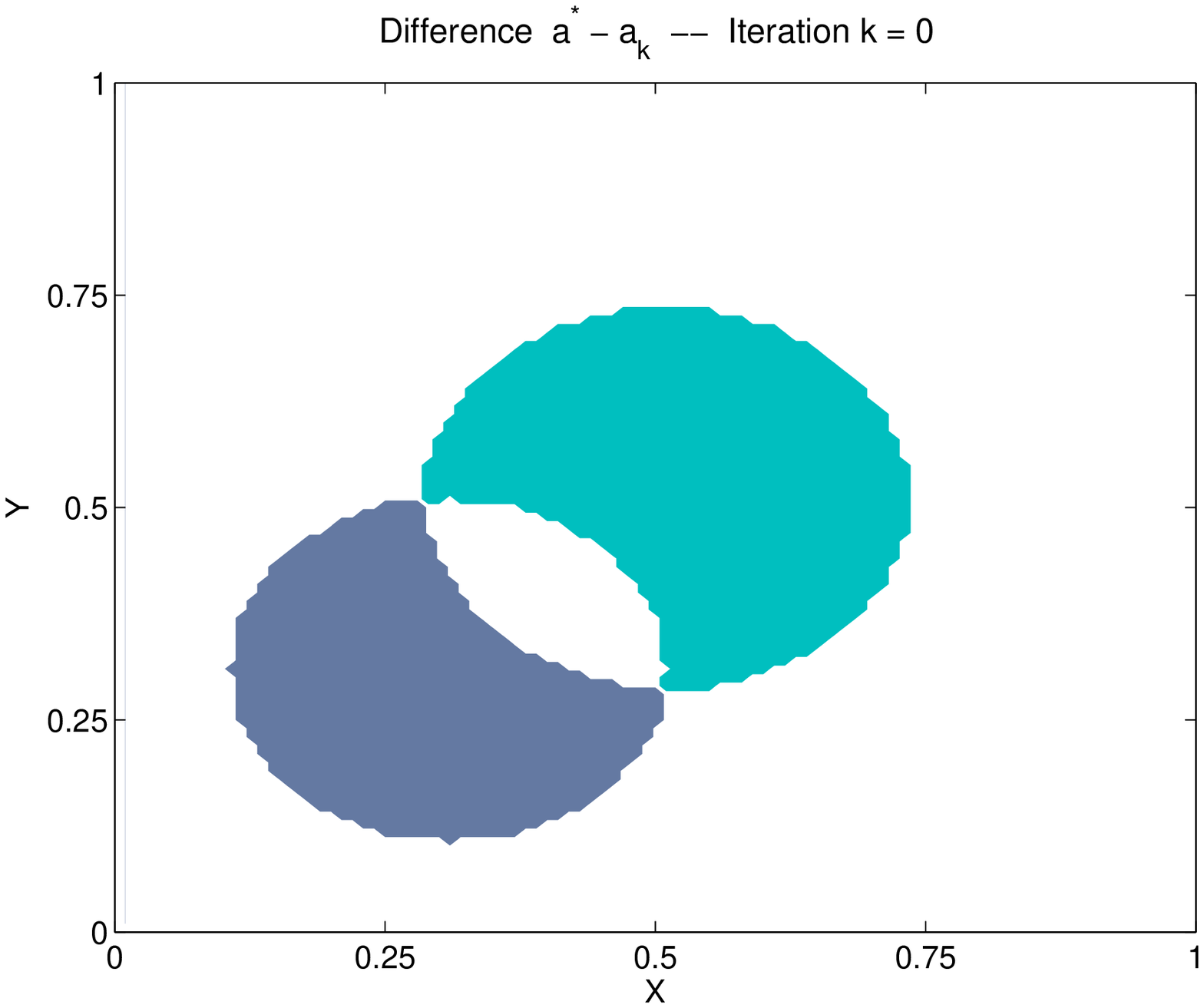} 
\includegraphics[width=4.2cm] {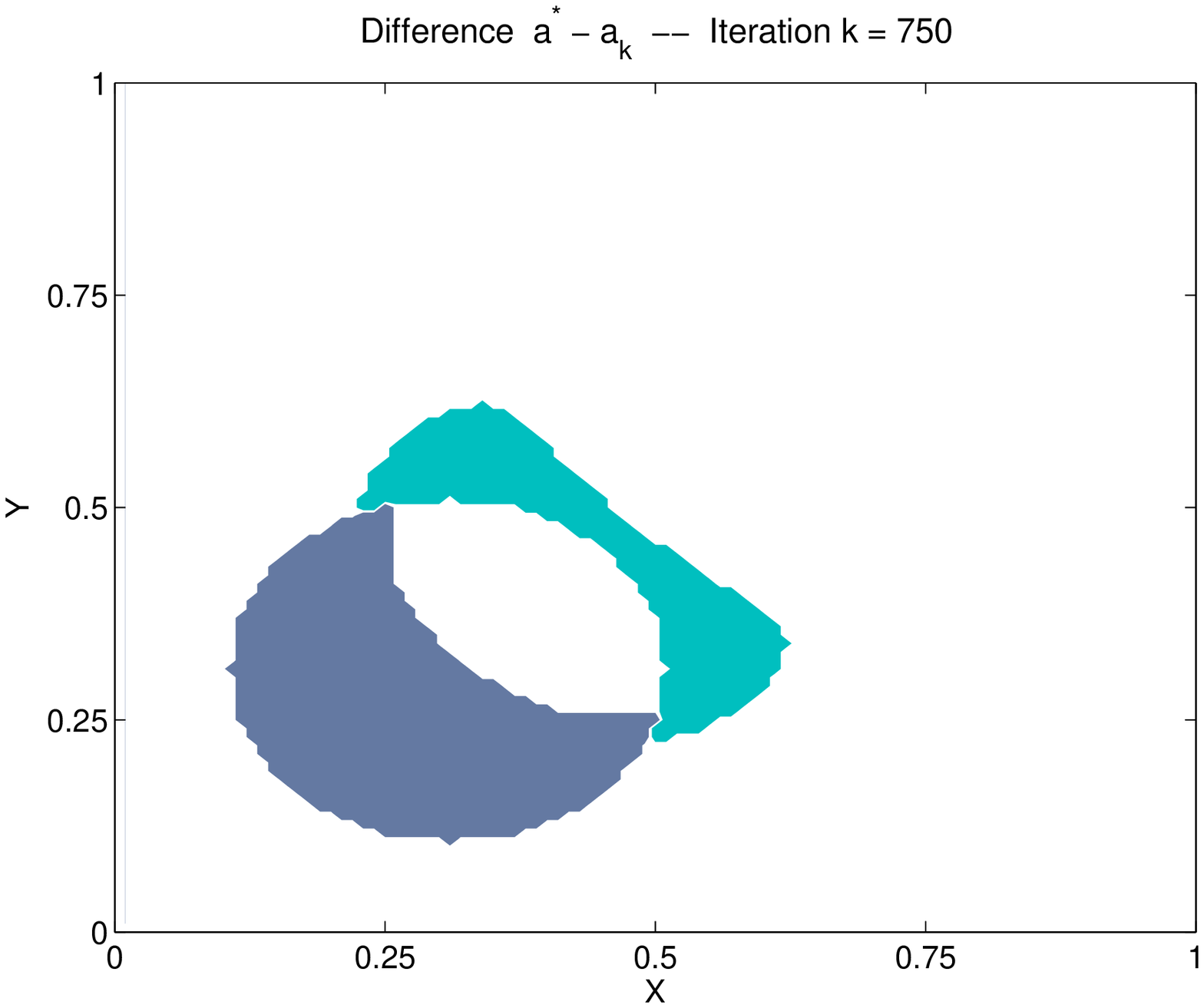}  
\includegraphics[width=4.2cm] {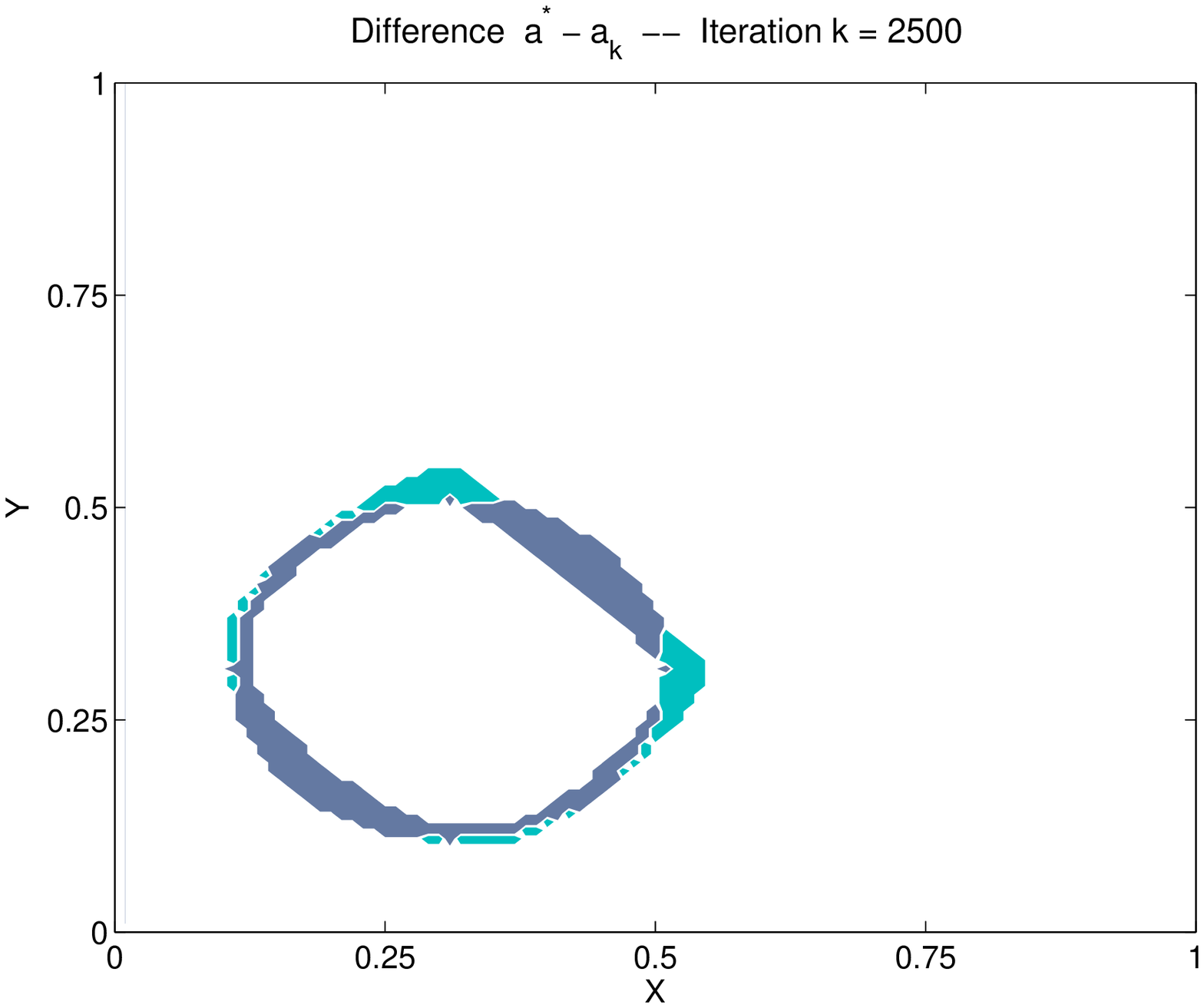} }   
\centerline{\hfil (a) \hskip4cm (b) \hskip4cm (c) \hskip4cm (d) \hfil}
\centerline{
\includegraphics[width=4.3cm] {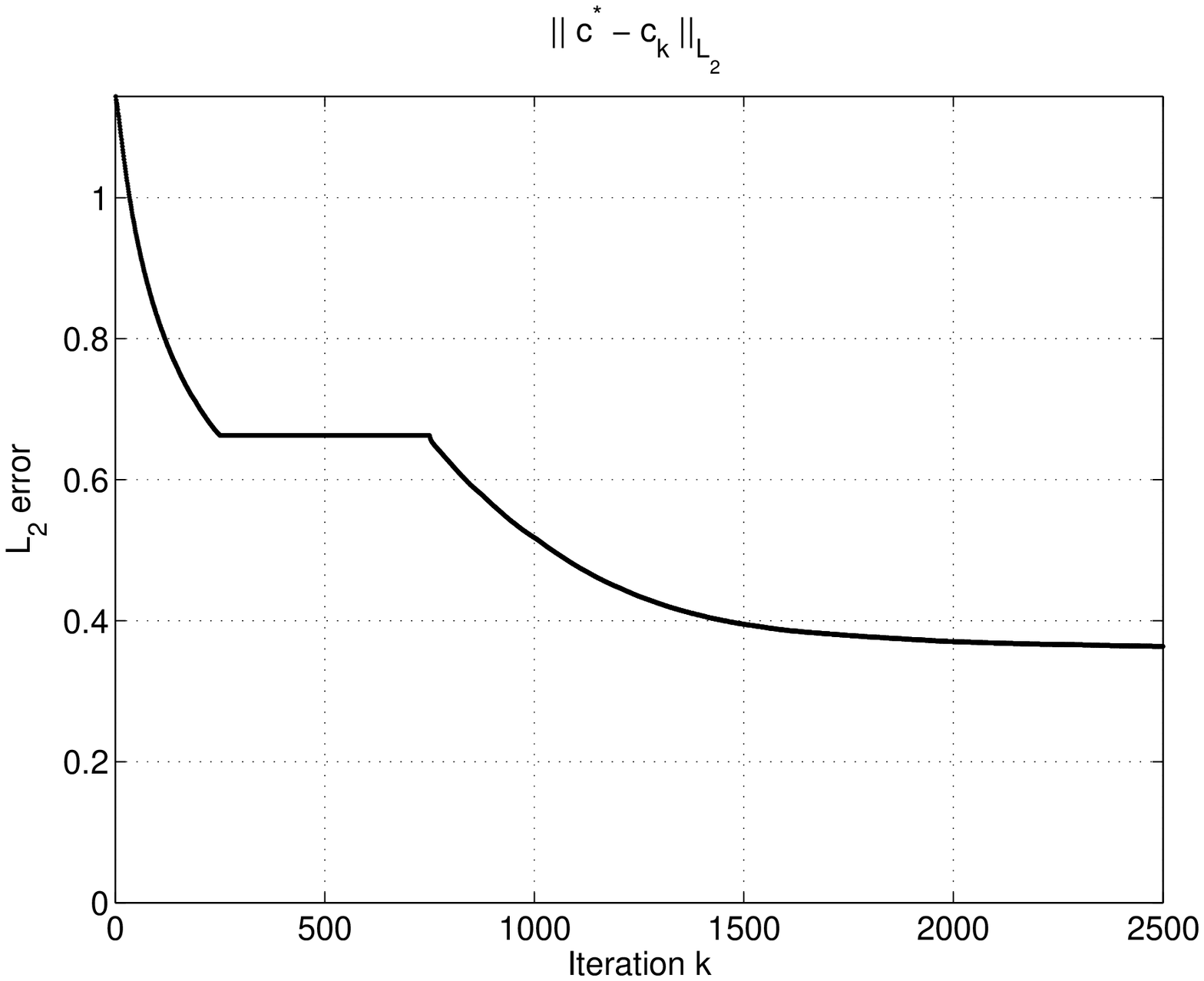}  
\includegraphics[width=4.2cm] {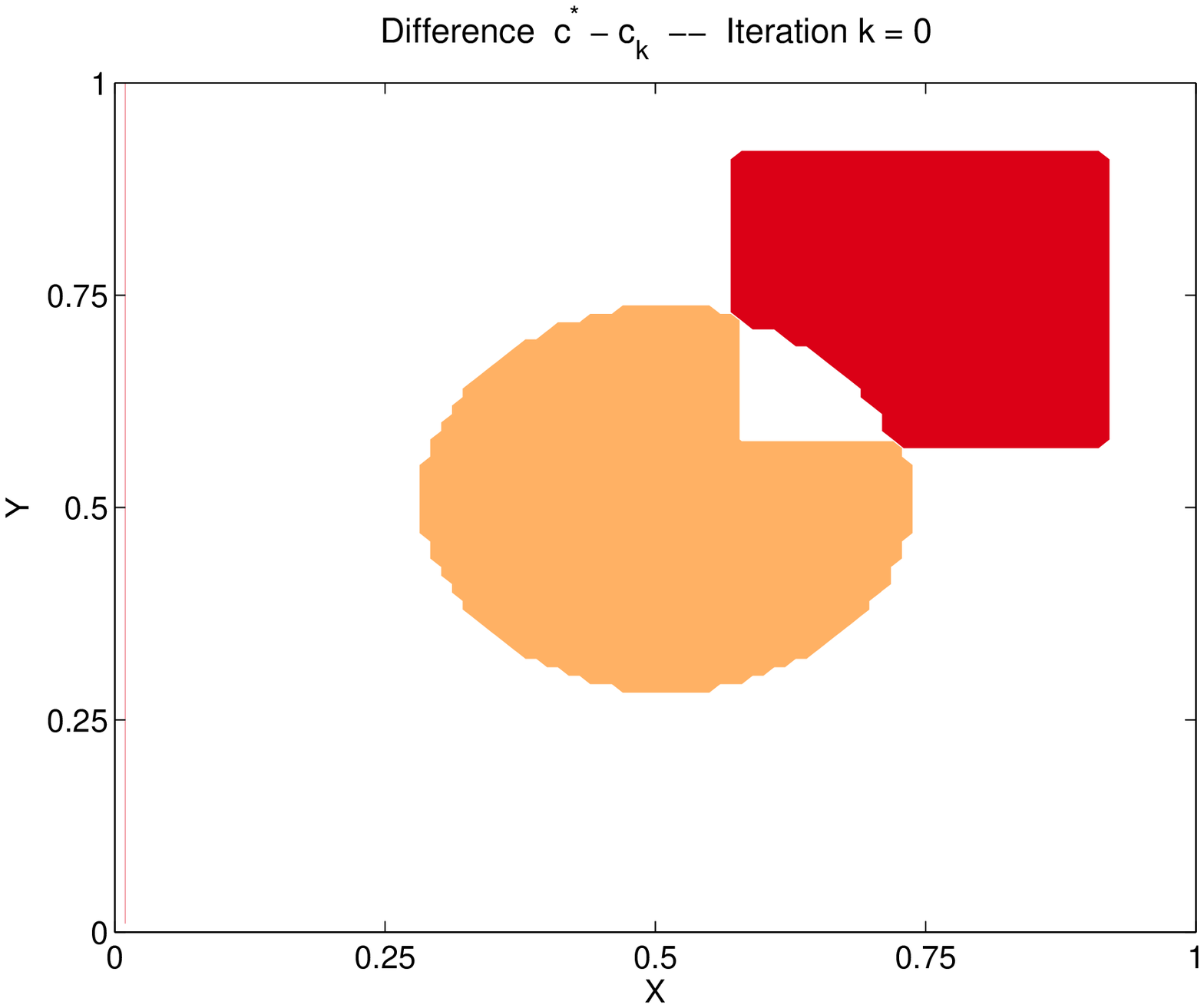} 
\includegraphics[width=4.2cm] {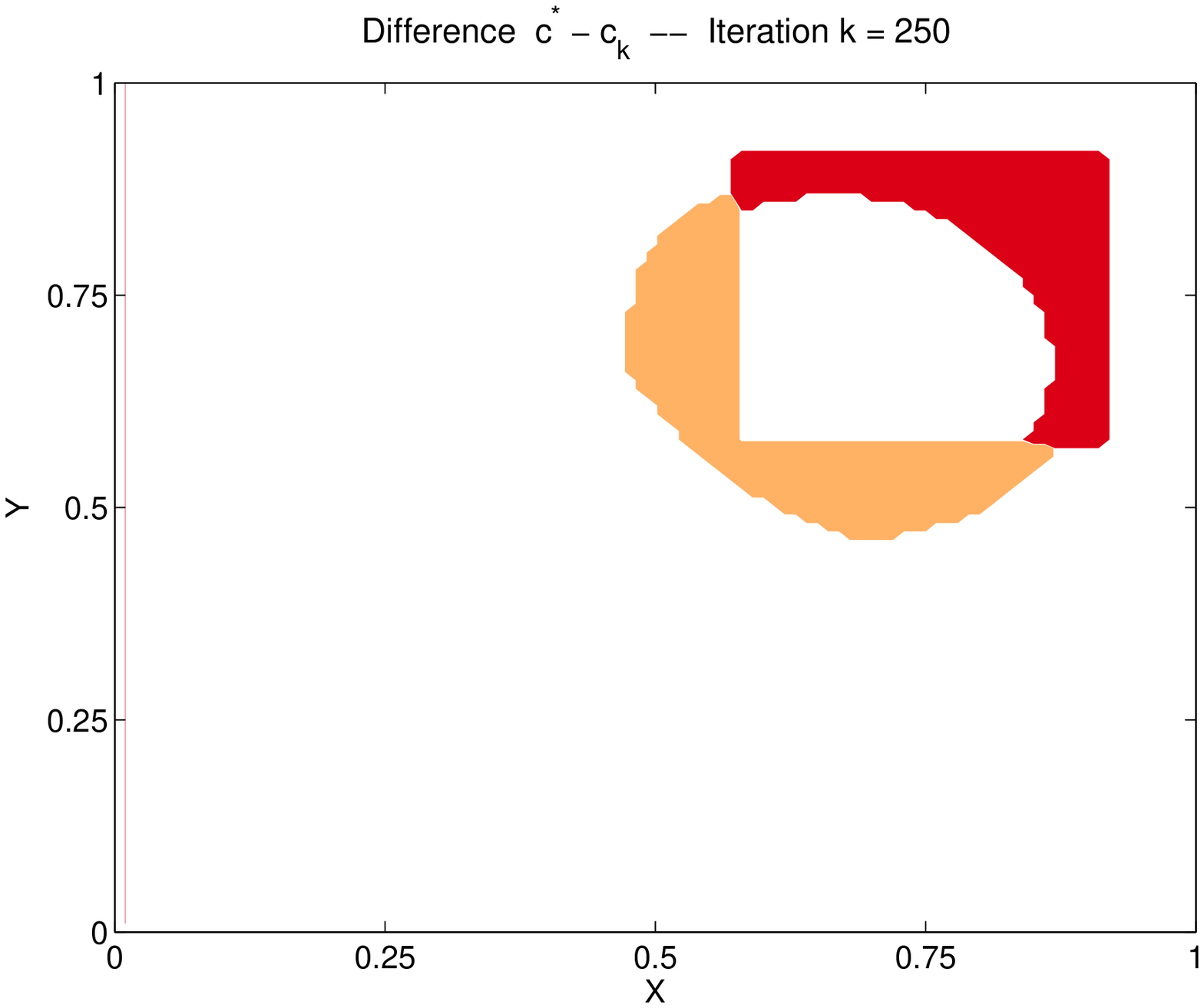}  
\includegraphics[width=4.2cm] {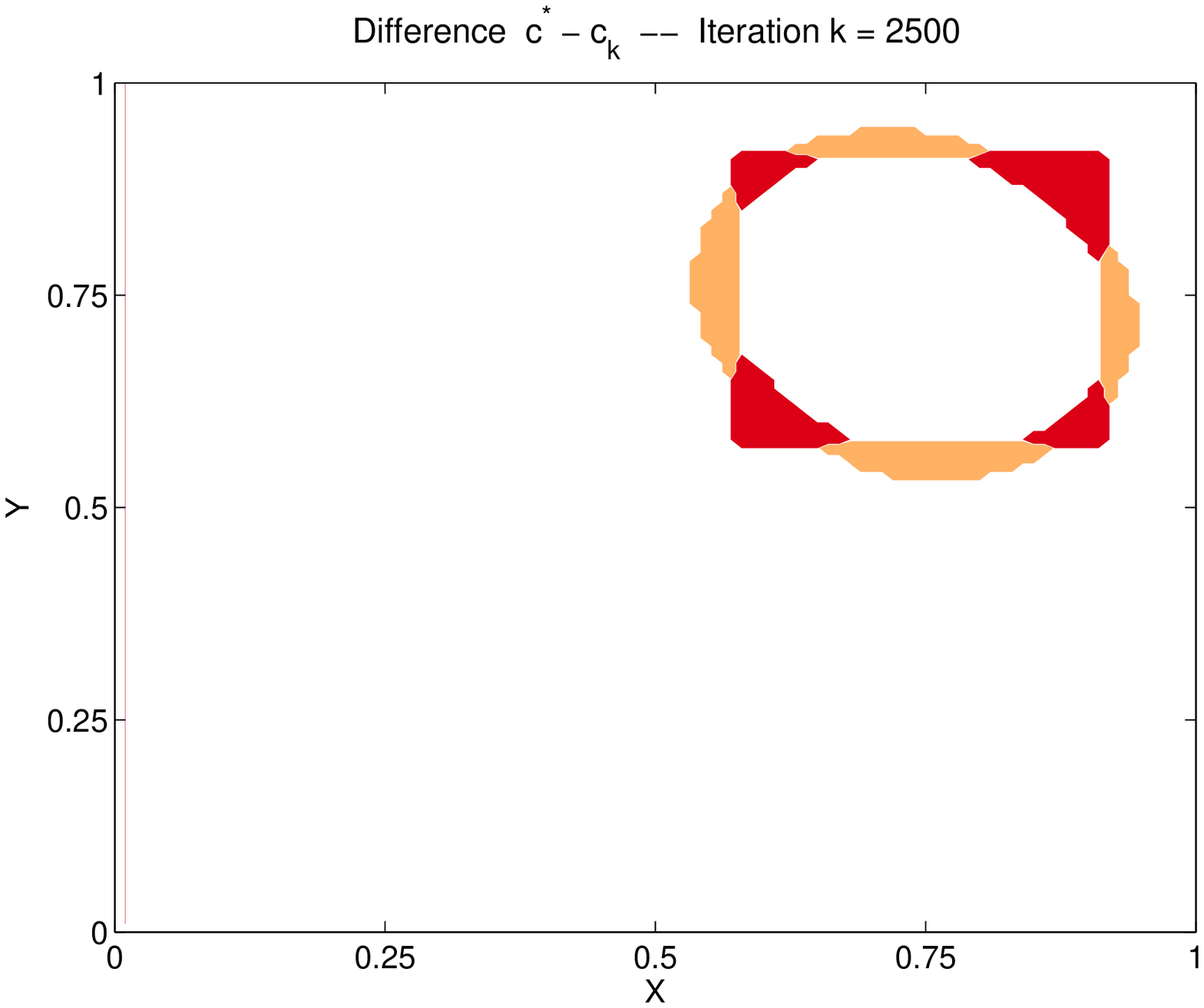} } 
\centerline{\hfil (e) \hskip4cm (f) \hskip4cm (g) \hskip4cm (h) \hfil}
\caption{\small Section~\ref{ssec:numer-ac}, 1st example.
{\bf (a)}--{\bf (d)} Iterative reconstruction of $a^*$.
{\bf (a)} Evolution of the $L^2$ error.
{\bf (b)} Difference $a_0 - a^*$.
{\bf (c)} Difference $a_{750} - a^*$.
{\bf (d)} Difference $a_{2500} - a^*$.
{\bf (e)}--{\bf (h)} Iterative reconstruction of $c^*$.
{\bf (e)} Evolution of the $L^2$ error.
{\bf (f)} Difference $c_0 - c^*$.
{\bf (g)} Difference $c_{250} - c^*$.
{\bf (h)} Difference $c_{2500} - c^*$. }
\label{fig:example-ac1}
\end{figure}

Thus, in order to save computational effort, we adopted the strategy to ``freeze'' the
coefficient $a_k(x) = a_0(x) \equiv 1$ during the first iterations, and to iterate
only with respect to $c_k$. We follow this strategy until the sequence $c_k$ stagnates
(this is an indication that the iteration error $\|c_k - c^*\|$ is small).
In Figure~\ref{fig:example-ac1}~(a) and~(e) this stage corresponds to the first
$k_1=250$ iterative steps (notice that $\|a_k - a^*\|$ remains constant for
$k = 0, \dots, k_1$, while the difference $c_{k_1} - c^*$ is plotted in~(g)).

After this first iteration stage, we freeze $c_k = c_{k_1}$ and iterate only with respect to
$a_k$. This characterizes the second stage of the method. A natural question at this
point would be: Why not to iterate with respect to both $(a_k,c_k)$ for $k \geq k_1$? \
We tried to proceed in this way, but what we observed is that: as long as $\|a_k - a^*\|$
does not significantly improve, the iterates $c_k$ stagnate with $\|c_k - c_{k_1}\|$
almost constant. \\
This second stage of the iteration can be observed in Figure~\ref{fig:example-ac1}~(a)
and~(e). Notice that $\|a_k - a^*\|$ decreases significantly, while $\|c_k - c^*\|$
remains constant for $k = k_1, \dots, k_2=750$ (the difference $a_{k_2} - a^*$ is
plotted in Figure~\ref{fig:example-ac1}~(c)).

After the conclusion of the second iteration stage, the pair
$(a_k,c_k)$ is already a good approximation for $(a^*,c^*)$ (see
Figure~\ref{fig:example-ac1}~(c) and~(g)). As a matter of fact, this
approximation is so good that, proceeding with the iteration
simultaneously with respect to both $(a_k,c_k)$, the iteration errors $\|a_k -
a^*\|$ and $\|c_k - c^*\|$ are monotone decreasing. However, having
in mind the convergence of the diffusion coefficient $a_k$ to the
correct solution $a^*$ takes more iterations than the absorption
coefficient $c_k$, in this third stage we decided that each iteration step consists in one iteration with respect to the absorption coefficient
$c_k$ and two iterations with respect to the diffusion coefficient $a_k$ (see
Figure~\ref{fig:example-ac1}~(a) and~(e) for $k \geq k_2$).

The introduction of this 3-stage iteration is motivated by the above mentioned facts
(F1) -- (F3). The calculation of optimal transition indexes $k_1$, $k_2$ between the
three stages is a difficult task.
However, since the degree of ill-posedness of the separate inverse problems for $a$
and $c$ is very distinct from each other %
it's not hard to get approximate values for $k_1$ and $k_2$ that will lead to a
large gain in computational effort by using this 3-stage strategy.
\medskip

The following second and third examples in this section do belong together. The
corresponding exact solutions are shown in Figure~\ref{fig:example-ac}~(b) and~(c)
respectively. Our goal is to investigate how the distance between the supports of
the exact solution pair $(a^*, c^*)$ may interfere with the
quality of the reconstruction of each single coefficient. In the second example
there is a positive distance between the supports, while in the
third example both supports overlap.

\begin{figure}[b!]
\centerline{
\includegraphics[width=4.3cm]{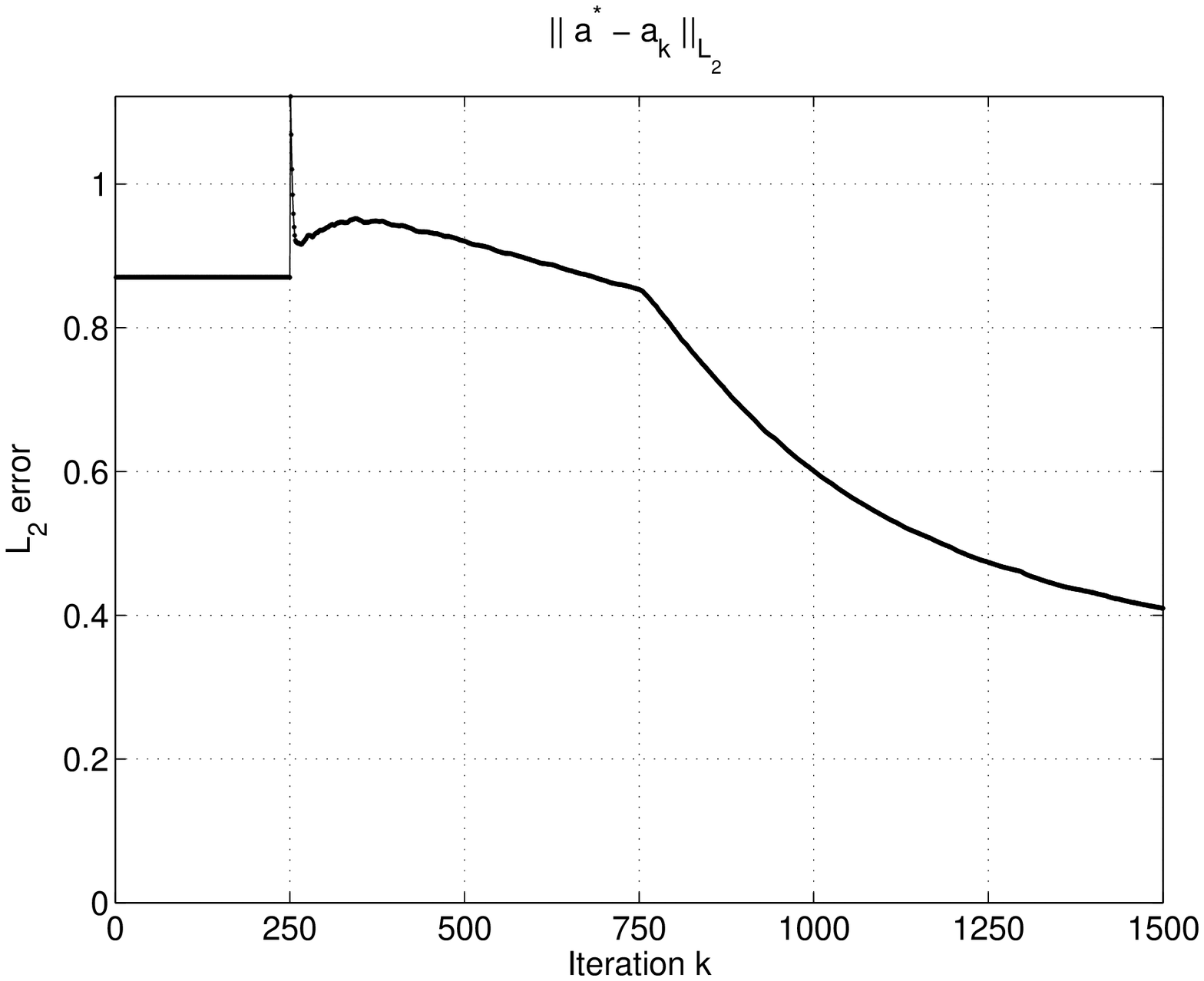} 
\includegraphics[width=4.2cm]{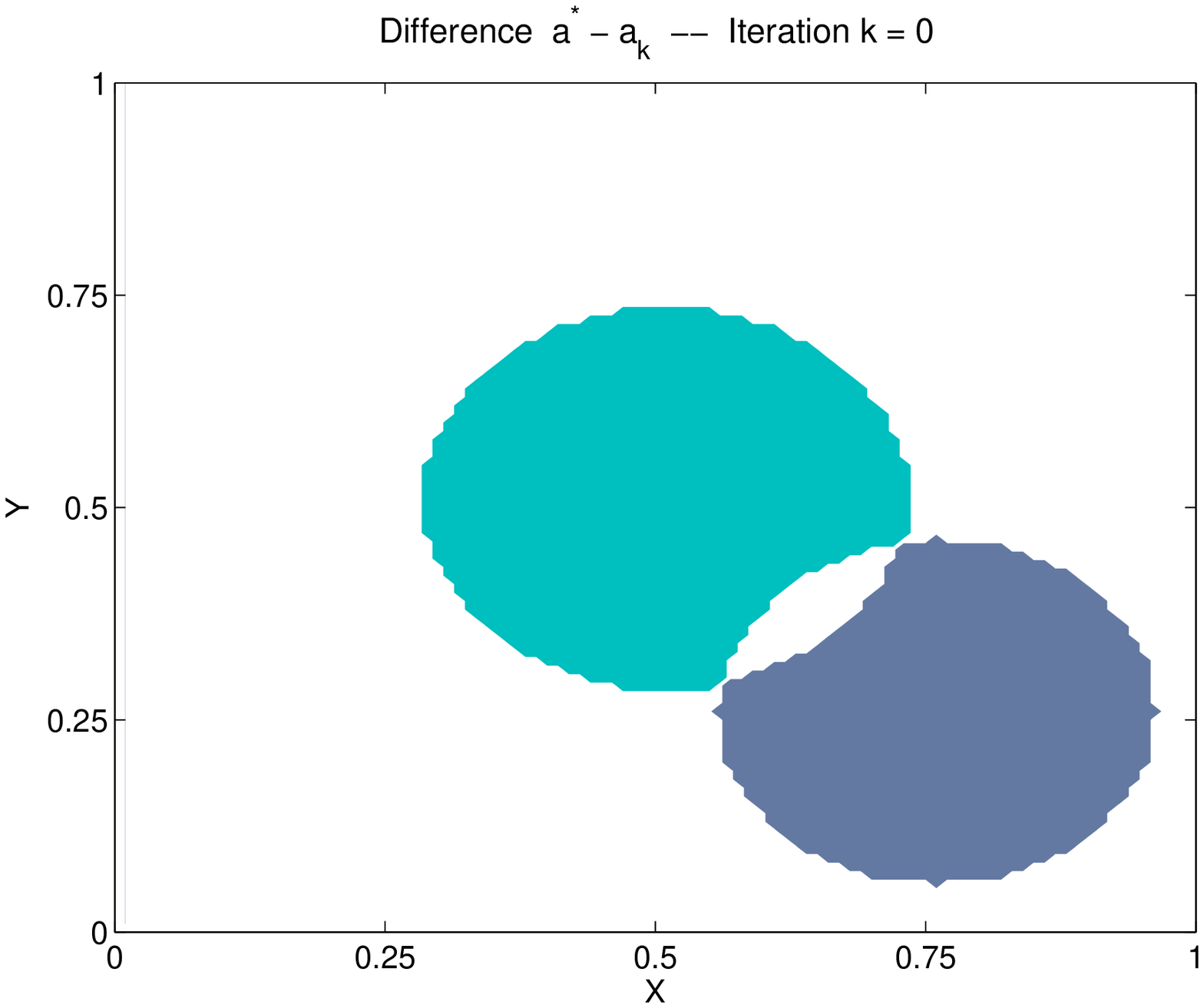} 
\includegraphics[width=4.2cm]{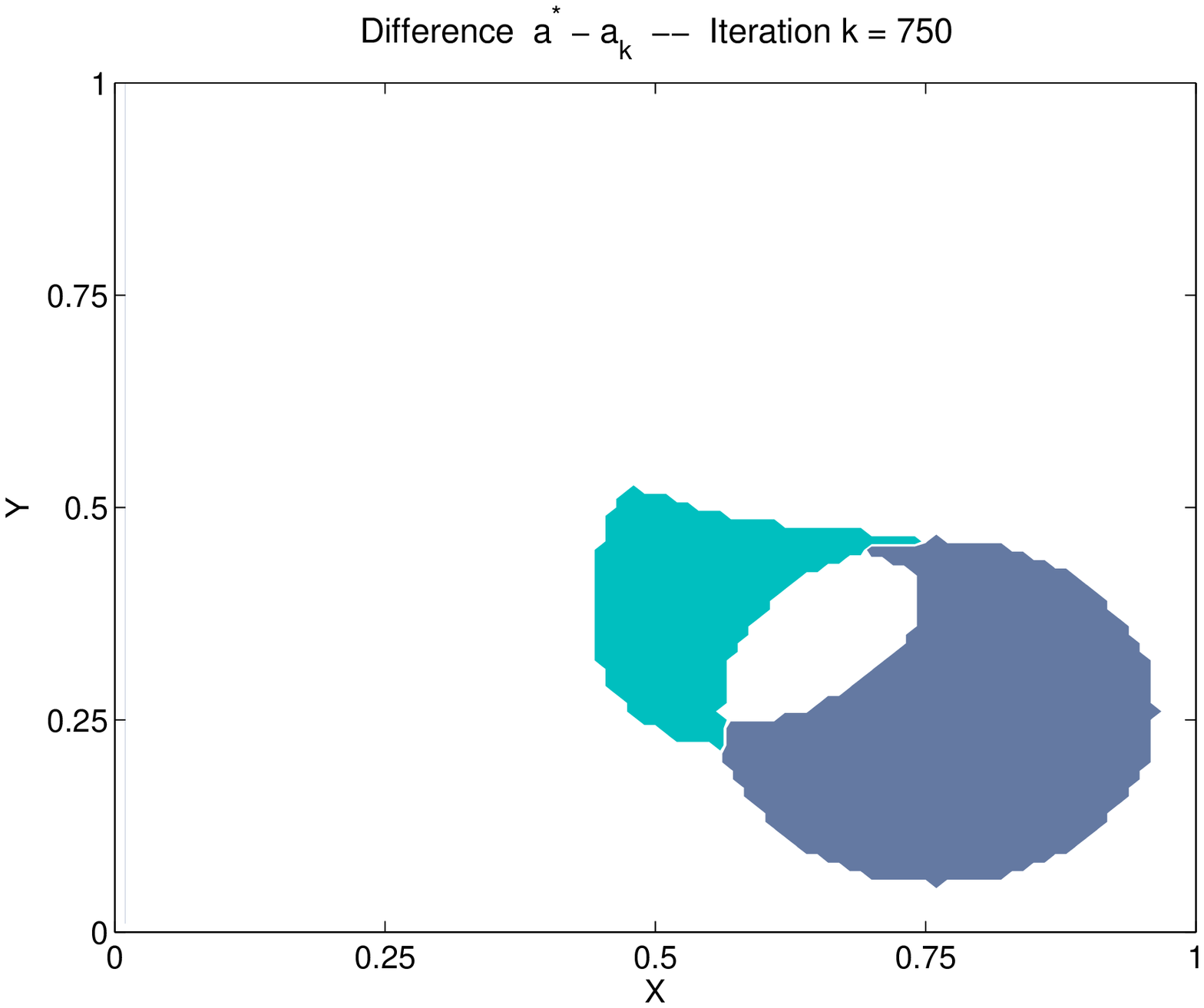} 
\includegraphics[width=4.2cm]{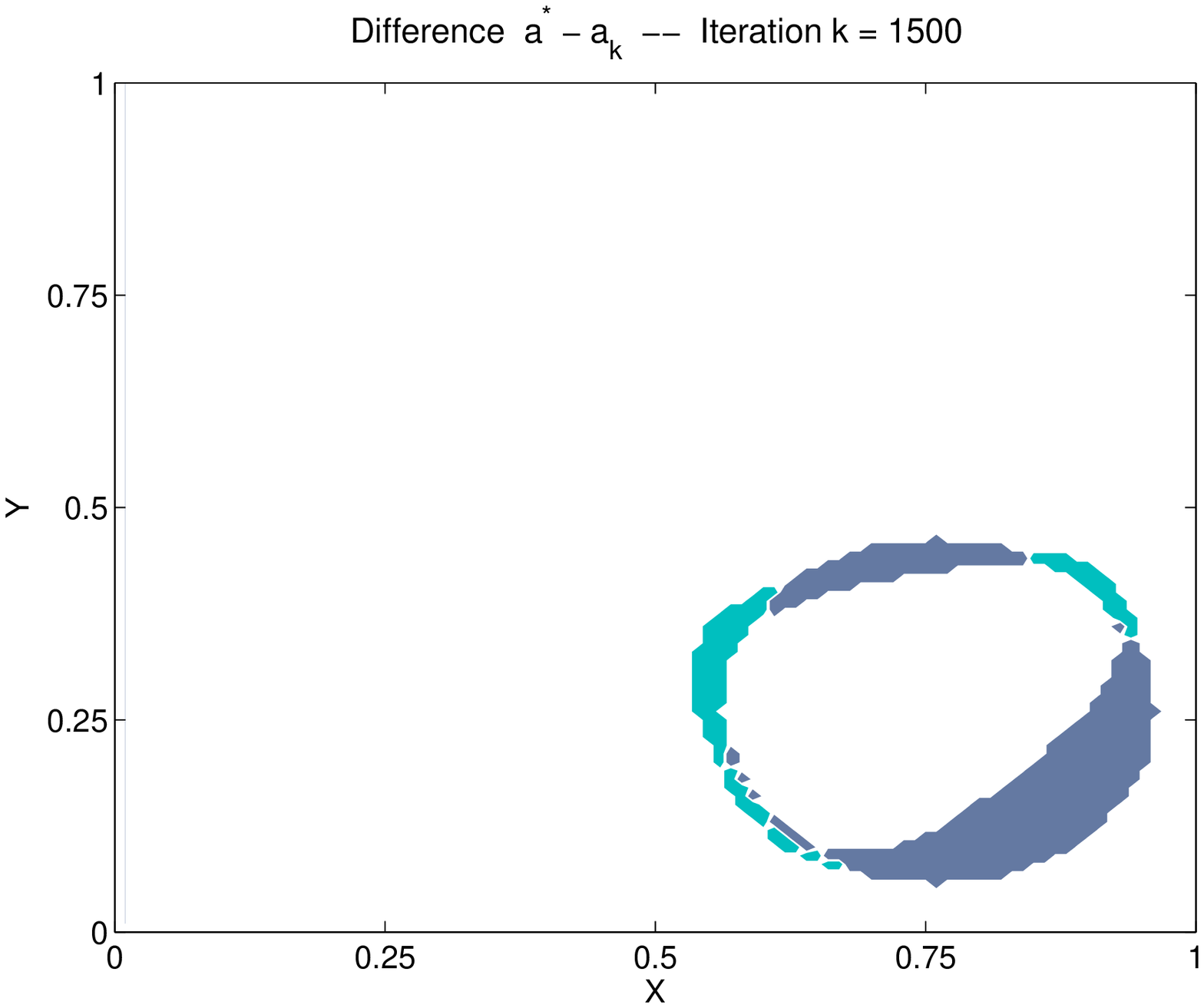}} 
\centerline{\hfil (a) \hskip4cm (b) \hskip4cm (c) \hskip4cm (d) \hfil}
\centerline{
\includegraphics[width=4.3cm]{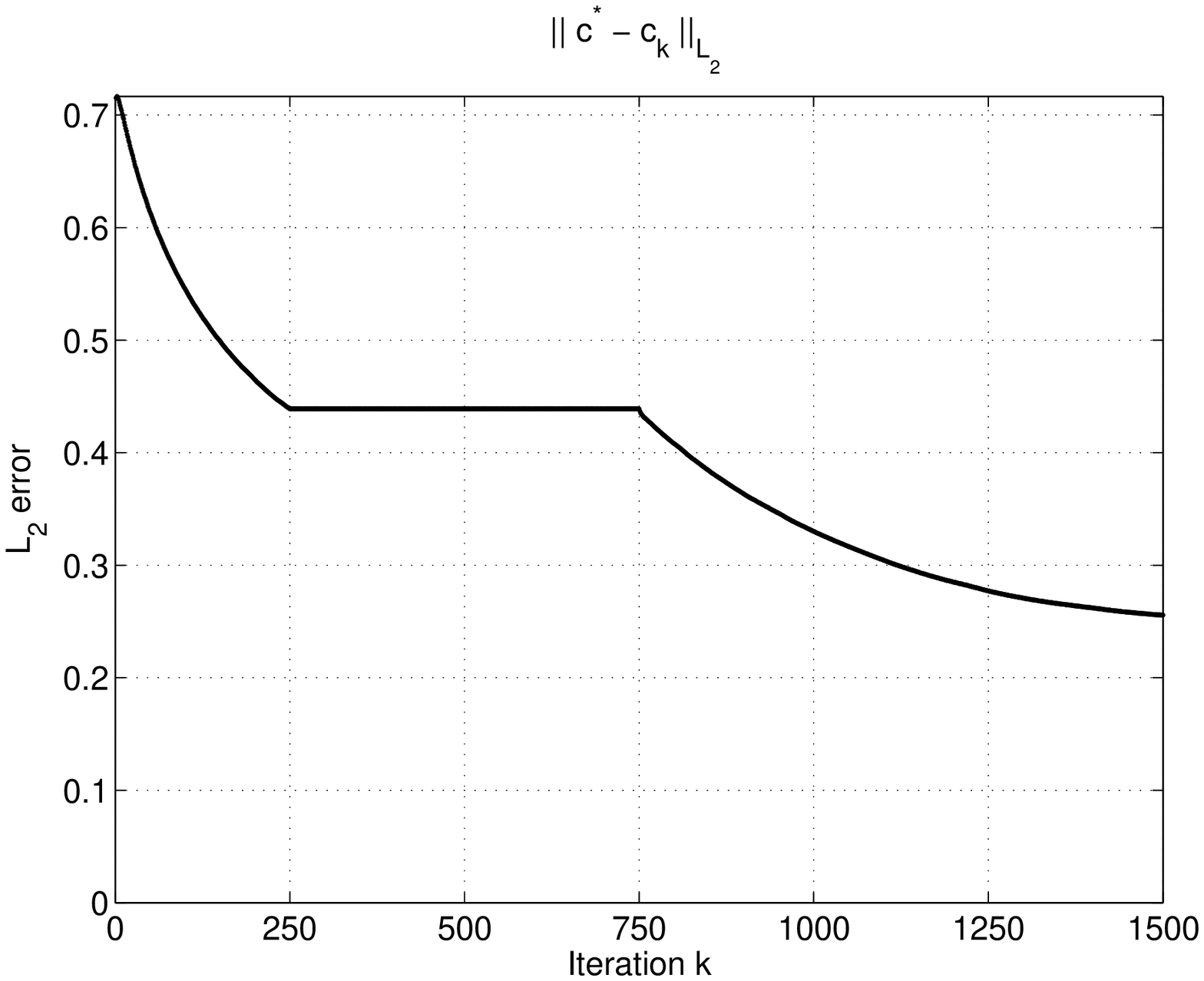}  
\includegraphics[width=4.2cm]{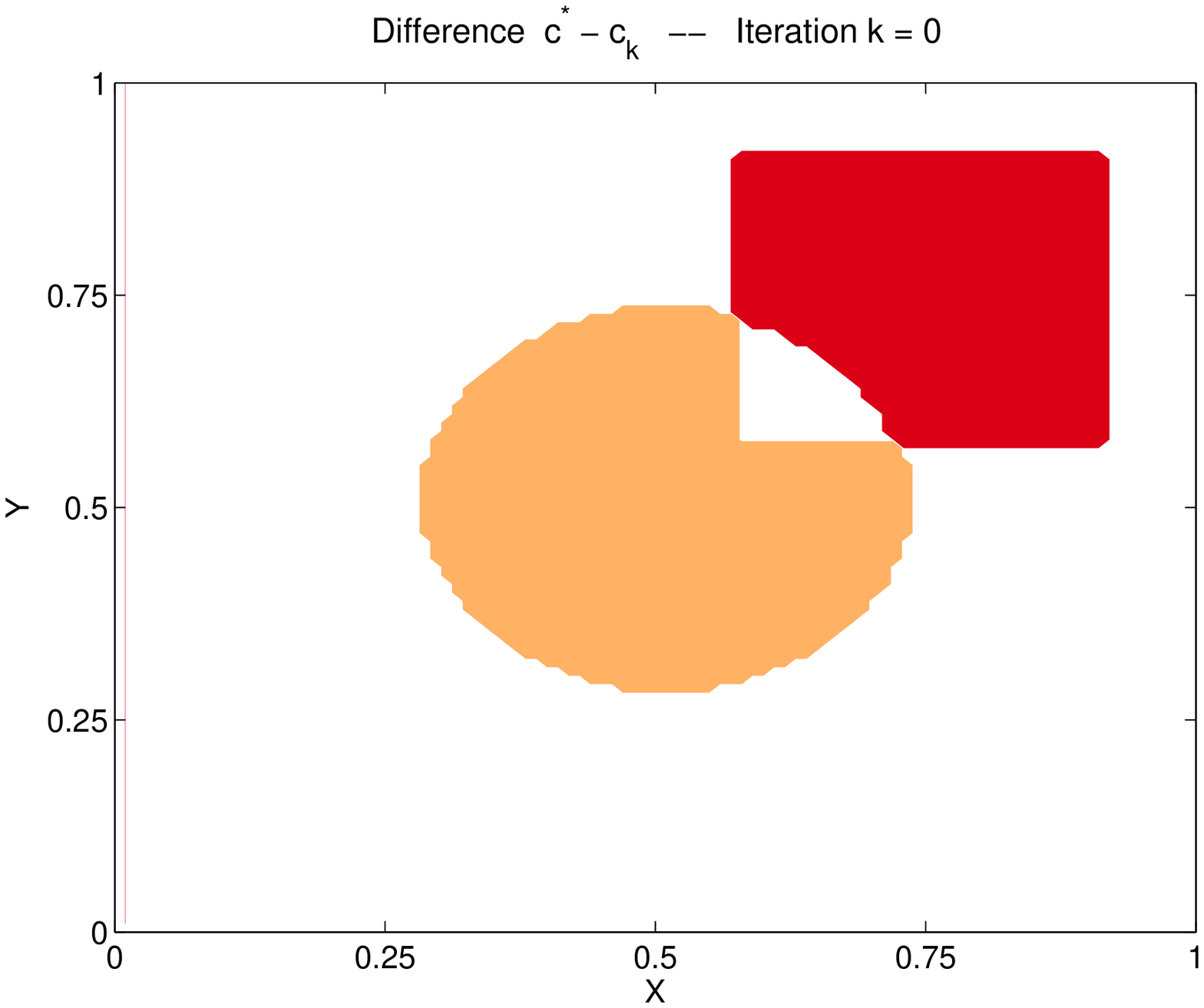}  
\includegraphics[width=4.2cm]{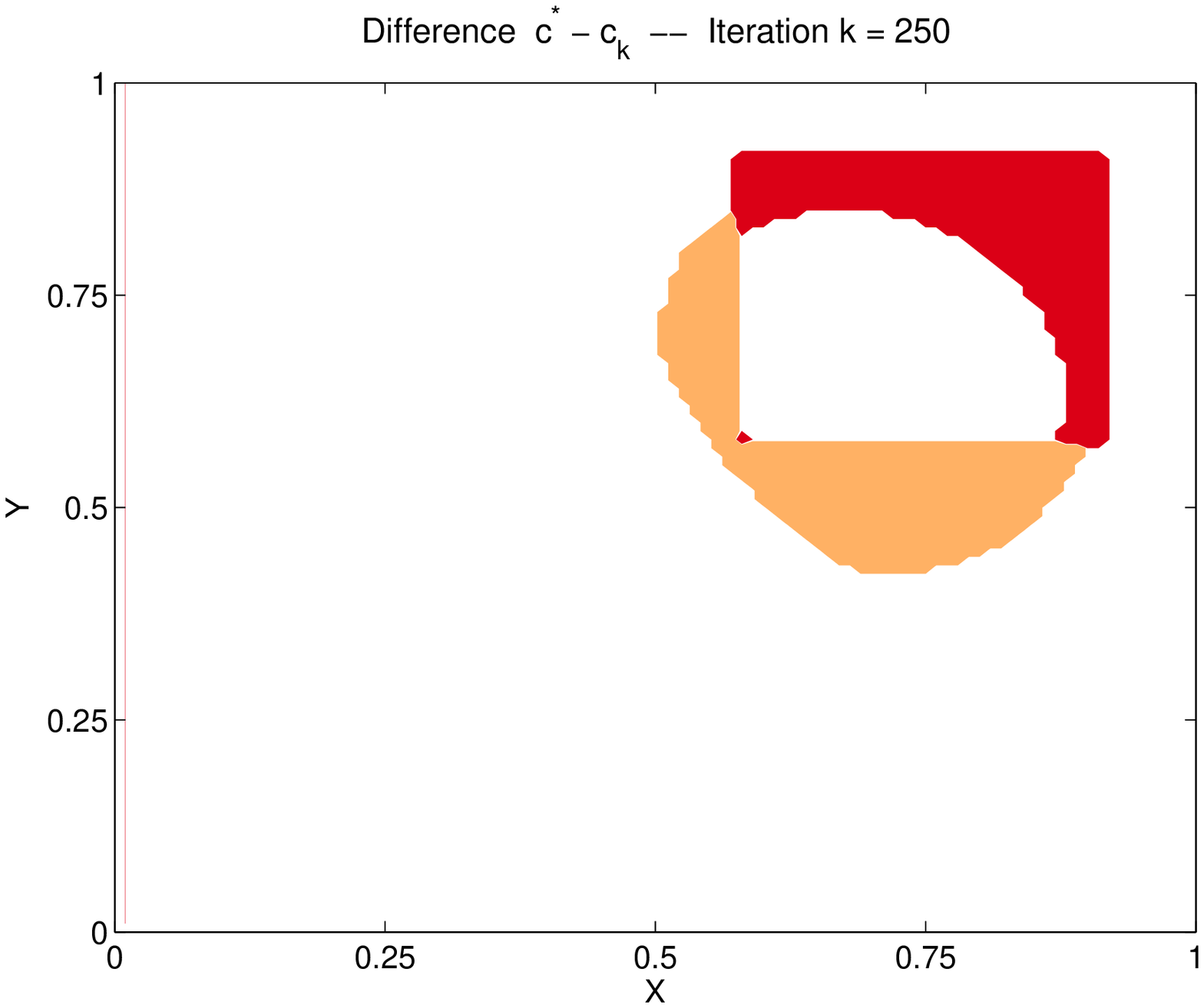}  
\includegraphics[width=4.2cm]{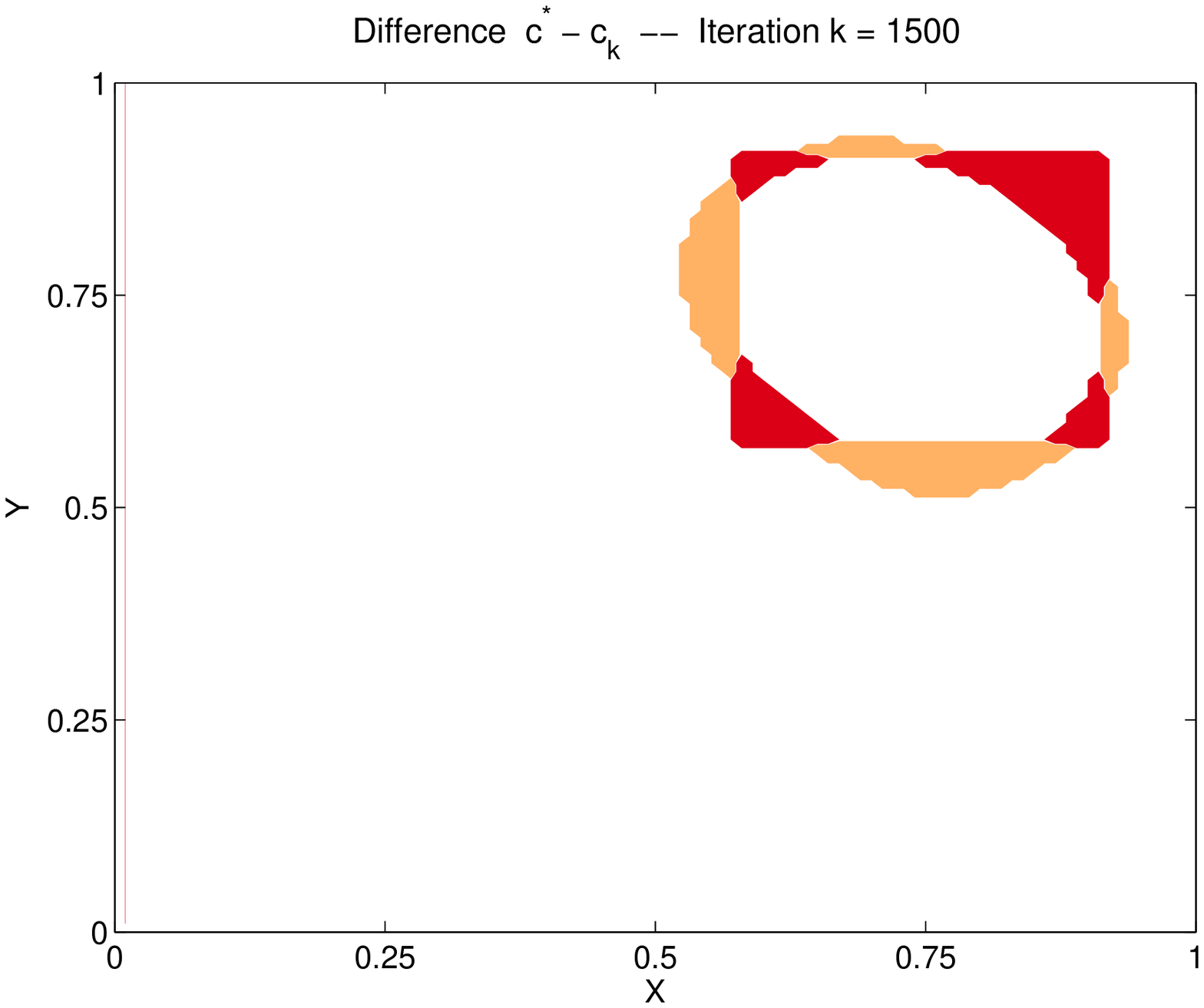}}  
\centerline{\hfil (e) \hskip4cm (f) \hskip4cm (g) \hskip4cm (h) \hfil}
\caption{\small Section~\ref{ssec:numer-ac}, 2nd example.
{\bf (a)}--{\bf (d)} Iterative reconstruction of $a^*$.
{\bf (a)} Evolution of the $L^2$ error.
{\bf (b)} Difference $a_0 - a^*$.
{\bf (c)} Difference $a_{750} - a^*$.
{\bf (d)} Difference $a_{1500} - a^*$.
{\bf (e)}--{\bf (h)} Iterative reconstruction of $c^*$.
{\bf (e)} Evolution of the $L^2$ error.
{\bf (f)} Difference $c_0 - c^*$.
{\bf (g)} Difference $c_{250} - c^*$.
{\bf (h)} Difference $c_{1500} - c^*$. }
\label{fig:example-ac2}
\end{figure}

Although the distance between supp$(a^*)$ and supp$(c^*)$ in example~2 is smaller
than in example~1 above, the 3-stage iteration behaves similarly in both examples.
The 1st-stage is ended after $k_1 = 250$ iterations, when the error $\|c_k - c^*\|$ has decreased considerably (see Figure~\ref{fig:example-ac2}~(e)).
The 2nd-stage corresponds to $k_1 \leq k \leq k_2 = 750$; at this point the difference between the exact solution $a^*$ and the iteration $a_{750}$ has
visibly decreased (see Figure~\ref{fig:example-ac2}~(c)). The 3rd-stage of the iteration corresponds to $k \geq k_2$. 
In this final stage, each iteration step consists in two iterations with respect to the diffusion coefficient $a_k$ and
one iteration with respect to the absorption coefficient $c_k$. The final results can be observed in Figure~\ref{fig:example-ac2}~(d) and Figure~\ref{fig:example-ac2}~(h) respectively.

The third and last example reveled itself as the most difficult identification
problem among all three considered in this section. The solution pair $(a^*,c^*)$
is chosen such that the supports of $a^*$ and $c^*$ intersect (see Figure~%
\ref{fig:example-ac}~(c)). We start the iteration once again keeping $a_k$
constant during the first stage. This part of the method is successful, 
since after $k_1=500$ iterations $c_{k_1}$ delivers a good approximation
for the exact solution $c^*$ (Figure~\ref{fig:example-ac3}~(g)). After that, we
start iterating with respect to $a_k$. After $k_2 = 750$ iterations we observe that
the error$\|a_k - a^*\|$ has decreased considerably (Figure~\ref{fig:example-ac3}~(a)). Finally, we start with 3rd-stage of the algorithm
and this is the point where the difficulties arise. No matter how many iterations we compute with respect to $c_k$, the approximation does
not get better than the one plotted in Figure~\ref{fig:example-ac3}~(g), which is
computed after $k_1 = 250$ steps.
After $1500$ steps, no significant improvement can be observed in the reconstruction of the absorption coefficient
(compare Figure~\ref{fig:example-ac3}~(g) and (h)). In this last example, the reconstruction of the diffusion
coefficient is very precise, but the approximation obtained for
the absorption coefficient is not so good.

It is worth noticing that the poor reconstruction of $c^*$ is not due to non-stable
behavior of our 3-stage algorithm. The particular exact solution $(a^*,c^*)$ in
this example (with intersecting supports) leads to a very hard identification problem 
already reported in \cite{Xu02, KALVK99, ArridgeSch99}.

It is worth mentioning that all problems presented in this Section were solved using
the standard level set method described in Section~\ref{sec:regularization}, i.e.,
updating both $(a_k,c_k)$ in every iterative step (and neglecting the 3-stage strategy).
The final results of these iterations were basically the same as the ones presented here. However,
the computational effort involved in the computation was by far much larger.

\begin{figure}[hb!]
\centerline{
\includegraphics[width=4.3cm]{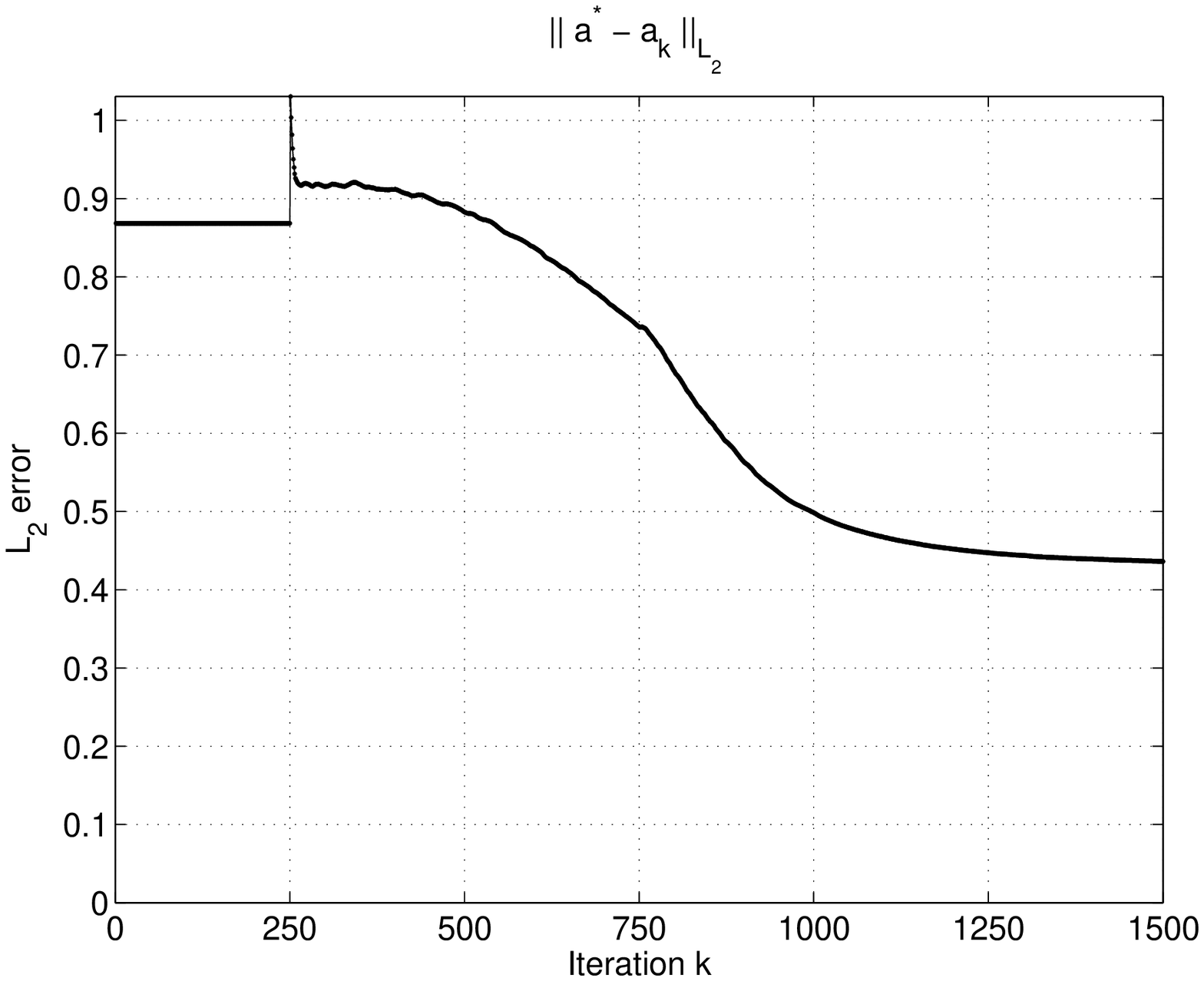} 
\includegraphics[width=4.2cm]{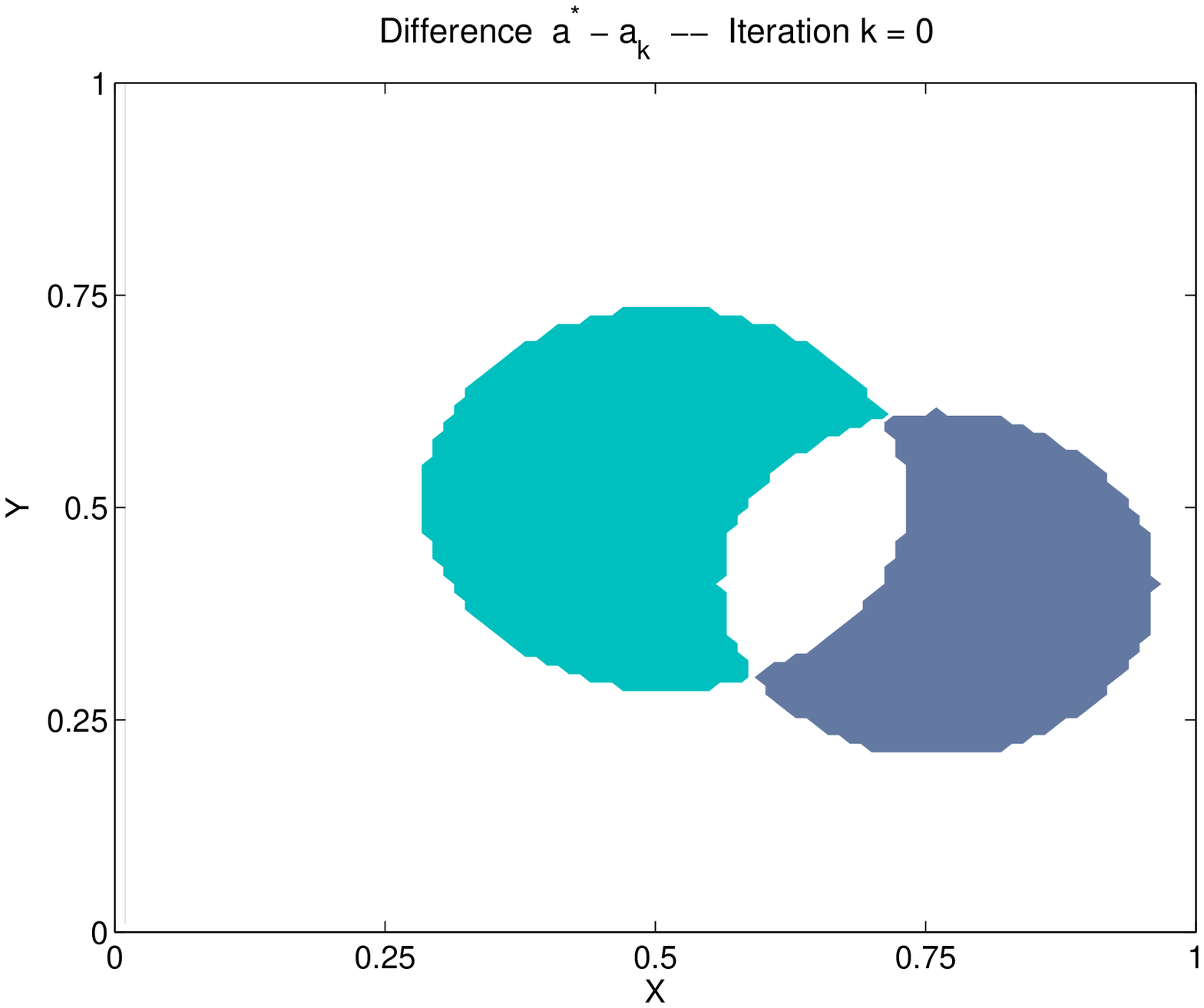} 
\includegraphics[width=4.2cm]{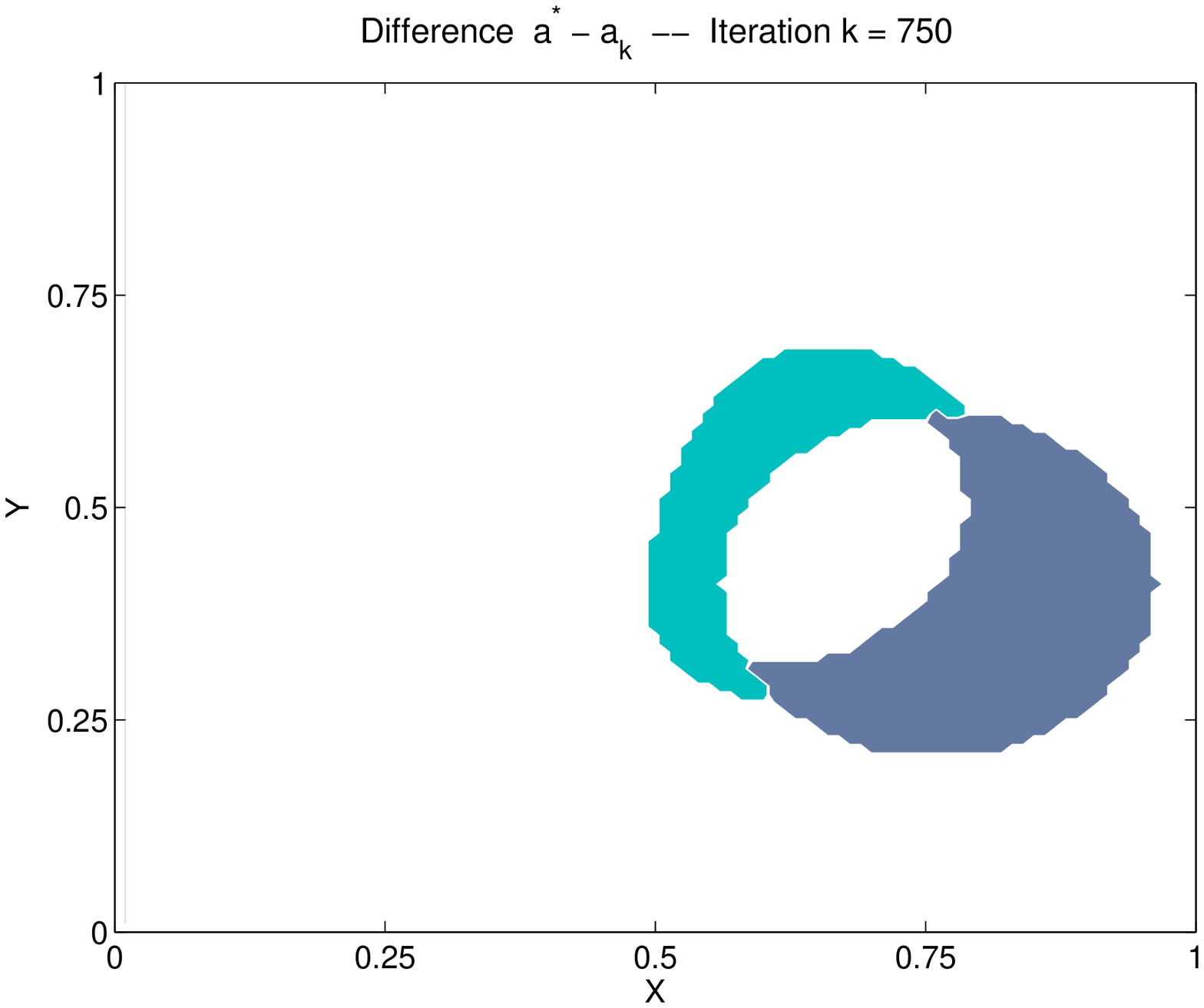} 
\includegraphics[width=4.2cm]{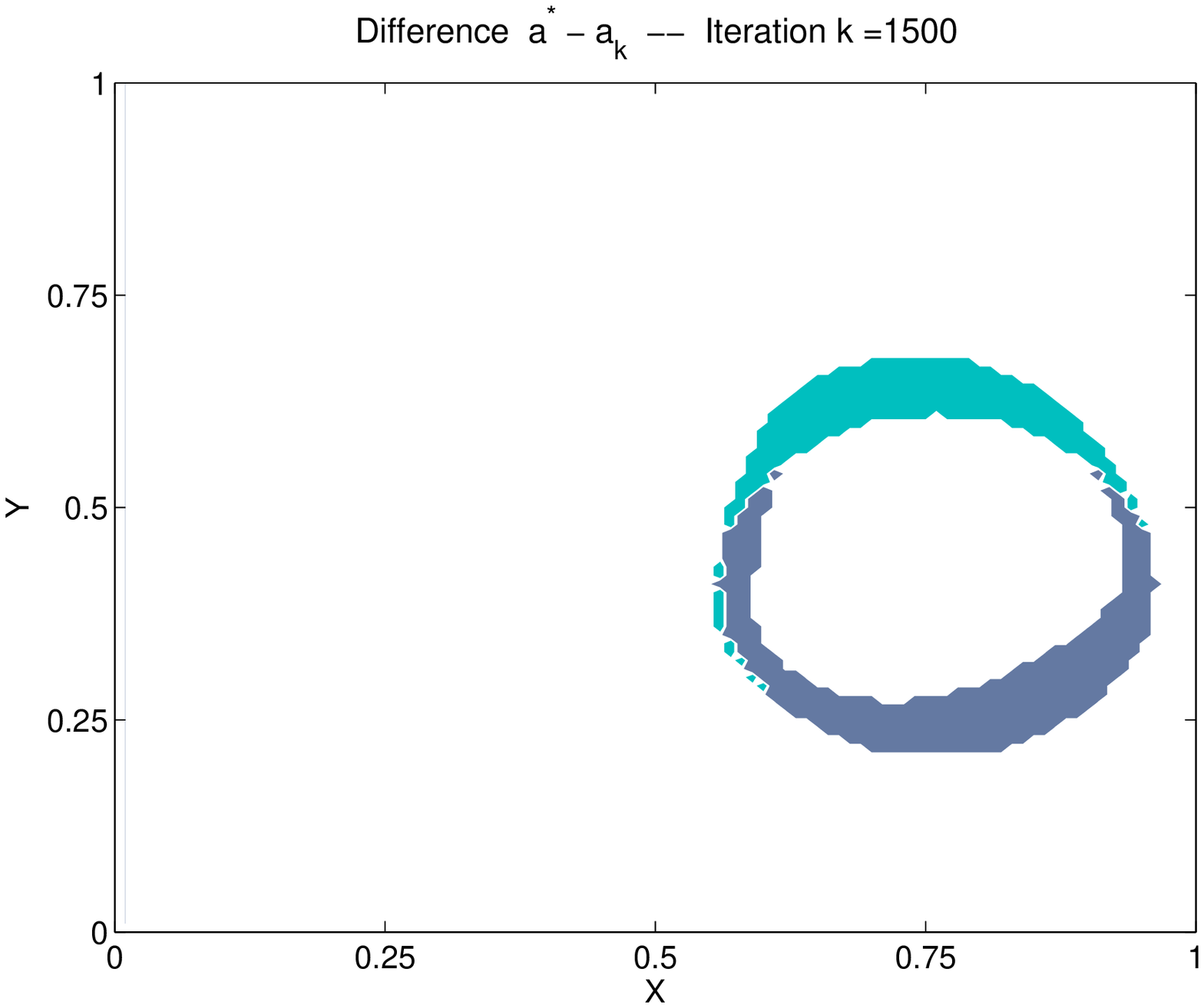}} 
\centerline{\hfil (a) \hskip4cm (b) \hskip4cm (c) \hskip4cm (d) \hfil}
\centerline{
\includegraphics[width=4.3cm]{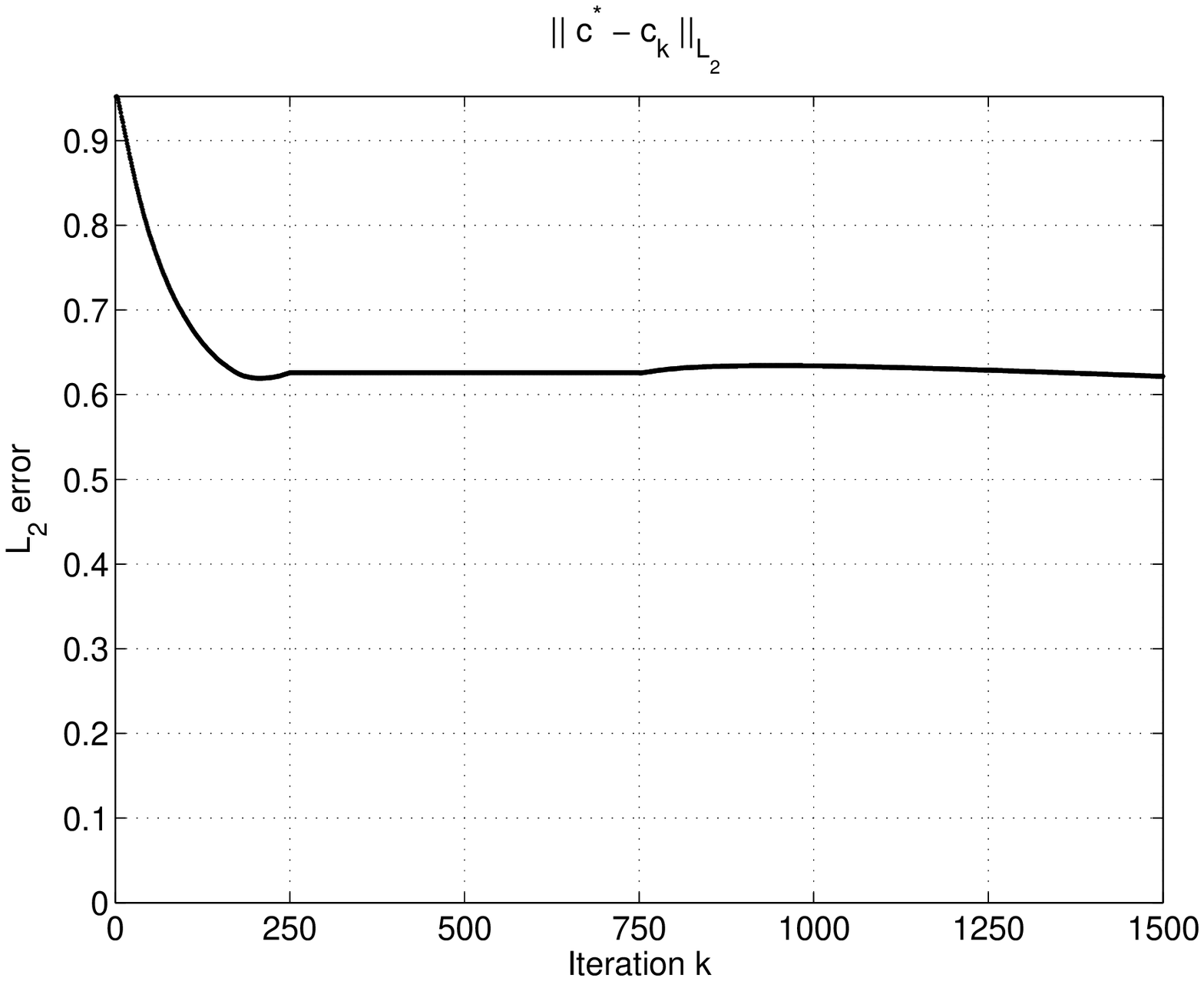}  
\includegraphics[width=4.2cm]{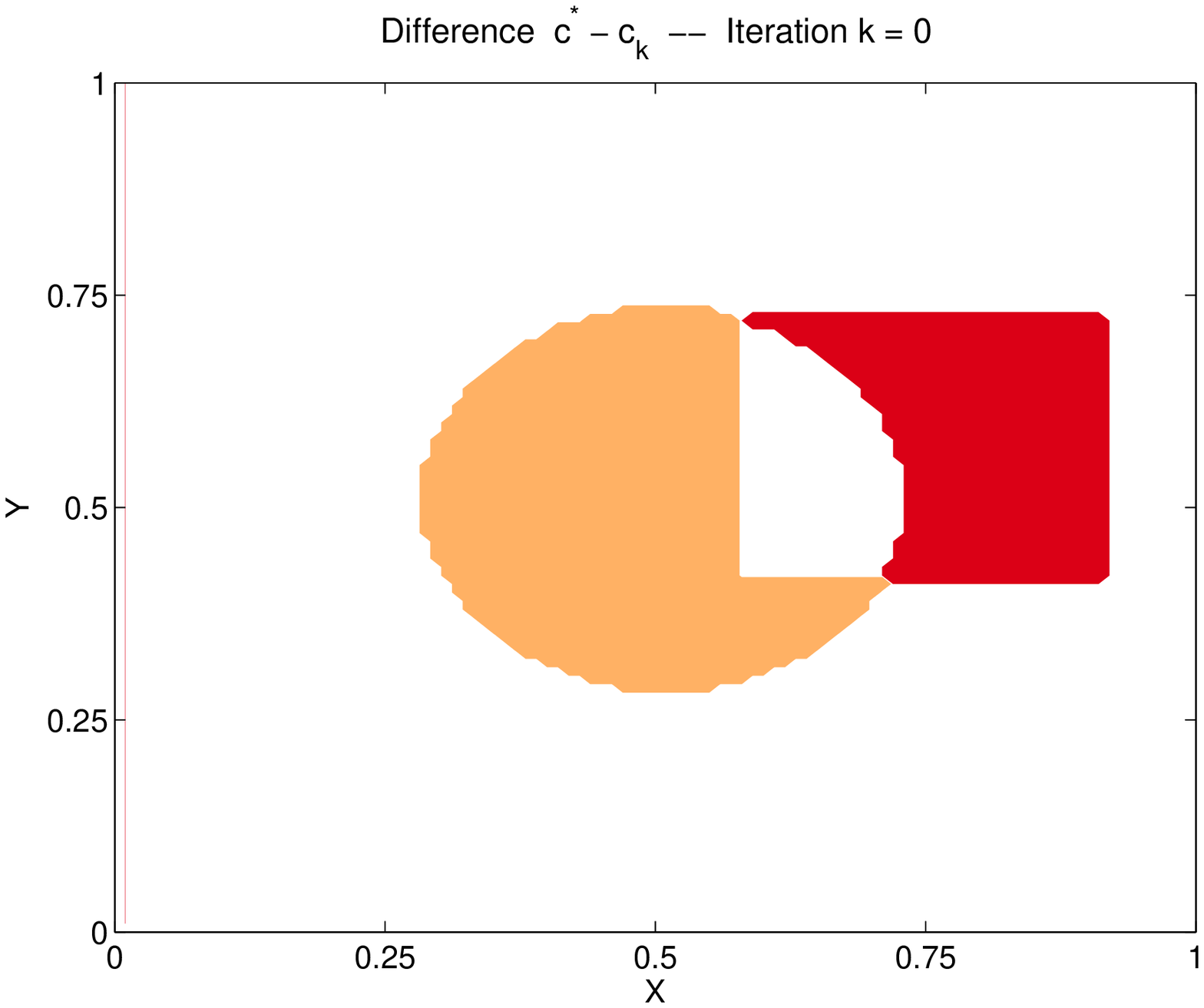} 
\includegraphics[width=4.2cm]{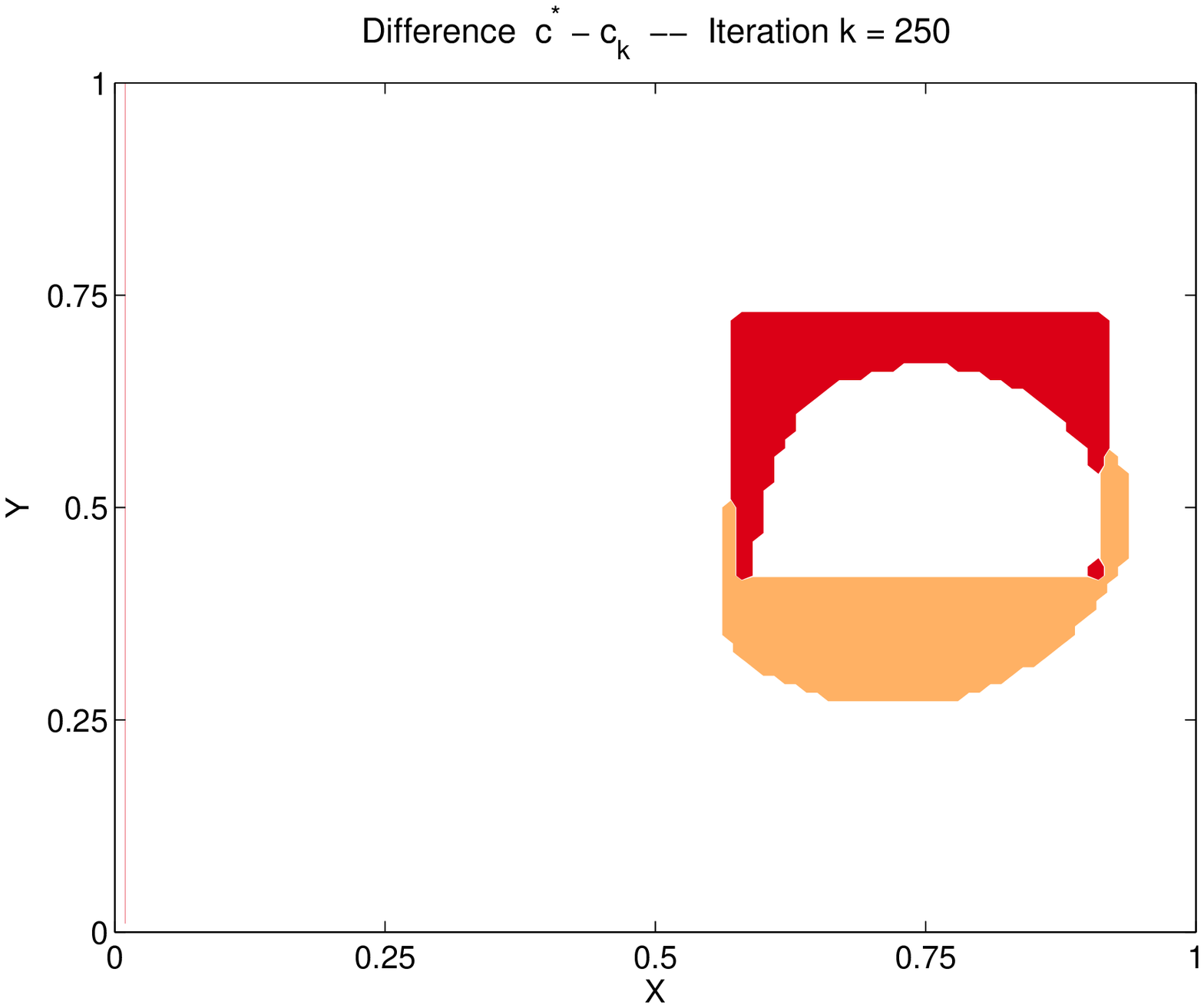} 
\includegraphics[width=4.2cm]{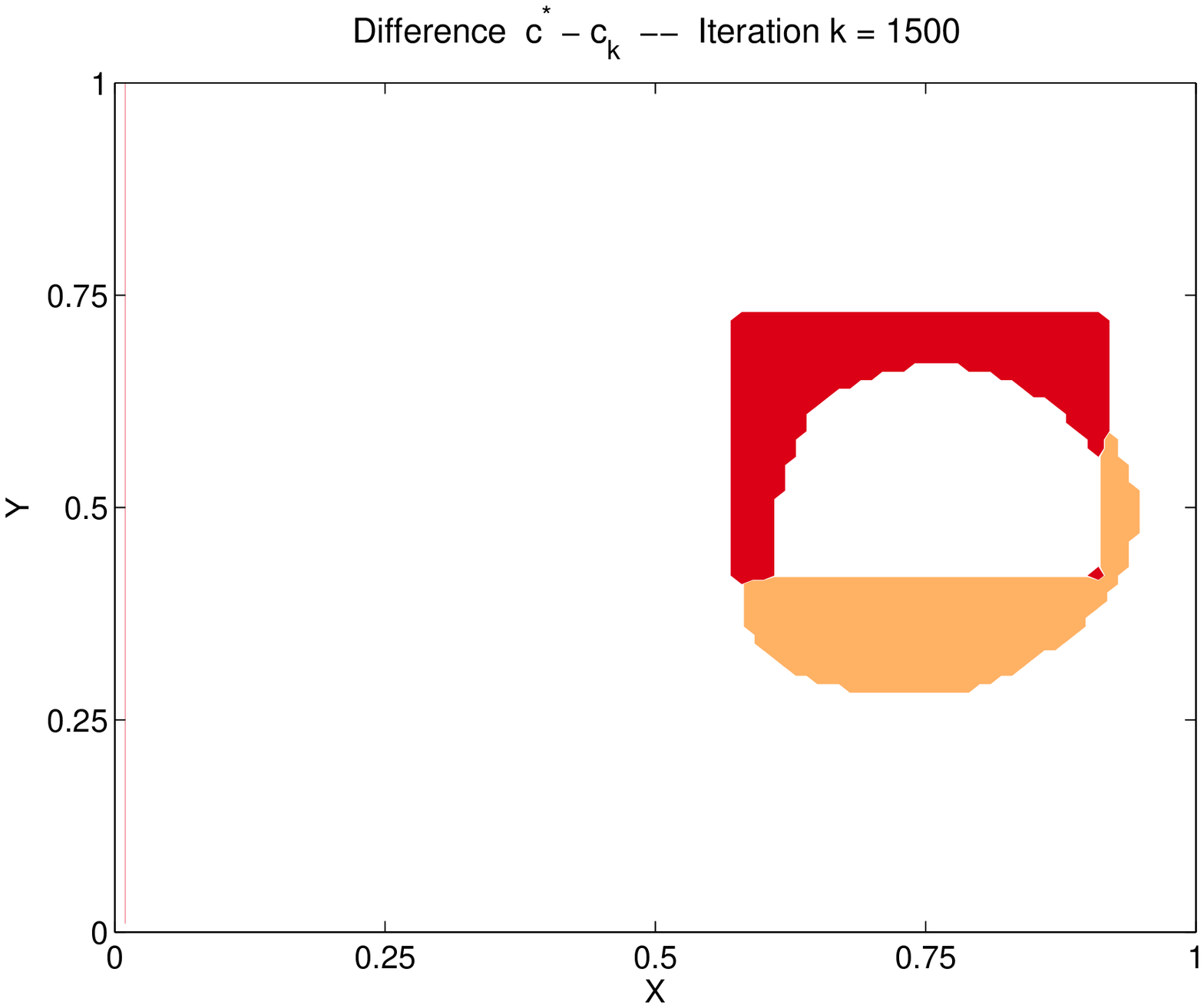}} 
\centerline{\hfil (e) \hskip4cm (f) \hskip4cm (g) \hskip4cm (h) \hfil}
\caption{\small Section~\ref{ssec:numer-ac}, 3rd example.
{\bf (a)}--{\bf (d)} Iterative reconstruction of $a^*$.
{\bf (a)} Evolution of the $L^2$ error.
{\bf (b)} Difference $a_0 - a^*$.
{\bf (c)} Difference $a_{750} - a^*$.
{\bf (d)} Difference $a_{1500} - a^*$.
{\bf (e)}--{\bf (h)} Iterative reconstruction of $c^*$.
{\bf (a)} Evolution of the $L^2$ error.
{\bf (f)} Difference $c_0 - c^*$.
{\bf (g)} Difference $c_{250} - c^*$.
{\bf (h)} Difference $c_{1500} - c^*$. }
\label{fig:example-ac3}
\end{figure}

\section{Conclusions}\label{sec:Conclusions}

In this paper, we develop a level set regularization approach for
simultaneous reconstruction of the piecewise constant coefficients
$(a,c)$ from a finite set of boundary measurements of optical tomography
in the diffusive regime. From the theoretical point of view, we prove that the forward map $F$ is
continuous in the $[L^1(\Omega)]^2$ topology. Hence, following standard arguments presented by the authors in previous papers
(see \cite{CLT09}) we get that the proposed level set strategy is a regularization method.
The main result behind the continuity of $F$ is a generalization of Meyers' Theorem for our particular case.

On the other hand, we propose a numerical algorithm to reconstruct simultaneously the diffusion and absorption coefficients.
Both coefficients are computed by minimizing a regularized energy functional.
Motivated by the fact that the reconstruction of the absorption coefficient $c$ is a mildly ill-posed inverse problem whereas
the reconstruction of the diffusion coefficient $a$ is exponentially ill-posed, we present a split strategy that consists in freezing $a=a_0$
and first iterate with respect to $c$ until the iteration stagnate. Then, keep $c=c_k$ and start to
iterate with respect to $a$ until stagnation of the iteration. Finally, iterate both coefficient. 
This numerical strategy has not only demonstrated that gives very good results but also reduces significantly the computational effort.

The situation of non-convergence of the level set algorithm, that is when coefficients $(a,c)$
have a crossing section (as in Subsection~\ref{ssec:numer-ac}) is not an easy problem and it has
already been reported in \cite{Xu02, KALVK99, ArridgeSch99}. We conjecture that the level set
algorithm will improve its performance if enough pairs of Neumann-to-Dirichlet data are available.
Since the situation with many measurements is numerically demanding, a strategy like the one proposed in  \cite{RKDA14}
could be more appropriated. We let this problem for future and careful investigation.


\section*{Acknowledgments}

The work of J.P.A. was partially 
supported by grants from CONICET and SECY-UNC.

A.D. acknowledges support from CNPq - 
Science Without Border grant
200815/2012-1, ARD-FAPERGS 
grant 0839 12-3 and CNPq grant 472154/2013-3.

The work of M.M.A. was partially 
supported by CNPq grants no.
406250/2013-8, 237068/2013-3 and
306317/2014-1.

The authors would like to thanks 
Prof. Dr. Uri M. Ascher for the all 
discussions and valuable suggestions.

\appendix 
\section{Proof of Theorem~\ref{th:Meyers}}
\label{sec:Appendix}

The main purpose of this 
appendix is to show that under mild assumptions on the
boundary (Neumann) data $g$ 
the solution $u$ of \eqref{eq:1}--\eqref{eq:2}
belongs to $W^{1,p}(\Omega)$ for some $p > 2$ 
(therefore better than the
standard regularity $u\in H^1(\Omega)$).

As far as we know, this type of regularity, 
namely $u\in W^{1,p}(\Omega)$ for $p>2$,
goes back to the pioneering work of Meyers~\cite{M63}, 
for elliptic BVPs
with Dirichlet boundary conditions. 
Later on, Gallouet and Monier~\cite{GM99}
generalized Meyers' result to Neumann BVPs. 
However, for the best
of the authors knowledge there is no proof of 
such a result for the
problem \eqref{eq:1}--\eqref{eq:2}.

The following proof was suggested by one of the anonymous referees. The authors are grateful to him
for this suggestion.

\paragraph{Proof of Theorem~\ref{th:Meyers}}
Let $u\in H^1(\Omega)$ be the unique solution
of \eqref{eq:1}--\eqref{eq:2}. It clearly satisfies the
weak formulation
\begin{align}
\label{eq:su70}
 \int_{\Omega}a\nabla u\cdot\nabla \varphi dx
+\int_{\Omega} cu\varphi dx=
\int_{\Gamma}g\varphi d\sigma\qquad \forall \varphi\in H^1(\Omega).
\end{align} 
Define now 
$\tilde u:=u-\dfrac{1}{|\Omega|}\int_{\Omega}u dx$,
which in particular satisfies the weak formulation
(5) in \cite{GM99}, i.e.,
%
\begin{displaymath}
\left\{
\begin{aligned}
& \tilde u\in H^1_*(\Omega),\\
& \int_{\Omega} a\nabla \tilde u\cdot\nabla \varphi dx=
\inner{f}{\varphi}_{(H^1)', H^1}\qquad \forall 
\varphi\in H^1(\Omega), 
\end{aligned}
\right.
\end{displaymath}
where $f$ is defined as
\[
 \inner{f}{\varphi}_{(H^1)',H^1}:=
\int_{\Gamma}g\varphi d\sigma-
\int_{\Omega}cu\varphi dx\qquad \forall \varphi\in H^1(\Omega).
\]
Note that $f$ naturally satisfies
$\inner{f}{1}_{(H^1)',H^1}=0$.
Hence, to finish the proof we only
need to apply the regularity result given
in \cite[Theorem 2]{GM99}.
To this end, it remains to show that the distribution
$f$ is in $W^{1,q}(\Omega)'$. 
As $g$ belongs to $W^{1-1/q,q}(\Gamma)'$
the distribution 
$\varphi\mapsto \int_{\Gamma}g\varphi d\sigma$
is in $W^{1,q}(\Omega)'$ thanks to the trace theorem
in Sobolev spaces. 
Next we shall prove that the
distribution
$h:\varphi\mapsto \int_{\Omega}\,c u\varphi \, dx$ 
also 
belongs to $W^{1,q}(\Omega)'$.
Consider first the case $N=2$. 
In this case, since $q<2$, 
we have the continuous 
embedding~\cite[Corollary 9.14]{bre-fa.book}
$W^{1,q}(\Omega)\hookrightarrow L^{q^*}(\Omega)$, where $q^*=2q/(2-q)>2$.
Letting $s:=q^*/(q^*-1)$ be the conjugate of $q^*$ we have $s<2$ (because $q^*>2$)
and, as a consequence, the continuous 
embedding~\cite{bre-fa.book} $H^1(\Omega)\hookrightarrow L^2(\Omega)\hookrightarrow L^s(\Omega)$.
Using the latter inclusions, 
the fact that $u\in H^1(\Omega)$, 
the second inequality in \eqref{eq:def.df7}
and the H\"{o}lder's inequality, we obtain 
(for all $\varphi \in W^{1,q}(\Omega)$):
\begin{align}
\nonumber
\left|\inner{h}{\varphi}\right|=\left|\int_{\Omega}\,c(x)u\varphi \, dx\right|&\leq
\overline{c}\|u\|_{L^s}\|\varphi\|_{L^{q^*}}\\
\label{eq:201b}
                                           &\leq \overline{c}\|u\|_{L^s}\|\varphi\|_{W^{1,q}},
\end{align}
which proves that $h\in W^{1,q}(\Omega)'$. 
Consider now the case $N\in \{3,4\}$ and let
$q^*:=qN/(N-q)>1$ and (as before) $s:=q^*/(q^*-1)$ 
its conjugate.  In this case, we have also
the continuous 
embeddings~\cite[Corollary 9.14]{bre-fa.book} 
$W^{1,q}(\Omega)\hookrightarrow L^{q^*}(\Omega)$
and, since $1\leq s\leq 2^*:=2N/(N-2)$, 
$H^1(\Omega)\hookrightarrow L^s(\Omega)$.
Using the same reasoning as in the case $N=2$ 
we find that~\eqref{eq:201b} also holds
when $N\in \{3,4\}$ which concludes
the proof of the desired regularity to the distribution $h$.
Altogether, we obtain that 
$f$ is well-defined and belongs to $W^{1,q}(\Omega)'$.


\bibliographystyle{plain}
\bibliography{optical-tomo}

\end{document}